\documentclass[twoside,11pt]{article}

\usepackage{jmlr2e}

\usepackage[utf8]{inputenc}
\usepackage[english]{babel}

\usepackage{setspace}
\usepackage[utf8]{inputenc} \usepackage[T1]{fontenc}    

\usepackage{amsmath}
\allowdisplaybreaks
\usepackage{amsfonts}       \usepackage{url}            \usepackage{blindtext}
\usepackage{booktabs}       \usepackage{enumitem}
\usepackage{graphicx}
\graphicspath{{images/}{../images/}}
\usepackage{hyperref}       \hypersetup{
    colorlinks, linkcolor={black},
    citecolor={blue!80!black},
    urlcolor={blue!80!black}    
}
\usepackage{nicefrac}       \usepackage{microtype}      \microtypesetup{expansion=false}
\usepackage{xcolor}         \usepackage{subcaption}
\usepackage{tikz}
\usepackage{natbib}
\usepackage[export]{adjustbox}
\usepackage{xcolor}
\usepackage{cleveref}

\newtheorem{result}[theorem]{Result}

\newcommand{\N}{\mathbb{N}}

\newcommand{\R}{\mathbb{R}}

\newcommand{\abs}[1]{\left|#1\right|}
\newcommand{\norm}[1]{\left\|#1\right\|}

\newcommand{\commentout}[1]{}

\newcommand{\realmatrix}

\usepackage{dsfont}

\newcommand{\cardinality}[1]{ | #1 | }
\renewcommand{\d}[1]{\ensuremath{\operatorname{d}\!{#1}}}
\newcommand{\e}[1]{ {\mathrm{e}}^{ #1 } }

\newcommand{\expectation}[1]{ \mathbb{E} [ #1 ] }
\newcommand{\expectationWrt}[2]{ \mathbb{E}_{#2} [ #1 ] }
\newcommand{\expectationBig}[1]{ \mathbb{E} \Bigl[ #1 \Bigr] }
\newcommand{\expectationBigWrt}[2]{ \mathbb{E}_{#2} \Bigl[ #1 \Bigr] }

\newcommand{\indicator}[1]{ \mathds{1} [ #1 ] }

\newcommand{\process}[2]{ \{ #1 \}_{ #2 } }

\newcommand{\bigO}[1]{ O(#1) }

\newcommand{\pnorm}[2]{ \| #1 \|{}_{#2} }

\newcommand{\probability}[1]{ \mathbb{P} [ #1 ] }

\newcommand{\transpose}[1]{ #1{}^{\mathrm{T}} }

\newcommand{\naturalNumbersPlus}{ \mathbb{N}_{+} }
\newcommand{\naturalNumbersZero}{ \mathbb{N}_{0} }
\newcommand{\realNumbers}{ \mathbb{R} }

\newcommand{\criticalpoint}[1]{  #1^{\textnormal{opt}} }
\newcommand{\iterand}[2]{ #1^{[#2]} }
\newcommand{\iterandd}[2]{ #1^{\{#2\}} }

\newcommand{\refEquation}[1]{{\textrm{\eqref{#1}}}}

\newcommand{\refTheorem}[1]{{\textrm{Theorem~\ref{#1}}}}
\newcommand{\refCorollary}[1]{{\textrm{Corollary~\ref{#1}}}}
\newcommand{\refProposition}[1]{{\textrm{Proposition~\ref{#1}}}}

\newcommand{\refAppendixSection}[1]{\textrm{Appendix~\ref{#1}}}

\def\eqcom#1{\overset{\textnormal{(#1)}}}

\newcommand{\itr}[2]{ \iterand{#1}{#2} }
\newcommand{\itrd}[2]{ \iterandd{#1}{#2} }

\def\E{{\mathbb E}}

\def\({{\Bigl(}}
\def\){{\Bigr)}}

\newcommand{\ba}{\begin{array}}
\newcommand{\ea}{\end{array}}

\newcommand{\eps}{\varepsilon}

\newcommand{\xdeleted}[1]{\deleted{}}

\def\putin#1{{\color{black}#1}}

\def\albert#1{\textcolor{black}{#1}}
\def\takeout#1{}

\usepackage[acronym]{glossaries}
\makeglossaries

\newacronym{GD}{GD}{Gradient Descent}
\newacronym{NN}{NN}{Neural Network} 
\newacronym{ODE}{ODE}{Ordinary Differential Equation}
\newacronym{PL}{PL}{Polyak--\L{}ojasiewicz}
\newacronym{ReLU}{ReLU}{Rectified Linear Unit}
\newacronym{SGD}{SGD}{Stochastic Gradient Descent}
\newacronym{MSE}{MSE}{Mean Square Error}
\newacronym{CIFAR}{CIFAR}{Canadian Institute For Advanced Research}
\newacronym{MNIST}{MNIST}{Modified National Institute of Standards and Technology}

\jmlrheading{1}{2000}{1-48}{4/00}{10/00}{XXXXX}{Albert Senen--Cerda and Jaron Sanders}

\ShortHeadings{Almost Sure Convergence of Dropout Algorithms for Neural Networks}{Senen--Cerda and Sanders}
\firstpageno{1}

\title{Almost Sure Convergence of Dropout Algorithms for Neural Networks}

\author{\name Albert Senen--Cerda \email a.senen.cerda@tue.nl\\
  \name Jaron Sanders  \email jaron.sanders@tue.nl \\
  \addr Department of Mathematics \& Computer Science\\ 
  Eindhoven University of Technology\\
  Eindhoven, The Netherlands
}

\editor{~}

\begin{document}

\maketitle

\begin{abstract}
  We investigate the convergence and convergence rate of stochastic training algorithms for \glspl{NN} that have been inspired by \emph{Dropout} \citep{hinton2012improving}.
  With the goal of avoiding overfitting during training in \glspl{NN}, dropout algorithms consist in practice of multiplying the weight matrices of a \gls{NN} componentwise by independently drawn random matrices with $\{0, 1 \}$-valued entries during each iteration of \gls{SGD}.
  This paper presents a probability theoretical proof that for fully-connected \glspl{NN} with differentiable, polynomially bounded activation functions, if we project the weights onto a compact set when using a dropout algorithm, then the weights of the \gls{NN} converge to a unique stationary point of a projected system of \glspl{ODE}.

  After this general convergence guarantee, we go on to investigate the convergence rate of dropout.
  Firstly, we obtain generic sample complexity bounds for finding $\epsilon$-stationary points of smooth nonconvex functions using \gls{SGD} with dropout that explicitly depend on the dropout probability.
  Secondly, we obtain an upper bound on the rate of convergence of \gls{GD} on the limiting \glspl{ODE} of dropout algorithms for \glspl{NN} with the shape of an arborescence of arbitrary depth and with linear activation functions.
  The latter bound shows that for an algorithm such as \emph{Dropout} or \emph{Dropconnect} \citep{wan2013regularization}, the convergence rate can be impaired exponentially by the depth of the arborescence.

  In contrast, we experimentally observe no such dependence for wide \glspl{NN} with just a few dropout layers.
  We also provide a heuristic argument for this observation.
  Our results suggest that there is a change of scale of the effect of the dropout probability in the convergence rate that depends on the relative size of the width of the \gls{NN} compared to its depth.
\end{abstract}


\section{Introduction}
\label{sec:Introduction}

\emph{Dropout} \citep{hinton2012improving} is a technique to avoid overfitting during training of \glspl{NN} that consists of temporarily `dropping' nodes of the network independently at each step of \gls{SGD}.
While in the original \emph{Dropout} algorithm in \cite{hinton2012improving} only nodes from the network were dropped, several stochastic training algorithms that avoid overfitting in \glspl{NN} have appeared since then; for example, \emph{Dropconnect} \citep{wan2013regularization}, \emph{Cutout} \citep{devries2017improved}.
\Cref{fig:Dropout_on_a_neural_network} depicts a \gls{NN} where we use \emph{Dropconnect} and drop individual edges instead of nodes.
In practice, such dropout algorithms consist of multiplying componentwise weight matrices of the \gls{NN} in each iteration by independently drawn random matrices with $\{ 0, 1 \}$-valued entries.
The elements of these random matrices indicate whether each individual edge or node is {filtered} ($0$) or is not filtered ($1$) during a training step.
The resulting weight matrices are then used in the backpropagation algorithm for computing the gradient of a \gls{NN}.
Mathematically, dropout turns the backpropagation algorithm into a step of a \gls{SGD} in which the primary source of randomness is the \gls{NN}'s configuration.
Under mild independence assumptions, the loss function of dropout is a risk function averaged over all possible \glspl{NN} configurations \citep{baldi2013understanding}.

\begin{figure}[!tbhp]
    \centering
    \begin{subfigure}{0.32\textwidth}
        \centering
        \def\layersep{2.0em}
        \begin{tikzpicture}[shorten >=1pt,->,draw=black!50, node distance=\layersep]
        \tikzstyle{every pin edge}=[<-,shorten <=1pt]
        \tikzstyle{neuron}=[circle,draw=black,fill=white!100,minimum size=10pt,inner sep=0pt]
        \tikzstyle{input neuron}=[neuron];
        \tikzstyle{output neuron}=[neuron];
        \tikzstyle{hidden neuron}=[neuron];
        \tikzstyle{annot} = [text width=4em, text centered]

\foreach \name / \y in {1,...,4}
\node[input neuron, pin=left:$x_{\y}$] (I-\name) at (0,-0.75*\y) {};

\foreach \name / \y in {1,...,5}
            \path[yshift=+0.375cm]
                node[hidden neuron] (Ha-\name) at (\layersep,-0.75*\y cm) {};

\foreach \name / \y in {1,...,4}
            \path[yshift=-0cm]
                node[hidden neuron] (Hb-\name) at (2*\layersep,-0.75*\y cm) {};
\foreach \name / \y in {1,...,3}
\path[yshift=-0.375cm]
                node[output neuron,pin={[pin edge={->}]right:$y_{\y}$}] (O-\name) at (3*\layersep, -0.75*\y cm) {};

\pgfmathsetseed{1987116}

\foreach \source in {1,...,4}
            \foreach \dest in {1,...,5}
                {
                \pgfmathparse{int(1000*rand+1000)} \ifnum \pgfmathresult>0
                    \path (I-\source) edge[black, opacity=1] (Ha-\dest); 
                \else
                    \path (I-\source) edge[black, opacity=0] (Ha-\dest);         	
                \fi
                }        

\foreach \source in {1,...,5}
            \foreach \dest in {1,...,4}
                {
                \pgfmathparse{int(1000*rand+1000)} \ifnum \pgfmathresult>0
                    \path (Ha-\source) edge[black, opacity=1] (Hb-\dest); 
                \else
                    \path (Ha-\source) edge[black, opacity=0] (Hb-\dest);         	
                \fi
                }        

        \foreach \source in {1,...,4}
            \foreach \dest in {1,...,3}
                {
                \pgfmathparse{int(1000*rand+1000)} \ifnum \pgfmathresult>0
                    \path (Hb-\source) edge[black, opacity=1] (O-\dest); 
                \else
                    \path (Hb-\source) edge[black, opacity=0] (O-\dest);         	
                \fi
                }        
\end{tikzpicture}         \caption{Case $p=1$.
        }
    \end{subfigure}
    \begin{subfigure}{0.32\textwidth}
        \centering
        \def\layersep{2.0em}
        \begin{tikzpicture}[shorten >=1pt,->,draw=black!50, node distance=\layersep]
        \tikzstyle{every pin edge}=[<-,shorten <=1pt]
        \tikzstyle{neuron}=[circle,draw=black,fill=white!100,minimum size=10pt,inner sep=0pt]
        \tikzstyle{input neuron}=[neuron];
        \tikzstyle{output neuron}=[neuron];
        \tikzstyle{hidden neuron}=[neuron];
        \tikzstyle{annot} = [text width=4em, text centered]

\foreach \name / \y in {1,...,4}
\node[input neuron, pin=left:$x_{\y}$] (I-\name) at (0,-0.75*\y) {};

\foreach \name / \y in {1,...,5}
            \path[yshift=+0.375cm]
                node[hidden neuron] (Ha-\name) at (\layersep,-0.75*\y cm) {};

\foreach \name / \y in {1,...,4}
            \path[yshift=-0cm]
                node[hidden neuron] (Hb-\name) at (2*\layersep,-0.75*\y cm) {};
\foreach \name / \y in {1,...,3}
\path[yshift=-0.375cm]
                node[output neuron,pin={[pin edge={->}]right:$y_{\y}$}] (O-\name) at (3*\layersep, -0.75*\y cm) {};

\pgfmathsetseed{1987116}

\foreach \source in {1,...,4}
            \foreach \dest in {1,...,5}
                {
                \pgfmathparse{int(1000*rand+1000)} \ifnum \pgfmathresult>1500
                    \path (I-\source) edge[black, opacity=1] (Ha-\dest); 
                \else
                    \path (I-\source) edge[black, opacity=0] (Ha-\dest);         	
                \fi
                }        

\foreach \source in {1,...,5}
            \foreach \dest in {1,...,4}
                {
                \pgfmathparse{int(1000*rand+1000)} \ifnum \pgfmathresult>1500
                    \path (Ha-\source) edge[black, opacity=1] (Hb-\dest); 
                \else
                    \path (Ha-\source) edge[black, opacity=0] (Hb-\dest);         	
                \fi
                }        

        \foreach \source in {1,...,4}
            \foreach \dest in {1,...,3}
                {
                \pgfmathparse{int(1000*rand+1000)} \ifnum \pgfmathresult>1500
                    \path (Hb-\source) edge[black, opacity=1] (O-\dest); 
                \else
                    \path (Hb-\source) edge[black, opacity=0] (O-\dest);         	
                \fi
                }        
\end{tikzpicture}         \caption{Case $p=0.25$.
        }
    \end{subfigure}
    \begin{subfigure}{0.32\textwidth}
        \centering
        \def\layersep{2.0em}
        \begin{tikzpicture}[shorten >=1.0pt,->,draw=black!50, node distance=0.0em]
        \tikzstyle{every pin edge}=[<-,shorten <=1pt]
        \tikzstyle{neuron}=[circle,draw=black,fill=white!100,minimum size=10pt,inner sep=0pt]
        \tikzstyle{input neuron}=[neuron];
        \tikzstyle{output neuron}=[neuron];
        \tikzstyle{hidden neuron}=[neuron];
        \tikzstyle{annot} = [text width=4em, text centered]

\foreach \name / \y in {1,...,1}
\node[input neuron, pin=left:$x_{\y}$] (I-\name) at (0.0,-0.5*\y -1.5) {};

\foreach \name / \y in {1,...,3}
            \path[yshift=-1.0cm]
                node[hidden neuron] (Ha-\name) at (\layersep,-0.5*\y cm) {};

\foreach \name / \y in {1,...,4}
            \path[yshift=-0.75cm]
                node[hidden neuron] (Hb-\name) at (2*\layersep,-0.5*\y cm) {};
            
\foreach \name / \y in {1,...,7}
\path[yshift=-0.0cm]
                node[output neuron,pin={[pin edge={->}]right:$y_{\y}$}] (O-\name) at (3*\layersep, -0.5*\y cm) {};

\pgfmathsetseed{1987116}

\path (I-1) edge[black, opacity=1] (Ha-1);        
        \path (I-1) edge[black, opacity=1] (Ha-2);
        \path (I-1) edge[black, opacity=1] (Ha-3);
\path (Ha-1) edge[black, opacity=1] (Hb-1);        
        \path (Ha-1) edge[black, opacity=1] (Hb-2);
        \path (Ha-2) edge[black, opacity=1] (Hb-3);
        \path (Ha-3) edge[black, opacity=1] (Hb-4);

        \path (Hb-1) edge[black, opacity=1] (O-1);        
        \path (Hb-1) edge[black, opacity=1] (O-2);
        \path (Hb-2) edge[black, opacity=1] (O-3);
        \path (Hb-2) edge[black, opacity=1] (O-4);
        \path (Hb-3) edge[black, opacity=1] (O-5);
        \path (Hb-4) edge[black, opacity=1] (O-6);
        \path (Hb-4) edge[black, opacity=1] (O-7);  
\end{tikzpicture}         \caption{An arborescence.
        }
        \label{fig:Arborescence_picture}
    \end{subfigure}

    \caption{\emph{(a,b)} \emph{Dropconnect}'s training step \citep{wan2013regularization} in a \gls{NN} with $L=3$ layers.
        In this algorithm, on every iteration, a random \gls{NN} is first generated by removing each edge with probability $p \in (0,1]$ independently of all other edges.
        The output of this random \gls{NN} is then used to update all weights using the backpropagation algorithm.
        \emph{(c)}
        An example arborescence of depth $L = 3$.
    }
    \label{fig:Dropout_on_a_neural_network}
\end{figure}

An interesting aspect of dropout algorithms is that they lie at the intersection of stochastic optimization and \emph{percolation theory}, which investigates properties related to connectedness of random graphs and deterministic (possibly infinite) graphs in which vertices and edges are deleted at random.
In the case of dropout, the output of the filtered \gls{NN} with temporarily deleted edges is used to update the weights.
If dropout filters too many weights, then little information about the input can pass through the network, which will consequently also yield a gradient update for that step that contains little relevant information.

As an example, we may consider again the networks in Figures~\ref{fig:Dropout_on_a_neural_network} (a)--(b) when we use \emph{Dropconnect}, that is, we filter each edge with probability $1-p$ independently of all other edges.
We can observe that the number of paths $\chi$ in Figure~\ref{fig:Dropout_on_a_neural_network} (b) that fully transverse the network ($\chi=5$) is much smaller compared to those of Figure~\ref{fig:Dropout_on_a_neural_network} (a) ($\chi=240$).
In an $L$-layer \gls{NN} with no biases, a path from the input layer to the output goes through $L$ weights that have filters.
Then, the probability that a path from input to output stays unfiltered and contributes to a weight update is $p^{L}$.
If we now fix one edge in the path, then the probability of updating its corresponding weight through that path in particular is also $p^{L}$.
There are, however, many other paths in a \gls{NN} passing through a single edge.
The probability that one of those paths is not filtered will be large and may compensate the exponential factor $p^L$.
Considering the connection to bond percolation, one may therefore suspect that dropout algorithms may perform worse than a routine implementation of the backpropagation algorithm.
However, dropout algorithms usually perform well since they avoid overfitting in \glspl{NN} \citep{hinton2012improving,srivastava2014dropout}.
From the point of view of bond percolation however, this should still come at the cost of slower convergence of dropout algorithms, and conceivably by as much as a factor $p^L$, where $L$ is the number of dropout layers.

Most theoretical focus has been on the generalization properties of \glspl{NN} trained with dropout algorithms.
We can mention \cite{hinton2012improving,baldi2013understanding,wager2013dropout,srivastava2014dropout,baldi2014dropout,cavazza2017dropout,mianjy2018implicit,mianjy2019dropout,pal2020regularization,wei2020implicit}, which we briefly review in \Cref{sec:Related_work}.
In this paper, however, we investigate dropout from the stochastic optimization perspective.
That is, we aim to answer if dropout algorithms converge and study the rate at which they converge, which is expected to depend on the dropout probability.
Compared to the study of the generalization properties of dropout, this aim has received less attention in the literature.
In particular, we can only mention \cite{Mianjy2020OnCA} and \cite{senencerda2020asymptotic}.
In \cite{Mianjy2020OnCA}, a convergence rate for the test error in a classification setting is obtained when training shallow \glspl{NN} with dropout.
This rate, is, however, independent of the dropout probability.
In \cite{senencerda2020asymptotic}, a convergence rate for the empirical risk associated with training shallow linear \glspl{NN} with dropout is obtained that depends on the dropout probability.
Both results refer to shallow \glspl{NN} where the width of the \gls{NN} plays a role in the convergence rate.
We refer to \Cref{sec:Related_work} below for further details.

From the previous discussion, however, we suspect that there is an effect of dropout in the convergence rate in \emph{deep} \glspl{NN} with several layers of dropout.
In this paper, we investigate this problem.
In particular, we provide convergence guarantees for training \glspl{NN} that have several layers of dropout and analyze simplified models for deep \glspl{NN}, for which it is possible to obtain an explicit convergence rate that depends on the dropout probability and depth.
We also consider the effect on the sample complexity of using dropout \gls{SGD} and complement the previous results with simulations on realistic \glspl{NN} to examine the convergence rate of dropout empirically.

Before introducing the results of the paper we briefly define the fundamental concepts related to training of \glspl{NN} with dropout that we will use throughout this paper.

\subsection{Dropout and \texorpdfstring{\gls{SGD}}{SGD}}
\label{sec:dropout_in_NN_intro}

A \gls{NN} $\Psi_{W}: \mathcal{X} \to \mathcal{Y}$ with weights $W$ is typically used to predict output $Y \in \mathcal{Y}$ given input $X \in \mathcal{X}$ both of which are sampled from some joint distribution.
For a given loss $l: \mathcal{Y} \times \mathcal{Y} \to \R$, the risk function of $\Psi_{W}$ is usually defined as
\begin{equation}
    \mathcal{U}(W)
    \triangleq \int l( \Psi_{W}(x), y) \d{ \probability{ (X,Y) = (x,y)} },
    \label{eqn:Oracle_risk_function}
\end{equation}
where the distribution is usually given by the empirical distribution of a finite number of samples $\{(x_i, y_i)\}_{i=1}^n \in \mathcal{X} \times \mathcal{Y}$.
In this case, the risk is an \emph{empirical risk}.

Ideally, the \gls{NN} is operated using weights in the set $\arg \min_W \mathcal{U}(W)$.
However, the weights are found in practice by using gradient descent or its stochastic variant \gls{SGD}, which aims to minimize the risk in \eqref{eqn:Oracle_risk_function} by updating the weights in the local direction that minimizes the function.
At time $t$, the weights $\itr{W}{t}$ of the \gls{NN} are namely updated by setting
\begin{equation}
    \itr{W}{t+1}
    = \itr{W}{t} - \itrd{\alpha}{t+1} \itr{\tilde{\Delta}}{t+1}.
    \label{eqn:Update_rule_for_Wtp1_in_terms_of_random_direction_Delta}
\end{equation}
Here, $\itr{\tilde{\Delta}}{t+1}$ is a stochastic estimate of the gradient of \eqref{eqn:Oracle_risk_function} and $\itrd{\alpha}{t+1}$ is a step size which we will specify later.
Let $B_{W}(X,Y)$ be the gradient at $W$ of \eqref{eqn:Oracle_risk_function}.
If the input and output samples $\itr{X}{t+1}, \itr{Y}{t+1}$ are provided at time $t$, then the update of \gls{SGD} is given by
\begin{equation}
    \itr{\tilde{\Delta}}{t+1} = B_{\itr{W}{t}}(\itr{X}{t+1}, \itr{Y}{t+1}).
    \label{eqn:update_direction_SGD_nodropout}
\end{equation}

As we have mentioned, dropout filters are applied to some of the weights $W$ during training by using matrices of random variables $F$ with $\{0,1\}$-valued entries.
Denote by $\itr{F}{t+1}$, $\itr{X}{t+1}, \itr{Y}{t+1}$ the dropout filters and the samples provided to the \gls{SGD} algorithm at time $t$, respectively.
Compared to \eqref{eqn:update_direction_SGD_nodropout}, a dropout algorithm defines the estimate of the gradient update as
\begin{equation}
    \itr{\Delta}{t+1}
    \triangleq \itr{F}{t+1} \odot \mathrm{B}_{ \itr{F}{t+1} \odot \itr{W}{t} }( \itr{X}{t+1}, \itr{Y}{t+1} ),
    \label{eqn:Random_direction_Delta_as_function_of_Backpropagation}
\end{equation}
where $\odot$ denotes the componentwise product.

Note that in \eqref{eqn:Random_direction_Delta_as_function_of_Backpropagation} the filters appear twice.
Firstly, they filter the weights $\itr{W}{t}$ when the gradient is computed depending only on the subnetwork provided by dropping some edges or nodes.
Secondly, they filter the updates in $\itr{\Delta}{t+1}$ since only the remaining weights will be updated.
We remark that in this general formulation, other distributions for the filters than those for dropout and dropconnect are allowed.
For specific examples of distribution of the filter matrices we refer to \Cref{section:generalized_dropout}.

We next present the results of this paper.

\subsection{Summary of results}

Our first result is a formal probability theoretical proof that for any (fully connected) \gls{NN} topology and with differentiable polynomially bounded activation functions (see \Cref{def:Polynomially_bounded_maps}), the iterates of projected \gls{SGD} with dropout-like filters converge.
In particular, a step of projected \gls{SGD} with dropout is given by
\begin{equation}
    \itr{W}{t+1}
    =
    \mathrm{P}_{\mathcal{H}}(\itr{W}{t} - \itrd{\alpha}{t+1} \itr{\Delta}{t+1})
    \quad
    \textnormal{for}
    \quad
    t \in \N_0,
    \label{update_psgd}
\end{equation}
where $\itr{\Delta}{t+1}$ is the estimate of the gradient with dropout in \eqref{eqn:Random_direction_Delta_as_function_of_Backpropagation} and $\mathrm{P}_{\mathcal{H}}$ is an operator that projects the iterates onto a compact convex set $\mathcal{H}$ \citep{oymak2018learning}.
In order to state our first result, we define a \emph{dropout algorithm's risk function} as
\begin{align}
    \mathcal{D}(W)
    \triangleq \int l(\Psi_{f \odot W}(x),y) \d{ \probability{ (F,X,Y) = (f,x,y) } },
    \label{eqn:Dropouts_emperical_risk_function}
\end{align}
and we will consider $l(a,b) = |a-b|^2$ to be the $\ell_2$-loss.
The result is stated informally in the next proposition.

\begin{result}
    (Informal statement of Proposition~\ref{prop:dropout_converges}.)
    Under sufficient regularity of the activation functions, bounded moments and independence of random variables and some assumptions on the boundary $\mathcal{H}$, with update \eqref{update_psgd}, the weights $(\itr{W}{t})_{t}$ converge to a unique stationary set of a projected system of \glspl{ODE}
    \begin{equation}
        \frac{\d{W}}{\d{t}}
        =
        - \nabla_W \mathcal{D}|_{\mathcal{H}}(W) + \pi(W),
        \label{eqn:ODE_dropout_intro}
    \end{equation}
    where $\pi(W)$ is a \emph{constraint term}, which describes the minimum vector required to keep the gradient flow of $\nabla \mathcal{D}$ in $\mathcal{H}$.
\end{result}

This result provides a formal guarantee with the sufficient conditions for dropout algorithms to be well-behaved and at least asymptotically (meaning after sufficiently many iterations) to not suffer from problems that could have arisen from the relation to bond percolation.
Moreover, for a wide range of \glspl{NN} and activation functions the function $\mathcal{D}(W)$ is the expectation of the risk over the dropout's filters distribution, which in our result is not restricted to dropping nodes and can even be coupled to the data.
This result also shows that \gls{SGD} with dropout converges to the stationary points of $\mathcal{D}(W)$.
While a guarantee is necessary, a convergence rate would yield more insight into the trade-offs of the algorithm, especially in the dependence on depth.

In our second result, we go one step beyond the convergence guarantee and compute a bound for the sample complexity of \gls{SGD} with dropout to an $\epsilon$-stationary point of a generic smooth nonconvex function $\mathsf{D}(W)$.
We say $W \in \mathcal{W}$ is an $\epsilon$-stationary point of $\mathsf{D}$ if $\pnorm{\nabla \mathsf{D}(W)}{2} \leq \epsilon$ holds.
Note that stationary points are not necessarily minima, but the sample complexity, understood as the number of iterations $T$ required to reach $\epsilon$-stationarity, is usually associated with the complexity of the function to be optimized.

For a generic smooth nonconvex function $\mathsf{D}(W)$, we consider dropout to be \gls{SGD} with the update in \eqref{eqn:Random_direction_Delta_as_function_of_Backpropagation}, where filters $F$ are chosen independently at each step and are $\{0,1\}$-valued for each parameter.
In our result we assume boundedness and Lipschitzness conditions on $\mathsf{D}(W)$.
Moreover, under some additional assumptions on the loss function, examples of \glspl{NN} with sigmoid activation functions $\sigma(t)= 1/(1+\exp(-t))$ are also covered by our result.
In this particular case, $\mathsf{D}(W) = \mathcal{D}(W)$ holds with the definition in \eqref{eqn:Dropouts_emperical_risk_function}.
For the general case we prove the following:

\begin{result}
    (Informal statement of Proposition~\ref{prop:convergence_rate_generic_dropout}.)
    Assume that $\mathsf{D}(W)$ has enough regularity and satisfies some boundedness and Lipschitzness assumptions.
    Let $\itrd{W}{t}$ be iterates of \eqref{update_psgd}.
    For any $T \in \N$ there exist $c>0$ and $c_1, c_2 > 0$ and $\itrd{\alpha}{t}=\eta$ constant such that if $p > c/T$, then as $T \to \infty$,
    \begin{equation}
        \min_{t \in [T]} \expectationBig{\pnorm{\nabla \mathsf{D}(\itrd{W}{t})}{2}^2} = O \Bigl( \sqrt{\frac{p(c_1 + (1-p)c_2)}{T}} \Bigr).
        \label{eqn:convergence_rate_generic_dropout_intro}
    \end{equation}
\end{result}

Hence, at least $T$ iterations of dropout-like \gls{SGD} algorithms are required to reach an $O((p(c_1 + (1-p)c_2)/T)^{1/4})$-stationary point of nonconvex smooth functions in expectation.
Here, $c_1, c_2$ are constants depending on the data and function, respectively.
Compared to the theoretical optimum rate of $O(T^{-1/4})$ for \gls{SGD} on nonconvex smooth functions \citep{drori2020complexity}, this result shows that dropout changes the optimization landscape and approximate stationary points are easier to find depending on the dropout probability.
In this setting, we also consider the complexity when we scale the weights by a factor $1/p$ during training, which is commonly used to compensate the effect of dropout on the convergence rate.

It must be emphasized that Proposition~\ref{prop:convergence_rate_generic_dropout} does not assume much structure on the objective function.
As consequence, in spite of the fact that the bound in \eqref{eqn:convergence_rate_generic_dropout_intro} holds in some settings with deep \glspl{NN}, the depth of such \gls{NN} would appear only \emph{implicitly} in the constants $c_1, c_2$.
In order to determine the dependence between the convergence rate and the depth of a \gls{NN} \emph{explicitly}, one must exploit the specific structure of a \gls{NN}, which we leverage in our next result.

Our third result in this paper is an explicit upper bound for the rate of convergence of regular \gls{GD} on the limiting \glspl{ODE} of dropout algorithms for arborescences (a class of trees, see \Cref{fig:Arborescence_picture} for an example), of arbitrary depth with linear activation functions $\sigma(t) = t$.
In particular, we will consider the update rule
\begin{equation}
    \itrd{W}{t+1} = \itrd{W}{t} - \alpha \nabla \mathcal{D}(\itrd{W}{t}).
    \label{eqn:iterates_gradient_descent_intro}
\end{equation}
Analyzing the convergence of training algorithms on simplified \glspl{NN} with linear activation functions is commonly used to gain insight into more complex models, see e.g.\ \citep{arora2018convergence,shamir2018exponential,bartlett2018gradient}.
Even without a dropout algorithm present, this task already provides a substantial theoretical challenge as the optimization landscape is nonconvex.
Our choice to restrict the analysis to arborescences allows us to quantitatively tie our upper bound for the convergence rate to the depth and the number of paths within the arborescence.
We prove the following:

\begin{result}
    (Informal statement of Proposition~\ref{prop:estimate_nu_dependent_p}.)
    Assume that the base graph $G$ of the \gls{NN} is an arborescence of depth $L$ with $|\mathcal{L}(G)|$ leaves and the filters $F$ follow the distribution prescribed by \emph{Dropconnect} or \emph{Dropout} with dropout probability $1-p$ (see Proposition~\ref{prop:estimate_nu_dependent_p}).
    Then there exist $\alpha > 0$ and $1> \eta > 0$ depending on the initialization such that the iterates of \eqref{eqn:iterates_gradient_descent_intro} satisfy
    \begin{equation}
        \mathcal{D}(\itrd{W}{t}) - \min_{W} \mathcal{D}(W)
        \leq
        \bigl( \mathcal{D}(\itrd{W}{0}) - \min_{W} \mathcal{D}(W) \bigr) \exp(- \omega t/2)
        ,
        \label{eqn:PL_inequality_upper_bounded_and_iterated_intro}
    \end{equation}
    with
    \begin{equation}
        \omega
        = \mathrm{O} \Bigl( \frac{p^L}{L \cardinality{\mathcal{L}(G)}^2} \eta^{2L}\Bigr)
        .
    \end{equation}
\end{result}

One important consequence of this result is that the convergence rate exponent indeed deteriorates by a factor $p^L$ in these \glspl{NN}.
Finally, we complement this result with numerical experiments.
We target the dependency of the convergence on $p$ for more realistic wider and nonlinear networks on commonly used datasets.
Perhaps surprisingly, we do not observe an exponential decrease of the convergence rate exponent due to dropout in these simulations.
We will offer some heuristic explanation for this result by looking at the update rate of a generic weight.

Our results lead to the following consequences.
First, whenever the iterates of a dropout algorithm with $\ell_2$-loss are bounded, they are guaranteed to converge to a stationary point of the risk function $\mathcal{D}(W)$ induced by the dropout algorithm.
Secondly, we prove rigorously that the convergence rate when training with e.g.\ \emph{Dropout} or \emph{Dropconnect} can change the convergence rate on the empirical risk depending on $p$ and in arborescences can decrease by as much as a factor $p^L$.
For more realistic wider networks, however, we conduct numerical experiments that suggest that the convergence rate is not necessarily affected by depth as much across different dropout rates $1-p$ in neural networks with just a few layers of dropout.

Our findings motivate further theoretical study of the convergence rate of dropout for deep and wide networks.
We suspect that there is a transition regime of the convergence rate.
Such transition would affect the dependence on $p$ and would be observed when going from networks with many layers of dropout with small width, where dependence on the rate may be exponential in $p$, to networks with a few layers of dropout but very wide, where dependence is not exponential anymore.

\subsection{Literature overview}
\label{sec:Related_work}

The first description of a dropout algorithm was by \cite{hinton2012improving}.
Diverse variants of the algorithm have appeared since, including versions in which edges are dropped \citep{wan2013regularization}; groups of edges are dropped from the input layer \citep{devries2017improved}; the distribution of the filters are Gaussian \citep{kingma2015variational,molchanov2017variational}; the removal probabilities change adaptively \citep{ba2013adaptive,li2016improved}; and that are suitable for recurrent \glspl{NN} \citep{zaremba2014recurrent,semeniuta2016recurrent}.
The performance of the original algorithm has been investigated on datasets \citep{hinton2012improving,srivastava2014dropout}, and dropout algorithms have found application in e.g.\ image classification \citep{krizhevsky2012imagenet}, handwriting recognition \citep{pham2014dropout}, heart sound classification \citep{kay2016dropconnected}, and drug discovery in cancer research \citep{urban2018deep}.

Theoretical studies of dropout algorithms have focused on their regularization effect.
The effect was first noted by \cite{hinton2012improving,srivastava2014dropout}, and subsequently investigated in-depth for both linear \glspl{NN} as well as nonlinear \glspl{NN} by \cite{baldi2013understanding,wager2013dropout,baldi2014dropout,wei2020implicit}.
Within the context of matrix factorization, it has been shown that \emph{Dropout}'s regularization induces a shrinkage and a thresholding of the singular values of the matrix at the optimum \citep{cavazza2017dropout}.
Characterizations of \emph{Dropout}'s risk function and \emph{Dropout}'s regularizer for (usually linear) \glspl{NN} can be found in \cite{mianjy2018implicit,mianjy2019dropout,pal2020regularization}.
Random networks with \emph{Dropout} have been also studied in \cite{sicking2020characteristics} and in \cite{huang2019mean}.

Detailed theoretical investigations into the convergence of dropout algorithms are however relatively scarce.
While revising this paper, new results appeared and these now give insight into the convergence rate of \emph{Dropout} in \textrm{ReLU} shallow \glspl{NN} for a classification task \citep{Mianjy2020OnCA}.
In \cite{Mianjy2020OnCA}, it is shown that $\mathrm{O}(1/\epsilon)$ iterations of \gls{SGD} to reach $\epsilon$-suboptimality for the test error are required; interestingly, it is independent of the dropout probability because of their assumption that the data distribution is separable by a margin in a particular Reproducing Kernel Hilbert space.
Compared to our generic convergence result, we do not assume structure on the predictor or data and look instead at the iterations required to reach $\epsilon$-stationarity in nonconvex functions using dropout-like \gls{SGD}.
A study of the asymptotic convergence rate of \emph{Dropout} and \emph{Dropconnect} on shallow linear neural networks has also appeared recently \citep{senencerda2020asymptotic}.
\putin{There, an asymptotic convergence rate for dropout linear shallow networks is provided.
    Namely, for wide linear shallow networks with width $D$ and dropout probability $1-p > 0$ a local convergence rate close to a minimum of $O(p(1-p)/(pD + 1-p))$ is found.
}
Finally, it must be noted that convergence properties have been thoroughly studied within the context of \glspl{NN} being trained without dropout algorithms, see e.g.\ \cite{arora2018convergence,shamir2018exponential, zou2020gradient, gao2021global} and references therein.

Dropout algorithms can, by construction, be understood as forms of \gls{SGD}.
More generally, dropout algorithms are all stochastic approximation algorithms.
The first stochastic approximations algorithms were introduced by \cite{robbins1951stochastic,kiefer1952stochastic}, and have been subject to enormous literature due to their ubiquity.
For overviews and their application to \glspl{NN}, we refer to books by \cite{kushner2003stochastic,borkar2009stochastic,bertsekas1995neuro}.

\section*{A word on notation}

In this paper we index deterministic sequences with curly brackets: $\itrd{\alpha}{1}, \itrd{\beta}{1}$, etc. This distinguishes them from sequences of random variables, which we index using square brackets, e.g.\ $\itr{X}{1}, \itr{Y}{1}$, etc.

Deterministic vectors are written in lower case like $x \in \realNumbers^d$, but an exception is made for random variables (which are always capitalized).
Matrices are also always capitalized.
For a function $\sigma : \R \to \R$ and a matrix $A \in \R^{a \times b}$, $a, b \geq 1$, we denote by $\sigma(A)$ the matrix with $\sigma$ applied componentwise to $A$.
Subscripts will be used to denote the entries of any tensor, e.g.\ $x_i$, $A_{i,j}$, or $T_{i,j,l}$.
For any vector $x \in \realNumbers^d$, the $\ell_2$-norm is defined as
$
    \pnorm{x}{2}
    \triangleq ( \sum_{i=1}^d | x_i |^2 )^{1/2}.
$
For any matrix $A \in \realNumbers^{a \times b}$, the Frobenius norm is defined as
$
    \pnorm{A}{\mathrm{F}}
    \triangleq ( \sum_{i=1}^a \sum_{j=1}^b | A_{i,j} |^2 )^{1/2}.
$
For two matrices $A, B$, the Hadamard (componentwise) product is denoted by $A \odot B$.

Let $\naturalNumbersPlus$ be the strictly positive integers and $\naturalNumbersZero \triangleq \naturalNumbersPlus \cup \{ 0 \}$.
For $l \in \naturalNumbersPlus$, we denote $[l] = \{1, \ldots, l\}$.
For a function $g \in C^2(\R^n)$, we denote the gradient and Hessian of $g$ with respect to the Euclidean norm $\norm{\cdot}_2$ in $\R^n$ by $\nabla g$ and $\nabla^2 g$, respectively.

 \section{Model}
\label{sec:Preliminaries}

We now formally define \glspl{NN}, which we had depicted in \Cref{fig:Dropout_on_a_neural_network}, as well as the class of activation functions that we will use for the convergence guarantee in our first result below.

\subsection{Neural networks, and their structure}

Let $L$ denote the number of layers in the \gls{NN}, and $d_l \in \naturalNumbersPlus$ the output dimension of layer $l = 1, \ldots, L$.
Let $W_{l+1} \in \realNumbers^{d_{l+1} \times d_{l}}$ denote the matrix of weights in between layers $l$ and $l+1$ for $l = 0, 1, \ldots, L-1$.
Denote $W = (W_L, \ldots, W_1) \in \mathcal{W}$ with $\mathcal{W} \triangleq \realNumbers^{d_{L} \times d_{L-1}} \times \cdots \times \realNumbers^{d_1 \times d_0}$ the set of all possible weights.
In this paper, we consider \glspl{NN} without biases.

\glsreset{NN}
\begin{definition}
\label{def:NN__Feedforward_recursion}
    Let $\sigma$ be an activation function $\sigma: \R \to \R$.
    A \emph{\gls{NN} with $L$ layers} is given by the class of functions $\Psi_{W} : \R^{d_0} \to \R^{d_L}$ defined iteratively by
    \begin{align}
        A_0
= x,
        \quad
A_{i}
= \sigma( W_i A_{i-1} )
        \quad
        \forall i \in \{ 1, \ldots, L-2 \},
        \quad
\Psi_{W}(x)
= W_L A_{L-1} = A_L.
    \end{align}
\end{definition}

Canonical activation functions include the \gls{ReLU} function $\sigma(t) = \max \{ 0, t \}$, the sigmoid function $\sigma(t) = 1/(1+\e{-t})$, and the linear function $\sigma(t) = t$.
In Sections~\ref{sec:Preliminaries} and \ref{sec:Results__Almost_sure_convergence_of_projected_dropout_algorithms} we restrict to the case that $\sigma$ belongs to a class of polynomially bounded differentiable functions.

\begin{definition}
    For $\sigma:\R \to \R$ differentiable, denote the $l$th derivative of $\sigma$ by $\sigma^{(l)}$.
    The set of polynomially bounded maps with continuous derivatives up to order $r \in \naturalNumbersZero$ is given by
    \begin{align}
        C_{\mathrm{PB}}^r(\R)
        = \bigl\lbrace \sigma \in C^r(\R) \big\vert
         &
        \forall l = 0, \ldots, r ~
        \exists k_l > 0 :
        \sup_{x \in \R} | \sigma^{(l)}(x) (1 + x^2)^{-k_l} | < \infty \bigr\rbrace.
        \nonumber
    \end{align}
    \label{def:Polynomially_bounded_maps}
\end{definition}

Note that the linear and sigmoid activation function both belong to $C_{PB}^r(\R)$ for any $r \in \naturalNumbersZero$.
Also, any polynomial activation function $P(x) \in \R[x]$ belongs to $C_{\mathrm{PB}}^{\mathrm{deg}(P)}(\R)$.
The \gls{ReLU} activation function is not in $C_{\mathrm{PB}}^r(\R)$ for any $r \in \naturalNumbersZero$.
However, because the class $C_{\mathrm{PB}}^r(\R)$ contains polynomials of any degree, we can approximate cases such as \gls{ReLU} by using, e.g., the softplus activation function $\sigma_t(x) = \log(1+\exp(tx))/t$, which satisfies that $\lim_{t \to \infty} \sigma_t(x) = \mathrm{ReLU}(x)$ for every $x \in \R$.
Note that the softplus activation function belongs to $C_{\mathrm{PB}}^2(\R)$.

\subsection{Backpropagation, and \texorpdfstring{\gls{SGD}}{stochastic gradient descent}}

In \Cref{sec:dropout_in_NN_intro} we have defined the risk $\mathcal{U}(W)$ that in the previous notation now depends on a loss $l: \R^{d_L} \times \R^{d_L} \to \R$.
Throughout this article, we will specify the Euclidean $\ell_2$-norm $l(x,y) \triangleq \pnorm{x - y}{2}^2$ as our loss function of interest without loss of generality.
\footnote{The results can be extended to other smooth loss functions $l(x,y)$ whose partial derivatives can be bounded by polynomials of finite degree.}

Furthermore, in the definition of $\mathcal{U}(W)$ in \eqref{eqn:Oracle_risk_function}, we make no distinction between an oracle risk function or empirical risk function.
Both situations are covered by the definition in \eqref{eqn:Oracle_risk_function}.
Hence, our results cover the empirical risk case when we have a finite number of samples, as well as the online learning case, where a new sample is provided at each step of \gls{SGD}.
What we do assume is that one has the ability to repeatedly draw independent and identically distributed samples either distribution.

In an attempt to find a critical point in the set $\arg \min_W \mathcal{U}(W)$, as mentioned in \eqref{sec:dropout_in_NN_intro}, \gls{SGD} is commonly used.
Let $\process{ (\itr{Y}{t}, \itr{X}{t}) }{t \in \naturalNumbersPlus}$ be a sequence of independent copies of $(X,Y)$, let $\itr{W}{0} \in \mathcal{W}$ be an arbitrary nonrandom initialization of the weights.
For $i = 1, \ldots, L$, $r = 1, \ldots, d_{i+1}$, $l=1,\ldots, d_{i}$, the weights are iteratively updated according to
\begin{equation}
    \itr{W_{i,r,l}}{t+1}
    = \itr{W_{i,r,l}}{t} - \itrd{\alpha}{t+1} \bigl( \mathrm{B}_{\itr{W}{t}}( \itr{X}{t+1}, \itr{Y}{t+1} ) \bigr)_{i,r,l}
    \quad
    \label{update_sgd}
\end{equation}
for $t = 0, 1, 2$, \emph{et cetera}.
Here $\{ \itrd{\alpha}{t} \}_{t \in \naturalNumbersPlus }$ denotes a positive, deterministic step size sequence, and the estimate of the gradient $B_W(\cdot,\cdot) = \nabla_{W} l(\Psi_W(\cdot),\cdot)$ is computed using the backpropagation algorithm, which is given in Definition~\ref{definition:Backpropagation} in \refAppendixSection{sec:Appendix_Backpropagation}.
The stochastic gradient is an unbiased estimate of the gradient of $\mathcal{U}(W)$.
In particular, we have
\begin{equation}
    \expectation{ \bigl( \mathrm{B}_W( X, Y ) \bigr)_{i,r,l} }
    = \expectationBig{\frac{ \partial l(\Psi_W(x),y) }{ \partial W_{i,r,l} }} = \frac{ \partial \mathcal{U}(W) }{ \partial W_{i,r,l} }
    = (\nabla U)_{i,r,l}.
    \label{eqn:Backpropagation_calculates_the_gradient_operator_pointwise}
\end{equation}

\takeout{
    The algorithm in \eqref{update_sgd} is a step in a \gls{SGD} algorithm.
    In particular, we have
    \begin{equation}
        \bigl( \mathrm{B}_W(x,y) \bigr)_{i,r,l}
        = \frac{ \partial l(\Psi_W(x),y) }{ \partial W_{i,r,l} }
        \label{eqn:Backpropagation_calculates_the_gradient_operator_pointwise}
    \end{equation}
    since $\sigma \in C^1(\R)$ by assumption.
    By substituting \eqref{eqn:Backpropagation_calculates_the_gradient_operator_pointwise} into \eqref{update_sgd}, one identifies the \gls{SGD} algorithm.
    Under additional assumptions on the distribution $\mu$ and by linearity of the expectation and gradient operators, one can then furthermore see that
    \begin{equation}
        \expectation{ \bigl( \mathrm{B}_W( X, Y ) \bigr)_{i,r,l} }
        = \frac{ \partial \mathcal{U}(W) }{ \partial W_{i,r,l} }
        = (\nabla U)_{i,r,l}
        ,
    \end{equation}
    which suggests that $\itr{W}{t}$ may converge to the critical set $\{ W ~ : ~ \nabla \mathcal{U}(W) = 0 \}$.
    However, there is no guarantee that the iterates $\itr{W}{t}$ in \eqref{update_sgd} converge to a point in $\arg \min_W \mathcal{U}(W)$ because \eqref{eqn:Oracle_risk_function} is not convex.
    This is one reason that the surprising success of \eqref{update_sgd} in \glspl{NN} is still a key question in the field of machine learning; as a starting point an interested reader may look at e.g.\ \citep{gunasekar2017implicit}.
}

\subsection{Dropout algorithms, and their risk functions} \label{section:generalized_dropout}

Dropout algorithms use $\{ 0, 1 \}$-valued random matrices as filters of weights during the backpropagation step of \gls{SGD}.
More precisely, we examine the following class of dropout algorithms.
Let $(F,X,Y) : \Omega \to \{0,1\}^{d_L \times d_{L-1}} \times \ldots \times \{0,1\}^{d_1 \times d_{0}} \times \realNumbers^{d_0} \times \realNumbers^{d_L}$ be a random variable on the probability space $(\Omega, \mathcal{F}, \mathbb{P})$.
Here, we write $F = ( F_L, \ldots, F_1 )$ and $F_{i+1} \in \{ 0, 1 \}^{d_{i+1} \times d_{i}}$ for $i = 0, \ldots, L-1$, similar to how we notate weight matrices.
Let $\{ ( \itr{F}{t}, \itr{X}{t}, \itr{Y}{t} )\}_{t \in \naturalNumbersPlus }$ be a sequence of independent copies of $(F,X,Y)$.
In tensor notation, the weights are updated by using \eqref{eqn:Update_rule_for_Wtp1_in_terms_of_random_direction_Delta} with the random direction $\itr{\Delta}{t+1}$ for dropout given in \eqref{eqn:Random_direction_Delta_as_function_of_Backpropagation}.
For each dropout algorithm a different filter distribution will be chosen.
We can mention a few:

\begin{itemize}
    \item[(i)] In canonical \emph{Dropout} \citep{hinton2012improving}, $F_{i,r,l^{\prime}} = F_{i,r,l} \sim \mathrm{Bernoulli}(p)$ for any $l,l^{\prime} \in [d_i]$ with $p = 1/2$.
    \item[(ii)] In \emph{Dropconnect} \citep{wan2013regularization}, $F_{i,r,l} \sim \mathrm{Bernoulli}(p)$ for all $i, r, l$ with $p = 1/2$.
    \item[(iii)] In \emph{Cutout} \citep{devries2017improved}, $ F_{1,r,l} = 0$ whenever $\abs{r - S_1} < c$, $c \in \N_{+}$ and $\abs{l - S_2} < c$ with $(S_1, S_2) \sim \mathrm{Uniform}([d_1] \times [d_0])$.
\end{itemize}

In fact, the class of dropout algorithms we consider is quite large.
For example, $\itr{F}{t}$ can depend on $(\itr{X}{t},\itr{Y}{t})$, and $\itr{F_i}{t}$ does not need to have the same distribution as $\itr{F_j}{t}$ for $i \neq j$.
Recall, however, that if for some filter $\itr{F}{t+1}_{i,r,l} = 0$ for some $i,r,l$, then in \eqref{eqn:Update_rule_for_Wtp1_in_terms_of_random_direction_Delta} , $\itr{ \Delta }{t}_{i,r,l} = 0$ and we have $\itr{W}{t}_{i,r,l} = \itr{W}{t+1}_{i,r,l}$.
In other words, filtered variables are not updated with these dropout algorithms.

If $\itr{F}{t}$ is independent of $( \itr{X}{t}, \itr{Y}{t} )$ for each $t \in \naturalNumbersZero$ and $\Omega$ countable, then the dropout algorithm's risk function in \eqref{eqn:Dropouts_emperical_risk_function} simplifies to
\begin{equation}
    \mathcal{D}(W)
    = \sum_f \probability{F = f } \sum_{x,y} l(\Psi_{f \odot W}(x),y) \probability{ (X,Y) = (x,y) }
    .
\end{equation}
Here the sums are over all possible outcomes of the random variables $F$ and $(X,Y)$, respectively.
One implication of Proposition~\ref{prop:dropout_converges} in the result of the next \Cref{sec:Results__Almost_sure_convergence_of_projected_dropout_algorithms} is that dropout algorithms of the kind in \eqref{eqn:Update_rule_for_Wtp1_in_terms_of_random_direction_Delta}, \eqref{eqn:Random_direction_Delta_as_function_of_Backpropagation} converge to a critical point of \eqref{eqn:Dropouts_emperical_risk_function}.
 \section{Convergence of projected dropout algorithms}
\label{sec:Results__Almost_sure_convergence_of_projected_dropout_algorithms}

Our first result pertains to the convergence of dropout algorithms for a wide range of activation functions and dropout filters.
While convergence is expected in practice, we prove such convergence rigorously.
In order to control the iterates of the stochastic algorithm, we project the iterates into a compact set.
The projection assumption is common when investigating the convergence of stochastic algorithms \citep{kushner2003stochastic,borkar2009stochastic,bertsekas1995neuro,oymak2018learning}; it essentially bounds the weights.
For example, for $\itr{V}{t} \in \R$ and an update function $f:\R \to \R$, $f(\itr{V}{t})$ is projected onto an interval $[a,b]$ is by clipping and setting $\itr{V}{t+1} = \min\{ \max\{ f( \itr{V}{t} ), a \}, b \}$.
There are also results involving generalization bounds for \glspl{NN} where bounded weights play a role in controlling the learning capacity of the \gls{NN} \citep{neyshabur2015norm}.

\subsection{Almost sure convergence}
We first consider the notation and assumptions regarding the projection step of \gls{SGD}.
Let $\mathcal{H} \subseteq \mathcal{W}$ be a convex compact nonempty set and let $\mathrm{P}_{\mathcal{H}} : \mathcal{W} \to \mathcal{H}$ be the projection onto $\mathcal{H}$.
By compactness and convexity of $\mathcal{H}$, the projection is unique.
In a projected dropout algorithm, the weight update in \eqref{eqn:Update_rule_for_Wtp1_in_terms_of_random_direction_Delta} is replaced by \eqref{update_psgd}.
Because of the projection, our analysis will tie the limiting behavior of \eqref{update_psgd} to a \emph{projected} \gls{ODE}.
To state such type of \gls{ODE}, we need to define a \emph{constraint term} $\pi(W)$, which is defined as the minimum vector required to keep the solution of the gradient flow
\begin{equation}
    \frac{\d{W}}{\d{t}}
    =
    - \nabla_W \mathcal{D}|_{\mathcal{H}}(W) + \pi(W)
    \label{eqn:ODE_dropout}
\end{equation}
in $\mathcal{H}$.
\refAppendixSection{sec:Projection_operator} defines the projection term carefully for the case that $\mathcal{H}$'s boundary is piecewise smooth.
Finally, define the set of stationary points
\begin{equation}
    S_\mathcal{H}
    \triangleq \{ W \in \mathcal{H} ~ : ~ - \nabla_W \mathcal{D}|_{\mathcal{H}}(W) + \pi(W) = 0 \}
    .
\end{equation}
The set $S_\mathcal{H}$ can be divided into a countable number of disjoint compact and connected subsets $S_1, S_2, \cdots$, say.
We choose the following set of assumptions:
\begin{itemize}
    \itemsep-0.25em
    \item[(N1)] $\sigma \in C^2_{\mathrm{PB}}(\R)$.
    \item[(N2)] $\expectation{ \pnorm{Y}{2}^m \pnorm{X}{2}^n } < \infty \, \forall m \in \{ 0, 1, 2 \}, n \in \naturalNumbersZero$.
    \item[(N3)] The random variables $(\itr{F}{t};\itr{X}{t}; \itr{Y}{t})_{t \in \N}$ are independent copies of $(F,X,Y)$.
    \item[(N4)] The step sizes $\itrd{\alpha}{t}$ satisfy
        \begin{equation}
            \label{eqn:Step_sizes_diverge_but_not_too_fast}
            \sum_{t=1}^\infty \itrd{\alpha}{t} = \infty,
            \quad
            \sum_{t=1}^\infty ( \itrd{\alpha}{t} )^2 < \infty.
        \end{equation}
    \item[(N5)] $\sigma \in C^r_{\mathrm{PB}}(\R)$, with $\dim(\mathcal{W}) \leq r$.
    \item[(N6)] $- \nabla_W \mathcal{D}|_\mathcal{H}(W) + \pi(W) \neq 0$ whenever $\nabla_W \mathcal{D}|_\mathcal{H}(W) \neq 0$.
\end{itemize}

We are now in position to state our first result:

\begin{proposition}
    \label{prop:dropout_converges}
Let $\process{ \itr{W}{t} }{ t \in \naturalNumbersZero }$ be the sequence of random variables generated by \eqref{update_psgd} with \eqref{eqn:Random_direction_Delta_as_function_of_Backpropagation} on a probability space $(\Omega, \mathcal{F}, \mathbb{P})$.
    Under assumptions (N1)--(N4) , there is a set $N \subset \Omega$ of probability zero such that for $\omega \not\in N$, $\{ \itr{W}{t}(\omega) \}$ converges to a limit set of the projected \gls{ODE} in \eqref{eqn:ODE_dropout}.
If moreover (N5)--(N6) hold, then for almost all $\omega \in \Omega$, $\{ \itr{W}{t}(\omega) \}_{t \in \N}$ converges to a unique point in $\{ W \in \mathcal{H} | \nabla \mathcal{D}|_{\mathcal{H}}(W) = 0 \}$.
\end{proposition}

Theoretically, Proposition~\ref{prop:dropout_converges} guarantees that projected dropout algorithms converge for regression with the $\ell_2$-norm almost surely.
Proposition~\ref{prop:dropout_converges} implies that if one is using a regular \emph{nonprojected} dropout algorithm and one sees that the iterates $\process{ \itr{W}{t} }{t > 0}$ are bounded, then these iterates are in fact converging to a stationary point of \eqref{eqn:Dropouts_emperical_risk_function}.
Assumptions (N5)--(N6) are technical but are expected to hold in many cases.
In particular, (N5) holds for the uniformly convergent approximation to a \gls{ReLU} activation function given by softplus $\sigma_t(x) = \log(1 + \exp(tx))/t$, and holds for many smooth activation functions.
Also (N6) is expected to hold when $\mathcal{H}$ is generic polytope for which the gradient $\nabla \mathcal{D}$ is not exactly orthogonal to the normal to the surface.

Observe also that Proposition~\ref{prop:dropout_converges} holds remarkably generally.
For example, the dependence structure of $(F,X,Y)$ as random variables is not restricted; it covers commonly used dropout algorithms such as \emph{Dropout}, \emph{Dropconnect}, and \emph{Cutout}; and it holds for differentiable activation functions.
Proposition~\ref{prop:dropout_converges} includes also online and offline learning, depending on the distribution $\mu$ from which we sample.

Our proof of Proposition~\ref{prop:dropout_converges} is in \refAppendixSection{appendix:Proof_that_dropout_converges} and relies on the framework of stochastic approximation in \cite[Theorem~2.1, p.~127]{kushner2003stochastic}.
In the background the stochastic process $\process{ \itr{W}{t} }{t > 0}$ is being scaled in both parameter space and time so that the resulting sample paths provably converge to the gradient flow in \eqref{eqn:ODE_dropout}.
Examining the proof, we expect that Proposition~\ref{prop:dropout_converges} can be extended to cases where the filters as random variables have finite moments, for example, when they are Gaussian distributed \citep{molchanov2017variational}.
Concretely, the proofs of Lemmas~\ref{lem:Boundedness_of_the_variance_of_stochastic_iterands} and \ref{lem:Expectation_cond_Ft_of_Delta_i} in \refAppendixSection{appendix:Proof_that_dropout_converges} rely only on the assumption that $F$ has finite moments, and may therefore be extended.

\subsection{Generic sample complexity for dropout \texorpdfstring{\gls{SGD}}{SGD}}

Examining Proposition~\ref{prop:dropout_converges}, we note that it does not give insight into the convergence rate or the precise stationary point of $\mathcal{D}(W)$ to which the iterates $\{ \itr{W}{t} \}$ converge.
A related goal in stochastic optimization is to ask for the number of iterations of \eqref{eqn:Update_rule_for_Wtp1_in_terms_of_random_direction_Delta} required to achieve a point close to stationarity in expectation, also referred to the sample complexity of the algorithm.
We say $W \in \mathcal{W}$ is an $\epsilon$-stationary point of a differentiable function $\mathsf{D}$ if $\pnorm{\nabla \mathsf{D}(W)}{2} \leq \epsilon$ holds.
For nonconvex functions $\mathsf{D}$ with a Lipschitz continuous gradient $\nabla \mathsf{D}$, \gls{SGD} convergence to an $\epsilon$-stationary point in expectation can be achieved in $O(\epsilon^{-4})$ iterations; see \cite{bottou2018optimization,drori2020complexity}.

We will consider nonconvex functions with a Lipschitz continuous gradient and assume that the filters $F$ and the data $Z=(X,Y)$ are independent.
We will also assume that the distribution of $Z$ is well-behaved so as to guarantee that we also have the following relations for the functions $r, \mathsf{U}$ and $\mathsf{D}$:
\begin{align}
    \mathsf{U}(W)        & = \expectationWrt{r(W, Z)}{Z}, \nonumber                                                                            \\
    \mathsf{D}(W)        & = \expectationWrt{\mathsf{U}(F \odot W)}{F}, \ \text{ and}                                                          \\
    \nabla \mathsf{D}(W) & = \expectationWrt{F \odot \nabla \mathsf{U}(F \odot W)}{F} = \expectationWrt{F \odot \nabla r(F \odot W, Z)}{F, Z}.
    \nonumber
\end{align}
Note that the function $r$ in this setting includes the loss function formulation from \eqref{eqn:Oracle_risk_function} with
\begin{equation}
    r(W, Z) = l( \Psi_{W}(X), Y), \quad \text{ and } \quad Z = (X,Y),
    \label{eqn:special_example_regression}
\end{equation}
and in general, at time $t$ the update rule will be
\begin{equation}
    \itr{W}{t+1}
    = \itr{W}{t} - \itrd{\alpha}{t+1} \itr{F}{t} \odot \nabla r \Bigl( \itr{W}{t} \odot \itr{F}{t}, \itr{Z}{t} \Bigr).
    \label{eqn:Update_rule_sample_complexity_dropout}
\end{equation}

In the case of \emph{dropout}, for example, we expect that the sample complexity of finding an $\epsilon$-stationary point for the empirical risk will change depending on the dropout probability $1-p$.
In particular, if $p \downarrow 0$ and $\pnorm{\nabla \mathsf{U}(W)}{\infty} < C$ holds for any $W \in \mathcal{W}$, then $\nabla \mathsf{D}(W) = \E_{F}[ F \odot \nabla \mathsf{U}(F \odot W) ] = O(pC)$.
On the other hand if $p \uparrow 1$, then the variance of $F \odot \nabla \mathsf{U}(F \odot W)$, will also be small.
We make these intuitions rigorous in the next proposition.
For some $N \in \N$, we let $\mathcal{W} = \R^{N}$ be the parameter space and $z \in \mathcal{Z} \subseteq \R^{d}$ a Lebesgue measurable set.
We assume the following:

\begin{itemize}
    \itemsep-0.25em
    \item[(Q1)] $r \in C^{1}(\mathcal{W}, \mathcal{Z})$ and $\sup_{W \in \mathcal{W}, Z \in \mathcal{Z}} ~ |r(W, Z)| < M$.
    \item[(Q2)] $\sup_{W \in \mathcal{W}, Z \in \mathcal{Z}}  ~ \pnorm{\nabla r(W, Z)}{2} < S$.
    \item[(Q3)] $\nabla \mathsf{U}(W)$ is Lipschitz with Lipschitz constant $\ell$ (also referred to as $\mathsf{U}$ being $\ell$-smooth).
    \item[(Q4)] The random variable $F: \Omega \to \{0,1\}^{N}$ satisfies $\expectation{F} = p(1,\ldots, 1) \in \mathcal{W}$ for $p \in (0,1]$.
    \item[(Q5)] The iterates $(\itr{W}{t})_t$ of \eqref{eqn:Update_rule_for_Wtp1_in_terms_of_random_direction_Delta} are bounded, that is, $\sup_{t} \pnorm{\itr{W}{t}}{2} < R$ almost surely.
\end{itemize}
Except for (Q4) and (Q5), all other assumptions are routinely used in sample complexity analysis.
While the assumptions of Proposition~\ref{prop:convergence_rate_generic_dropout} below hold for general nonconvex smooth functions $\mathsf{D}$, in the case of \glspl{NN} and the setting in \eqref{eqn:special_example_regression} we remark that there are examples that satisfy these assumptions such as the following one:

\begin{example}
    In a binary classification setting, the set $\mathcal{Z}$ is compact, that is, the data pairs $(x,y) \in \mathcal{Z}$ take values in a compact set where $y \in \{0,1\}$ are labels for the two classes.
    A \gls{NN}, denoted by $\tilde{\Psi}_{W}(\cdot)$, uses sigmoid activation functions $\sigma(t) = 1/1+\exp(-t)$ with output in $\R$.
    The output of $\tilde{\Psi}_{W}$ is then used for binary classification with a logistic map, that is, the predicted probability of belonging to one of the classes is given by $\Psi_{W}(x) = 1/(1 + \exp(-\tilde{\Psi}_{W}(x))$.
    In this setting, assumptions (Q1)--(Q3) will hold if the loss $l$ is also smooth (such as the $\ell_2$-loss).
    In this case, we have $\mathcal{D}(W) = \mathsf{D}(W)$ and the constants in (Q1)--(Q5) will also indirectly depend on the depth and width of the \gls{NN}.
    \label{example:NN_satisfies_conditions}
\end{example}

Regarding (Q4), note that it allows for dependencies between filters.
We also assume (Q5) for the sake of simplicity: we could instead use projected SGD with updates from \eqref{update_psgd} instead of (Q5), but using projected SGD would leave the scalings in $p$ and $T$ invariant.
\footnote{With projected SGD, we would moreover have to use the expression $\nabla \mathsf{U}^{p}(w) = (w - \mathrm{P}_{\mathcal{H}}(\itr{W}{t} - \itrd{\alpha}{t+1} \itr{\Delta}{t+1}))/\itrd{\alpha}{t+1}$, which makes the analysis more tedious.
Note that $\nabla \mathsf{U}^{p}(w) = \nabla \mathsf{U}(w)$ whenever $w \in \mathrm{int}(\mathcal{H})$.
See \cite{bubeck2015convex} for an example of such analysis.
}
Recall that $\mathsf{D}(W) = \expectationWrt{\mathsf{U}(F \odot W)}{F}$.
The proof the following proposition can be found in \refAppendixSection{sec:appendix_proof_generic_rate_dropout}.

\begin{proposition}
    Let $(\itr{F}{t})_{t \in \N}$ be a sequence of independent random variables with distribution $F$.
    Let $\itr{W}{t}$ be iterates of \eqref{eqn:Update_rule_sample_complexity_dropout}.
    Assume (Q1)--(Q5).
    Define $J = S^2 + \frac{3}
        {2}N^2(\ell^2 R^2 + 2\ell R)$.
    \\
    \noindent
    \emph{(a)}
    Let $T \in \N_{+}$.
    If $p > M\ell /(NS^2T)$, then there exists a constant stepsize $\itrd{\alpha}{t} = \eta > 0$ such that for all $t \in [T]$,
    \begin{equation}
        \min_{t \in [T]} \expectationBig{\pnorm{\nabla \mathsf{D}(\itr{W}{t})}{2}^2} \leq 4\sqrt{p(S^2 + (1-p)J)}\sqrt{\frac{M \ell N}{T}}.
    \end{equation}
    \emph{(b)}
    Let $T \geq 4$.
    There exists a sequence of decreasing stepsizes satisfying $\itrd{\alpha}{t} = 1/(\ell\sqrt{t})$ for all $t \in [T]$ such that
    \begin{equation}
        \min_{t \in [T]} \expectationBig{\pnorm{\nabla \mathsf{D}(\itr{W}{t})}{2}^2} \leq \frac{4M\ell^2 + 4Np(S^2 + (1-p)J) \log(T)}{\sqrt{T}}.
    \end{equation}
    \label{prop:convergence_rate_generic_dropout}
\end{proposition}

In Proposition~\ref{prop:convergence_rate_generic_dropout}, we observe that finding approximate stationary points is easier with a larger dropout probability $1-p$ for a wide range of filter distributions like those determining \emph{dropout} and \emph{dropconnect}, as guaranteed by (Q4).
In Proposition~\ref{prop:convergence_rate_generic_dropout}(a) we also see a dependence of the convergence rate on $\sqrt{p(S^2 + (1-p)J}$.
The term $pS^2$ corresponds to the variance of the gradient due the distribution of data in $\mathcal{Z}$ and decreases with $p$; while the term $p(1-p)J$ stems from the variance due to dropout.
Note that the sum achieves a maximum for $p \in (0,1)$.
We note that Proposition~\ref{prop:convergence_rate_generic_dropout} does not suggest that the convergence to minima is faster for smaller $p$.
In particular, saddle points can become easier to find as $p \uparrow 0$.
As seen later in the numerical experiments with \glspl{NN} in \Cref{sec:numerical_experiments}, or in similar work from \cite{Mianjy2020OnCA, senencerda2020asymptotic}, the \gls{NN} structure and data distribution can change the convergence rate dependence on the dropout probability.
As an example, in \cite{senencerda2020asymptotic} it is suggested that the convergence rate dependence on $p$ and the width of the \gls{NN} can have different regimes depending on whether we are close to a minimum or not.
Similarly, smaller $p$ does not necessarily improve generalization.
In particular, if the dropout probability $1-p$ is large, the optimization landscape will be flat with many approximate stationary points.
In this case, \gls{SGD} with dropout with a limited sample complexity of $T$ iterations will not explore the landscape as much as when using a smaller dropout probability.
With a flatter landscape in mind, it may be better in the complexity trade-off to use a larger $p$ for finding an approximate minimum and generalize better instead of finding a stationary point.

A possible approach to avoid the flattening of the landscape is to scale the weights appropriately during training.
This is, for example, what is conducted in practice in some implementations of dropout.\footnote{For example, scaling is implemented with the Dropout layer implementation in \emph{Keras}, \url{https://keras.io/}.
}
Assuming (Q4) holds, we consider the update rule
\begin{equation}
    \itr{W}{t+1}
    = \itr{W}{t} - \itrd{\alpha}{t+1} \frac{\itr{F}{t}}{p} \odot \nabla r \Bigl( \itr{W}{t} \odot \frac{\itr{F}{t}}{p}, \itr{Z}{t} \Bigr).
    \label{eqn:Update_rule_scaled_dropout}
\end{equation}

With \eqref{eqn:Update_rule_scaled_dropout}, the use of filters is compensated by increasing the size of the updates and weights accordingly.
In this case, \gls{SGD} with this update rule is actually minimizing the function
\begin{equation}
    \tilde{\mathsf{D}}(W) = \mathsf{D}\Bigl(\frac{W}{p} \Bigr),
\end{equation}
which also compensates in expectation the effect of the filters.
With the update rule in \eqref{eqn:Update_rule_scaled_dropout}, we can again obtain an expression for the complexity of finding an $\epsilon$-stationary point of $\tilde{\mathsf{D}}(W)$.
The following is proved in \refAppendixSection{sec:appendix_proof_generic_rate_dropout}:

\begin{proposition}
    Let $(\itr{F}{t})_{t \in \N}$ be a sequence of independent random variables with distribution $F$.
    Assume (Q1)--(Q5).
    Let $\itr{W}{t}$ be iterates of \eqref{eqn:Update_rule_scaled_dropout}.
    Let $T \in \N_{+}$.
    If $p > M\ell /(NS^2T)$, then there exists a constant stepsize $\itrd{\alpha}{t} = \eta > 0$ such that for all $t \in [T]$,
    \begin{equation}
        \min_{t \in [T]} \expectationBig{\pnorm{\nabla \tilde{\mathsf{D}}(\itr{W}{t})}{2}^2} \leq 4\sqrt{\frac{1}{p^3}\Bigl( S^2 + \frac{(1-p)}{p^2}\Bigl(p^2 S^2 + \frac{3}
            {2}N^2\bigl( p\ell^2 R^2 + 2\ell R\bigr) \Bigr)\Bigr)} \sqrt{\frac{M \ell N}{T}}.
    \end{equation}
    \label{prop:convergence_rate_generic_dropout_scaled}
\end{proposition}

Proposition~\ref{prop:convergence_rate_generic_dropout_scaled} shows that for the scaled dropout \gls{SGD} of \eqref{eqn:Update_rule_scaled_dropout} the complexity of finding an $\epsilon$-stationary point monotonically increases with $1-p$.
This result contrasts with Proposition~\ref{prop:convergence_rate_generic_dropout}, where a different behavior was observed.
We remark, however, that this result assumes (Q5), which for small $p$ cannot realistically hold since a bound $R$ for the norm of the weights may also scale by a factor $1/p$.
This result, just like with Proposition~\ref{prop:convergence_rate_generic_dropout}, also does not imply that good weights $W \in \mathcal{W}$ become easier to find by using the update \eqref{eqn:Update_rule_scaled_dropout}.
Indeed, scaling partially avoids the flattening of the landscape---the Lipschitz constant of $\nabla \tilde{\mathsf{D}}$ is namely scaled by a factor $1/p^2$---but the variance of \gls{SGD} due to dropout is also increased considerably.
This variance becomes dominant when the dropout rate $1-p \uparrow 1$ due to the inverse dependence on $p$ in the sample complexity.

Propositions~\ref{prop:convergence_rate_generic_dropout} and \ref{prop:convergence_rate_generic_dropout_scaled} show that the complexity of finding $\epsilon$-stationary points heavily depends on the algorithm used.
However, when we restrict the results to deep \glspl{NN} such as with Example~\ref{example:NN_satisfies_conditions}, the bounds do not provide much information on the dependence of the convergence rate on the depth of the network.
This fact also shows the limitations of using a generic sample complexity analysis.

In order to obtain an explicit convergence rate depending on the depth, we need to use the additional structure of the \gls{NN}.
In the next section we will be able to compute the convergence rate to a global minimum for \glspl{NN} that are shaped like arborescences and obtain an explicit bound that depends on the depth of the arborescence and the dropout probability.

 \section{Convergence rate of \texorpdfstring{\gls{GD} on $\mathcal{D}(W)$}{GD on D(W)} for arborescences with linear activation}
\label{sec:Results__Convergence_rate_of_GD_on_D_for_arborescences}

We obtained a convergence guarantee as well as a bound for the sample complexity of dropout in the previous section.
Next, we focus on the convergence rate of dropout in functions that model the structure of \glspl{NN}.
In particular, we will derive an explicit convergence rate for dropout algorithms in the case that we have linear activations $\sigma(z) = z$ and that the \gls{NN} is structured as an arborescence: see \Cref{fig:Arborescence_picture}.
Specifically, we will study the following regular \gls{GD} algorithm on dropout's risk function:
\begin{equation}
    \itrd{W}{t+1}
    = \itrd{W}{t} - \alpha \nabla \mathcal{D}(\itrd{W}{t}) \quad \text{for} \quad t \in \N_{0}.
    \label{eqn:iterates_gradient_descent}
\end{equation}
Here, we keep the step size $\alpha > 0$ fixed.
Note that this algorithm generates a deterministic sequence $\process{ \itrd{W}{t} }{t \in \naturalNumbersZero }$ as opposed to a sequence of random variables $\process{ \itr{W}{t} }{ t \in \naturalNumbersZero }$ as generated by \eqref{eqn:Update_rule_for_Wtp1_in_terms_of_random_direction_Delta} or \eqref{eqn:Random_direction_Delta_as_function_of_Backpropagation}.
We will use a linear activation function $\sigma(t) = t$, which combined with the arborescence structure will allow us to obtain an explicit convergence rate.
While the iterates of \eqref{eqn:iterates_gradient_descent} are not stochastic, analogous to Proposition~\ref{prop:dropout_converges}, the stochastic iterates will converge to a gradient flow of an \gls{ODE}, whose discretization is given in \eqref{eqn:iterates_gradient_descent}.
Analyzing \glspl{ODE} related to \glspl{NN} is common in literature \cite{tarmoun2021understanding,jacot2018neural}.
For more discussion on the relationship between the iterates of \eqref{eqn:iterates_gradient_descent} and dropout we refer to \refAppendixSection{sec:appendix_ode_method}.

Our main convergence result in Proposition~\ref{prop:convergence_of_dropout_on_an_expanding_tree} below holds for general distribution functions.
However, we show here the cases of \emph{Dropout} and \emph{Dropconnect}, which are most insightful.
We use the following notation adapted from graph theory.
Consider a fixed, directed \emph{base graph} $G = ( \mathcal{E}, \mathcal{V} )$ without cycles in which all paths have length $L$, which describes a \gls{NN}'s structure as follows.
Each vertex $v \in \mathcal{V}$ represents a neuron of the \gls{NN}, and each directed edge $e = (u,v) \in \mathcal{E}$ indicates that neuron $u$'s output is input to neuron $v$.
Note that to each edge $e \in \mathcal{E}$ in the \gls{NN}, a weight $W_e \in \realNumbers$ and a filter variable $F_e \in \{ 0, 1 \}$ are associated.
We will write $\mathcal{W} = \R^{|\mathcal{E}|}$ for simplicity.
For an arborescence $G$, we denote by $\mathcal{L}(G)$ the edge set of leaves.
\putin{Let $M > 2\delta > 0$ be real numbers and suppose that we initialize the weights $\{ W_e\}_{e\in \mathcal{E}}$ as follows:
\begin{align}
    M > \itrd{W_e}{0} > \sqrt{2} \delta            & \text{ for } e \in \mathcal{E}\backslash \mathcal{L}(G) \nonumber \\
    |W_l| \leq \delta/ \sqrt{\abs{\mathcal{L}(G)}} & \text{ for } l \in \mathcal{L}(G).
    \label{eqn:unbalanced_initialization}
\end{align}
}

The proof of Proposition~\ref{prop:estimate_nu_dependent_p} is deferred to \refAppendixSection{sec:Appendix__Corollary}, which is a consequence of our more general result in Proposition~\ref{prop:convergence_of_dropout_on_an_expanding_tree}.

\begin{proposition}
    \label{prop:estimate_nu_dependent_p}
    Assume that the base graph $G$ is an arborescence of depth $L$ with $\cardinality{\mathcal{L}(G)}$ leaves, the activation function $\sigma(t) = t$ is linear, $F$ is independent of $(X,Y)$, and $\{ \itrd{W}{0}_e \}_{e \in \mathcal{E}}$ is initialized according to \eqref{eqn:unbalanced_initialization}.
    If the $\{ F_e \}_{e \in \mathcal{E}}$ follow the distribution prescribed by \emph{Dropconnect} or \emph{Dropout}, then there exists $\alpha > 0$ such that the iterates of \eqref{eqn:iterates_gradient_descent} satisfy
    \begin{equation}
        \mathcal{D}(\itrd{W}{t}) - \mathcal{D}(\criticalpoint{W})
        \leq
        \bigl( \mathcal{D}(\itrd{W}{0}) - \mathcal{D}(\criticalpoint{W}) \bigr) \exp(- \omega t/2)
        .
        \label{eqn:PL_inequality_upper_bounded_and_iterated}
    \end{equation}
    with
    \begin{equation}
        \omega
        = \mathrm{O} \Bigl( \frac{p^L}{L \cardinality{\mathcal{L}(G)}^2} \Bigl(\frac{2\delta^2}{M^{2}}\Bigr)^{2L}\Bigr)
        .
    \end{equation}
\end{proposition}

\subsection{Discussion}

In Proposition~\ref{prop:estimate_nu_dependent_p} we consider the cases of \emph{Dropout} and \emph{Dropconnect}, in which nodes or edges are dropped with probability $1-p$, respectively.
Observe that the convergence rate exponent depends on $p^L$ and $(2\delta^2/M^2)^{2L}$ where $2\delta^2/M^2 < 1$; see \eqref{eqn:unbalanced_initialization}.
The first term in particular indicates that as the \gls{NN} becomes deeper, the convergence rate exponent of \gls{GD} with \emph{Dropout} or \emph{Dropconnect} will decrease by a factor $p^L$.
The second term $(2\delta^2/M^2)^{2L}$ shows the increased difficulty of training deeper \glspl{NN} and has been observed e.g., by \cite{shamir2018exponential,arora2018convergence}.
The exponential dependence in $L$ is moreover tight when using \gls{GD} and is intrinsic to the method \citep{shamir2018exponential}.
\putin{Hence, dropout adds another exponential dependence to the convergence rate in arborescences, which is due to the stochastic nature of the algorithm.
    In \Cref{fig:path_NN} an experiment confirming this intuition on the convergence rate of dropout on a single path for different depths can be seen.
}

Finally, our proofs of Proposition~\ref{prop:estimate_nu_dependent_p} and the related more general result in Proposition~\ref{prop:convergence_of_dropout_on_an_expanding_tree} below can be found in \refAppendixSection{sec:Appendix__Convergence_of_GD}.
The proof strategy is to show that a \gls{PL} inequality holds, which allows one to obtain convergence rates for \gls{GD} on nonconvex functions \citep{karimi2016linear}.
The new part of the argument is that we use conserved quantities and a double induction to identify a compact set in which the iterates remain and simultaneously a \gls{PL} inequality holds.
The method that we develop and which is sketched in the next subsection depends intricately on the arborescence structure and cannot be readily applied to other cases.

To compare this result with more realistic models, we will examine the convergence rate of dropout in deep and wide \glspl{NN} in \Cref{sec:convergence_rate_wider_NN} with a heuristic and experimental approach.

\begin{figure*}
\centering
    \begin{subfigure}{0.45\textwidth}
        \includegraphics[height=0.85\textwidth]{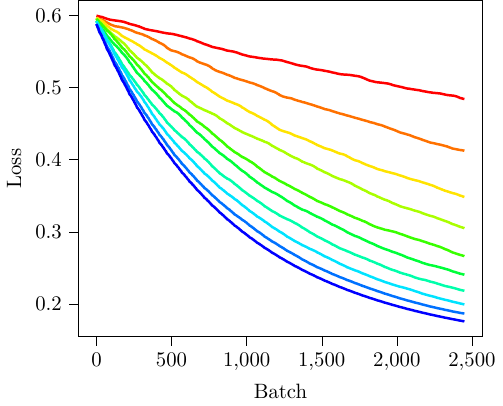}
        \caption{}
    \end{subfigure}
    \hspace{1em}
    \begin{subfigure}{0.45\textwidth}
        \includegraphics[height=0.85\textwidth]{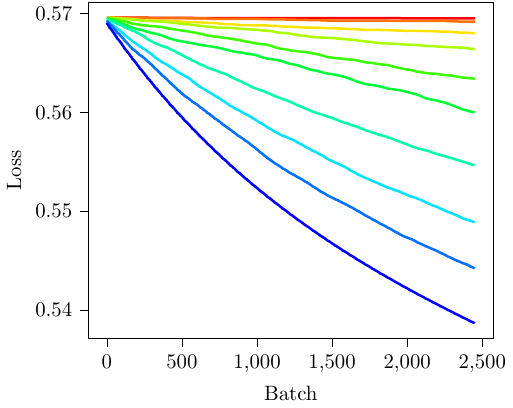}
        \caption{}
    \end{subfigure}
    ~\vspace{1em}

    \noindent
    \begin{subfigure}{0.45\textwidth}
        \includegraphics[height=0.85\textwidth]{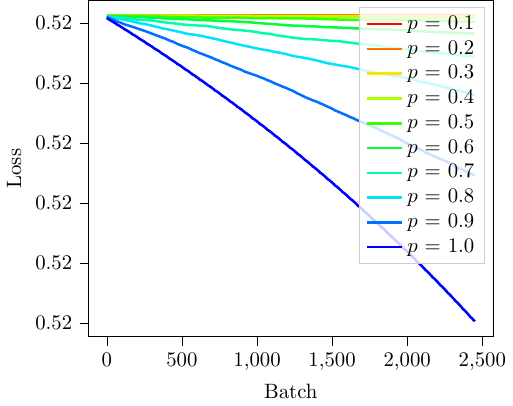}
        \caption{}
    \end{subfigure}
    \hspace{1em}
    \begin{subfigure}{0.45\textwidth}
        \includegraphics[height=0.85\textwidth]{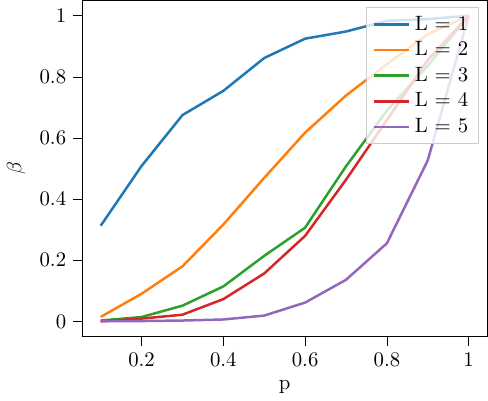}
        \caption{}
    \end{subfigure}
    \caption{The average loss depending on the number of steps of \gls{SGD} with dropout of the function $f(w) = (y - \prod_{i=1}^L w_i x)^2$ and its average convergence slope. \emph{(a)}
        The average loss for $L=1$.
        \emph{(b)}
        The average loss for $L=3$.
        \emph{(c)}
        The average loss for $L=5$.
        \emph{(d)}
        The slope $\beta$ of the fit of $y=-\beta x + \gamma$ for the curves in (a), (b) and (c).
        The slopes $\beta$ for a given $l$ have been normalized at $p=1$ for comparison across depths $L$.
        Note that for larger $L$, the effect of $p$ becomes also more pronounced.
        This is in agreement with the conclusion in \Cref{sec:Results__Convergence_rate_of_GD_on_D_for_arborescences}, where we expect a convergence rate depending on $p^{L}$.
        In this case, other effects of depth are also observed, such as a dependence on the initialization.
    }	\label{fig:path_NN}
\end{figure*}

\subsection{Sketch of the proof}

Besides the previous notation, we need to introduce notation corresponding to subgraphs and paths.
Let $\mathcal{G}$ be the set of all subgraphs of the base layered directed graph $G$ with $d$ vertices, and let $\mathcal{E}(g)$ be the set of edges of a subgraph $g \in \mathcal{G}$.
Let $\Gamma_i^j(g;e)$ be defined as the set of all paths in the directed graph $g$ that start at vertex $i$, traverse edge $e$, and end at vertex $j$.
If the origin or end vertices are in the input or output layer, the subscript or superscript is dropped from the notation, respectively.
For every path $\gamma \triangleq ( \gamma_1, \ldots, \gamma_L ) \in \Gamma(g)$, we write $P_\gamma \triangleq \prod_{e \in \gamma} W_e$ and $F_\gamma \triangleq \prod_{e \in \gamma} F_e$ for notational convenience.
Finally, let $G_F \triangleq ( \mathcal{E}_F, \mathcal{V})$ be the random subgraph of base graph $G$ that has edge set $\mathcal{E}_F \triangleq \{ e \in \mathcal{E} | F_e = 1 \}$.
We denote $\mu_g \triangleq \probability{G_F = g}$, and $\eta_\gamma \triangleq \sum_{ \{ g \in \mathcal{G} | \gamma \in \Gamma(g) \} } \mu_g$.
We first provide an explicit characterization of dropout's risk function in \eqref{eqn:Dropouts_emperical_risk_function} in terms of paths in the graph that describes the structure of the \gls{NN}.
This is possible since we assume linear activation functions.
The following lemma now holds, and is proved in \refAppendixSection{sec:Appendix__Proof_of_the_path_representation_of_DW__Tree_case}.

\begin{lemma}
    \label{lem:Path_representation_of_DW}
    Assume that the base graph $G$ is a fixed, directed graph without cycles in which all paths have length $L$ and there are $d_{L}$ output nodes (N6'), that $\sigma(t) = t$ (N7), and that $F$ is independent of $(X,Y)$ (N8).
    Then
    \begin{equation}
        \mathcal{D}(W)
        = \sum_{ g \in \mathcal{G} } \mu_{g} \expectationBig{ \sum_{s=1}^{d_L} \bigl( Y_s - \sum_{ \gamma \in \Gamma^s(g) } P_\gamma X_{\gamma_0}\bigr)^2}.
        \label{eqn:Path_representation_of_DW__Tree_case}
    \end{equation}
    Moreover $\mathcal{D}(W) = \mathcal{J}(W) + R(W)$, where
    \begin{align}
        \mathcal{J}(W)
         &
        = \sum_{\gamma \in \Gamma(G)} \eta_{\gamma} \expectation{ ( Y_{\gamma_L} - P_\gamma X_{\gamma_0} )^2 },
        \label{eqn:Regularization_term_in_terms_of_paths}
        \\
        R(W)
         &
        = - \sum_{ g \in \mathcal{G} } \mu_g \expectationBig{ \sum_{s=1}^{d_L} \sum_{ \gamma \in \Gamma^s(g) } \Bigl( \Bigl( 1 - \frac{1}{ \cardinality{ \Gamma^{s}(g) } } \Bigr) Y_{s}^2
            - P_\gamma X_{\gamma_0} \sum_{ \delta \in \Gamma^{s}(g) \backslash \{ \gamma \} } P_\delta X_{\delta_0} \Bigr)}.
    \end{align}
    Here, the constants $\eta_{\gamma}, \mu_{\gamma}$ depend explicitly on $F$'s distribution and the \gls{NN}'s architecture.
\end{lemma}

Note that Lemma~\ref{lem:Path_representation_of_DW} essentially changes variables to rewrite the dropout risk function as a sum over paths instead of a sum over graphs.
This representation allows us to clearly identify the regularization term $R(W)$.
For example in the case of \emph{Dropconnect} \citep{wan2013regularization}, where the filter variables $\{ F_e \}_{e \in \mathcal{E}}$ are independent random variables with distribution $\mathrm{Bernoulli}(p)$, Lemma~\ref{lem:Path_representation_of_DW} holds with $\mu_g = p^{ \cardinality{ \mathcal{E}(g) } } ( 1 - p )^{ \cardinality{ \mathcal{E}(G) } - \cardinality{ \mathcal{E}(g) } }$.
Also note that if for all subgraphs $g \in \mathcal{G}$ and vertices $i \in [d]$ the number of paths that end at $i$ satisfies $\cardinality{ \Gamma^i(g) } = 1$ , such as when $G$ is an arborescence, then for all subgraphs $g \in \mathcal{G}$ and paths $\gamma \in \Gamma(g)$ there is only one path ending at a leave node $\gamma_L$, that is, $\Gamma^{\gamma_L}(g) = \{ \gamma \}$.

We now focus on a base graph that is an arborescence of arbitrary depth; see \Cref{fig:Arborescence_picture}.
Hence we now replace (N6') in Lemma~\ref{lem:Path_representation_of_DW} that assumes a generic graph by assumption (N6), where $G$ is specifically an arborescence.
The following specification of Corollary~\ref{cor:Path_representation_of_DW__Tree_case_arborescence} is also proven in \refAppendixSection{sec:Appendix__Proof_of_the_path_representation_of_DW__Tree_case}.

\begin{corollary}
    \label{cor:Path_representation_of_DW__Tree_case_arborescence}
    Assume that the base graph $G$ is an arborescence of depth $L$ (N6), and (N7)--(N8) from Lemma~\ref{lem:Path_representation_of_DW}.
    Then $\mathcal{D}(W) = \mathcal{I}(W) + \mathcal{D}(\criticalpoint{W})$, where
    \begin{gather}
        \mathcal{I}(W)
        \triangleq
        \sum_{\gamma \in \Gamma(G)} \nu_{\gamma} (z_{\gamma} - P_{\gamma})^2
        ,
        \nonumber \\
        \mathcal{D}(\criticalpoint{W})
        =
        \sum_{\gamma \in \Gamma(G)} \eta_{\gamma} ( \expectation{Y_{\gamma_L}^2} - {\expectation{Y_{\gamma_L} X_{\gamma_0}}^2} / {\expectation{X_{\gamma_0}^2}} )
        ,
    \end{gather}
    and
    $
        \nu_{\gamma} \triangleq \eta_{\gamma} \expectation{X_{\gamma_0}^2}
    $,
    $
        z_{\gamma} \triangleq {\expectation{Y_{\gamma_L} X_{\gamma_0}}} / {\expectation{X_{\gamma_0}^2}}
    $ for $\gamma \in \Gamma(G)$.
    Consequently, $R(W) = 0$ for an arborescence.
\end{corollary}

The convergence result we are about to show uses the fact that for the system of \glspl{ODE} $\d{W} / \d{t} = - \nabla_W \mathcal{D}(W)$ there are conserved quantities.
Within the proof, these conserved quantities have the crucial role of guaranteeing compactness for the iterates.
Specifically, let $\mathcal{L}(g;f)$ denote the leaves of the subtree of $g \in \mathcal{G}$ rooted at a vertex $f \in \mathcal{E}(g)$, and define the set of leaves of $G$ as $\mathcal{L}(G) \triangleq \cup_{f \in \mathcal{E}} \mathcal{L}(G;f)$.
We remark that in the previous notation $d_L = |\mathcal{L}(G)|$.
For $W \in \mathcal{W}$ and each leaf $f \in \mathcal{E} \backslash \mathcal{L}(G)$, define the quantity
\begin{equation}
    C_f
    =
    C_f(W)
    \triangleq
    W_f^2 - \sum_{ l \in \mathcal{L}(G;f) } W_l^2
    .
    \label{eqn:definition_first_integrals}
\end{equation}
Define $C_{\min}
    \triangleq \min_{e \in \mathcal{E} \backslash \mathcal{L}(G)} C_e$ and $\itrd{C_e}{t} = C_e(\itrd{W}{t})$ for $t \in \naturalNumbersPlus$ also, both of which we require later.
Lemma~\ref{lem:Conserved_quantities_in_a_tree} now proves that the function $C_f$ in \eqref{eqn:definition_first_integrals} is a conserved quantity; the proof is in  \refAppendixSection{sec:Appendix__Proof_of_Conserved_quantities_in_a_tree}.

\begin{lemma}
    \label{lem:Conserved_quantities_in_a_tree}
    Assume (N2) from Proposition~\ref{prop:dropout_converges}, (N6) from \refCorollary{cor:Path_representation_of_DW__Tree_case_arborescence} , (N7), (N8) from Lemma~\ref{lem:Path_representation_of_DW}.
    Then under the negative gradient flow $\d{W}/\d{t} = - \nabla \mathcal{D}(W)$,
    \begin{equation}
        \frac{ \d{C_f} }{ \d{t} }
        = 0
    \end{equation}
    for all $f \in \mathcal{E} \backslash \mathcal{L}(G)$.
\end{lemma}

We are almost in position to state our second result, but need to introduce still some notation.
We define the following constants
\begin{equation}
    \pnorm{\nu}{1}
    \triangleq
    \sum_{\gamma \in \Gamma(G)} \nu_{\gamma}
    ,
    \quad
    \nu_{\min}
    \triangleq
    \min_{\gamma \in \Gamma(G)} \nu_{\gamma},
    \quad
    \nu_{\max}
    \triangleq
    \max_{\gamma \in \Gamma(G)} \nu_{\gamma},
\end{equation}
for notational convenience.
Also, for $0 < \delta < M$, we define
\begin{equation}
    \mathcal{S}
    \triangleq \{ W \in \mathcal{W} ~ : ~
    M > \abs{W_f} > \delta > 0 \, \ \forall f \in \mathcal{E}(G) \backslash \mathcal{L}(G); \allowbreak
    M > \abs{W_f} \, \forall f \in \mathcal{L}(G)
    \}
    ,
\end{equation}
a bounded set of parameters where if the weight is associated with a leaf, they are furthermore bounded away from zero.
Let finally
\begin{align}
    B(\epsilon, I)
    \triangleq
    \Bigl\{ &
    W \in \mathcal{W} ~ : ~
    \mathcal{I}(W) \leq \epsilon,
    W_f^2 - \sum_{l \in \mathcal{L}(G;f)} W_{l}^2 \in I_f \text{ for } f \in \mathcal{E} \backslash \mathcal{L}(G)
    \Bigr\}
    \label{eqn:definition_compact_set_wrt_I}
\end{align}
denote the set of all weight parameters that are $\eps$-close to a critical point and for which the conserved quantities in \eqref{eqn:definition_first_integrals} deviate by no more than $\bigO{ \itrd{C_f}{0} }$ from their initial value $\itrd{C_f}{0}$.
These deviations are made explicit by the intervals
\begin{equation}
    I_f
    \triangleq [\itrd{C_{f}}{0}/2, 3\itrd{C_{f}}{0}/2] \textnormal{ for }
    f \in \mathcal{E} \backslash \mathcal{L}(G)
    ,
    \textnormal{ and the set }
    I
    \triangleq \times_{f \in \mathcal{E} \backslash \mathcal{L}(G)} I_f \subseteq \R^{\cardinality{\mathcal{E}} - \cardinality{\mathcal{L}(G)}}
    .
\end{equation}
Our proof shows that the iterates $\process{ \itrd{W}{t} }{t \geq 0}$ stay in the intersection $\mathcal{S} \cap B(\eps,I)$, and this implies that the weights (including those associated with the leaves) remain bounded.
The following now holds, and its proof can be found in \refAppendixSection{sec:Appendix__Convergence_of_GD}.

\begin{proposition}
    \label{prop:convergence_of_dropout_on_an_expanding_tree}
    Assume (N2) from Proposition~\ref{prop:dropout_converges}, (N6) from \refCorollary{cor:Path_representation_of_DW__Tree_case_arborescence}, (N7)--(N8) from Lemma~\ref{lem:Path_representation_of_DW}, that $\itrd{W}{0} \in \mathcal{S} \cap B(\epsilon, I)$ and $M^{L} \geq \abs{z_{\gamma}}$ for all $\gamma \in \Gamma(G)$ (N9), that $\frac{1}{2} C_{\min}(\itrd{W}{0}) > \delta^2$ (N10).
    If
    \begin{align}
        \alpha
        \leq
        \min \Bigl( \nu_{\min} \frac{ \e{1/2}(\itrd{C_{\min}}{0})^{L} }{16\norm{\nu}_1 L M^{2(L-1)} \mathcal{I}( \itrd{W}{0})}, \frac{1}{12 \nu_{\max} \abs{\mathcal{E}} \abs{\Gamma(G)} M^{2(L-1)}}, \frac{1}{2 \nu_{\min} (\itrd{C_{\min}}{0})^{L-1}} \Bigr)
        ,
        \label{eqn:Convergence_of_Dropout_tree_step_size}
    \end{align}
    then the iterates of \eqref{eqn:iterates_gradient_descent} satisfy
    \begin{equation}
        \mathcal{D}(\itrd{W}{t}) - \mathcal{D}(\criticalpoint{W})
        \leq
        \bigl( \mathcal{D}(\itrd{W}{0}) - \mathcal{D}(\criticalpoint{W}) \bigr) \exp(- \tfrac{ \alpha \tau }{2} t)
        .
        \label{eqn:PL_inequality_upper_bounded_and_iterated_full}
    \end{equation}
    where $\tau = 4 \nu_{\min} \exp(-1/2) (\itrd{C_{\min}}{0})^{L-1}$.
\end{proposition}

Proposition~\ref{prop:convergence_of_dropout_on_an_expanding_tree} identifies explicitly how the convergence rate of \gls{GD} on a dropout's risk function depends on the dropout algorithm and the structure of the arborescence: parameters such as $p, |\mathcal{L}(G)|, L$ are implicitly present in the constants $\nu_{\min}$ and $\pnorm{\nu}{1}$ in $\alpha, \tau$.

Note that Assumptions (N9)--(N10) are relatively benign.
These assumptions are for example satisfied when initializing $M > \itrd{W_e}{0} > \sqrt{2} \delta$ for $e \in \mathcal{E}\backslash \mathcal{L}(G)$ and setting $|W_l| \leq \delta/ \sqrt{\abs{\mathcal{L}(G)}}$ for all $l \in \mathcal{L}(G)$ and $\epsilon = \mathcal{I}(\itrd{W}{0})$, which we assume in Proposition~\ref{prop:estimate_nu_dependent_p}.
In other words, this initialization sets the weights that are associated with leaves small compared to all other weights.

\section{Effect of dropout on the convergence rate in wider networks}
\label{sec:convergence_rate_wider_NN}
In Proposition~\ref{prop:convergence_of_dropout_on_an_expanding_tree}, we have proven that the convergence rate depends on $p^L$ for \glspl{NN} shaped like arborescences.
Let $G_{\mathrm{tree}}$ be a tree and $e \in \mathcal{E}(G_{\mathrm{tree}})$ be an edge.
Denote by $\itr{\Gamma}{t}(e)$ the set of paths passing through $e$ that are not filtered by dropout at time $t$.
We observe that at any given time $t$ of dropout \gls{SGD},
\begin{equation}
    \probability{\itr{w}{t}_e \text{ is updated}} =
    \probability{\itr{\Gamma}{t}(e) \neq \emptyset} = p^{L}.
\end{equation}

If we denote by $t_{\mathrm{update}}(G_{\mathrm{tree}}) = 1/p^{L}$ the average update time for a weight in $G_{\mathrm{tree}}$, then we need $1/p^{L}$ more time on average for a given edge to be updated than when we do not use dropout.
For wider networks $G$, however, edges can be updated simultaneously and repeatedly via different available paths.
By the previous intuition we might still expect that, if the updates are sufficiently independent, the convergence rate depends approximately on $1/t_{\mathrm{update}}$.
In order to verify this intuition we will determine $t_{\mathrm{update}}$ for \glspl{NN} that are much wider than deep, and later simulate their convergence rates also in realistic settings.

Suppose now that $G$ is a graph of a fully-connected \gls{NN} with $L$ dropout layers each of which has width $D$.
For each of the vertices $u \in G$ in a dropout layer, there is an associated dropout filter variable $F_{u} \sim_{\mathrm{i.i.d.
        }} \mathrm{Ber}(p)$  where $p > 0$ is fixed.
That is, we use \emph{dropout}.
Note that any other additional input or output layer without filters only changes the number of paths by a multiplicative factor.
Hence, we will restrict to the case that all nodes in the layers have filter variables.
In this case, we may consider a path $\gamma = (u_1, \ldots, u_{L})$ as a set of $L$ vertices---one for each dropout layer---instead of edges.
For two paths $\gamma$ and $\delta$, we consider their intersection $\gamma \cap \delta$ as the subset of vertices belonging to both paths.
Hence, $|\gamma \cap \delta| = l$ implies that the intersection has $l$ vertices, not necessarily forming a path.

We remark that we can restrict to the case $L > 2$.
In the case of one dropout layer $L=1$, an edge $e = (u,v)$ conected to a dropout node $u$ is updated if and only if the filter $F_u = 1$, where $u \in G$ is the adjacent vertex to $e$ with a dropout filter, so that in this case $\probability{\itr{w}{t}_e \text{ is updated}} = 1-p$.
For $L=2$, an edge $e=(u,v)$ is updated if and only if $F_u = F_v = 1$, so that $\probability{\itr{w}{t}_e \text{ is updated}} = 1-p^2$.
Recall that we denote by $\Gamma(e)$ the set of paths $\gamma$ of $G$ passing through $e$.
For a path $\gamma \in \Gamma(e)$, in the following, we let $F_{\gamma} = \prod_{u \in \gamma} F_{u}$ be the indicator of a path being filtered.
Thus, $F_{\gamma}$ is $1$ is $\gamma$ is not filtered and $0$ otherwise.
We will use Greek letters for paths and Latin letters for vertices when referring to filters $F_{\gamma}$ and $F_{u}$ respectively.

\begin{lemma}
    Let $G$ be a graph of a fully-connected \gls{NN} with $L > 2$ dropout layers, each with the same width $D$ and with dropout filters $F_u$ for $u \in G$.
    For an edge $e \in \mathcal{E}(G)$, let $F_{\Gamma(e)} = \sum_{\gamma \in \Gamma(e)} F_{\gamma}$ denote the random variable that counts the number of nonfiltered traversing paths through $e$.
    If $L, p$ are fixed, then as $D \to \infty$,
    \begin{equation}
        \probability{F_{\Gamma(e)} = 0} = 1-p^2 + O\Bigl(\frac{pL}{D}
        \Bigr).
    \end{equation}
    \label{lem:probability_zero_paths}
\end{lemma}
\begin{proof}
    We will use the Paley--Zygmund inequality.
    For a nonnegative random variable $Z$ with finite second moment, for any $\theta \in (0,1)$,
    \begin{equation}
        \probability{Z > \theta \E[Z]} \geq (1- \theta)^2 \frac{\expectation{Z}^2}{\expectation{Z^2}}.
        \label{eqn:paley_zygmund_inequality}
    \end{equation}

    We will use \eqref{eqn:paley_zygmund_inequality} with the random variable $F_{\Gamma(e)}$.
    The idea is that if $D$ is much larger than $L$, the average number of paths passing through $e$ is also large.
    We are using dropout, so the filter variable corresponding to an edge $e=(u,v)$ will depend on only the vertex $u$, that is, $F_e = F_u$.
    For counting paths we also need to take into account that the filter $F_v$ will occurring in all paths passing through $e$.
    Since only the two vertices $u$ and $v$ of $e$ are fixed we can compute
    \begin{equation}
        \expectation{F_{\Gamma(e)}} = \sum_{\gamma \in \Gamma(e)} \expectation{F_{\gamma}} = p^{L}|\Gamma(e)| = p^{L}D^{L-2}.
        \label{eqn:first_moment}
    \end{equation}
    We define the set of broken paths in $\Gamma(e)$ as
    \begin{equation}
        \Gamma_b(e) = \{ \gamma = (u_{i_1}, \ldots, u_{i_k}) \in G^{k} : \exists \eta, \delta \in \Gamma(e), \gamma = \eta \cap \delta \},
    \end{equation}
    that is, $\gamma \in \Gamma_b(e)$ if and only if there exist $\eta, \delta \in \Gamma(e)$ such that $\gamma = \eta \cap \delta$.
    In particular, $\Gamma_b(e)$ contains paths and unions of vertices of paths that pass through $e$.
    Then we have:
    \begin{align}
        \expectation{F_{\Gamma, e}^2} & = \sum_{\gamma \in \Gamma(e)} \sum_{\delta \in \Gamma(e)} \expectation{F_{\gamma} F_{\delta}} \eqcom{i} = \sum_{\gamma \in \Gamma(e)} \sum_{l=2}^{L}\sum_{\substack{\delta \in \Gamma(e) \\ |\gamma \cap \delta| = l}} \probability{F_{\gamma} = 1, F_{\delta} = 1}\\
                                      & \eqcom{ii}= \sum_{\gamma \in \Gamma(e)} \sum_{l=2}^{L}\sum_{\substack{\delta \in \Gamma(e)                                                                                               \\ |\gamma \cap \delta| = l}} p^{l} p^{2L - 2l} \eqcom{iii}= \sum_{l=2}^{L} \sum_{\substack{\eta \in \Gamma_b(e)\\ |\eta| = l}} \sum_{\substack{\gamma, \delta \in \Gamma(e)\\ \eta \subseteq \delta, \gamma\\ \gamma \cap \delta = \eta}} p^{l} p^{2L - 2l} \\
                                      & \eqcom{iv}= \sum_{l=2}^{L} \sum_{\substack{\eta \in \Gamma_b(e)                                                                                                                          \\ |\eta| = l}} (D(D-1))^{L-l} p^{l} p^{2L - 2l} \label{eqn:recover_path_result}\\
                                      & \eqcom{v}= \sum_{l=2}^{L} \binom{L-2}{l-2} D^{l-2} (D(D-1))^{L-l} p^{l} p^{2L - 2l}                                                                                                      \\
                                      & = p^{2L-2} D^{2L-4} + O(Lp^{2L-3}D^{2L - 5}),
        \label{eqn:second_moment}
    \end{align}
    where (i) we have first used that $F_{\gamma}$ are indicators for occurring $\gamma \in \Gamma(e)$ and that at least $l \geq 2$ since vertices $u$ and $v$ are shared among all paths in $\Gamma(e)$; secondly, that we have separated the sum over paths into a path $\gamma$ and all other paths $\delta$ that coincide in $l$ vertices.
    In (ii) we have computed the probability by noting that for $\gamma$ and $\delta$ such that $|\gamma \cap \delta| = l \geq 2$, $\expectation{F_{\gamma}F_{\delta}} = p^{l} p^{2L - 2l}$, where the term $p^l$ accounts for the $l$ shared filters corresponding to $l$ shared vertices and $p^{2L - 2l}$ for the remaining products of filters.
    Note that we have used the independence assumption for filters here.
    (iii) We have used here that $\eta = \delta \cap \gamma \in \Gamma_b(e)$, so that we can separate the previous sum into first, fixing the $l$ vertices where two paths intersect---including $e$---with $\eta \in \Gamma_b(e)$ such that $|\eta|=l$, and then looking for all possible $\delta, \gamma \in \Gamma(e)$ such that $\gamma \cap \delta  = \eta$.
    For (iv) we fix $l$ vertices where $\gamma$ and $\delta$ coincide, then there are still $(D(D-1))^{L-l}$ possible ordered vertex pairs to choose from all the other vertices where $\gamma$ and $\delta$ do not coincide.
    (v) For the remaining sum, for each $l$ fixed locations---including the vertices of $e$, which are fixed---we can still choose $D^{l-2}$ remaining possible vertices.
    Additionally, there are for each $l$, $\binom{L-2}{l-2}$ distinct $l-2$ locations for these vertices.
    Hence, plugging \eqref{eqn:second_moment} and \eqref{eqn:first_moment} into \eqref{eqn:paley_zygmund_inequality} yields
    \begin{align}
        \probability{F_{\Gamma(e)} > \theta p^{L}D^{L-2}} & \geq (1-\theta)^2 \frac{p^{2L}D^{2L-4}}{p^{2L-2} D^{2L-4} + O(Lp^{2L-3}D^{2L - 5}))} \\
                                                          & = (1-\theta)^2 \frac{p^2}{1 + O(L/(Dp))}                                             \\
                                                          & = (1-\theta)^2(p^2 + O(pL/D)).
    \end{align}

    In particular, setting $\theta^{-1} = 2p^{L}D^{L-2}$ and computing the higher order noting that $L > 2$, we obtain that
    \begin{equation}
        \probability{F_{\Gamma(e)} > 1/2} \geq p^2 + O(pL/D),
    \end{equation}
    or alternatively noting that $\{F_{\Gamma(e)} \leq 1/2\} = \{F_{\Gamma(e)} = 0\}$, since $F_{\Gamma(e)} \in \N$ we obtain
    \begin{equation}
        \probability{F_{\Gamma(e)} = 0} \leq 1- p^2 + O(pL/D).
    \end{equation}
    Finally note that $1-p^2 \leq \probability{F_{\Gamma(e)} = 0}$ since the edge $e$ can be present in a path only if the filters at both vertices of $e$ have value $1$, which occurs with probability $p^2$, so that $\probability{F_{\Gamma(e)} > 0} < p^2$.
\end{proof}

Note that in the proof of Lemma~\ref{lem:probability_zero_paths} we can recover the scaling $p^L$ that we have seen in Proposition~\ref{prop:convergence_of_dropout_on_an_expanding_tree} by setting $D=1$ in \eqref{eqn:second_moment} and in \eqref{eqn:recover_path_result}.

From Lemma~\ref{lem:probability_zero_paths} we expect that for a wide network with $L$ layers where $D \gg L$ and an edge $e \in \mathcal{E}(G)$, we have that
\begin{equation}
    \probability{\itr{w}{t}_e \text{ is updated}} = p^2 + O(pL/D).
    \label{eqn:expected_update_rate_wide_network}
\end{equation}

If the convergence rate is related to the update rule, then we would expect that for a wide network the rate would be independent of $L$ which is different from the path network considered in \refProposition {prop:convergence_of_dropout_on_an_expanding_tree}.
In the next section we will verify this intuition on real datasets.
Note, however, that we do not expect to see the dependence on $p$ as shown in \eqref{eqn:expected_update_rate_wide_network}: this heuristic argument provides only the rate at which a weight is updated, and stochastic averaging is not solely driving the convergence rate.
In particular, from an example for wide shallow linear networks in \cite{senencerda2020asymptotic}, close to a critical point of a dropout \gls{ODE}, the dependence scales with a factor $p(1-p)$ instead of $p$.
This is due to the fact that for larger $p$, there are regions of the landscape close to minima that become flat, as also hinted by Proposition~\ref{prop:convergence_rate_generic_dropout}.
Indeed, when $p \uparrow 1$ the term $(1-p)J \downarrow 0$ in the convergence rate of Proposition~\ref{prop:convergence_rate_generic_dropout} lowers the complexity of finding an $\epsilon$-stationary point.
Hence, there are landscape regimes and initialization issues that also account for the convergence rate in \glspl{NN}.

\subsection{Numerical Experiments}
\label{sec:numerical_experiments}

In this section we conduct the dropout stochastic gradient descent algorithm numerically,\footnote{The source code of our implementation is available at \url{https://gitlab.tue.nl/20194488/almost-sure-convegence-of-dropout-algorithms-for-neural-networks}.
} for different datasets and network architectures.
We measure the convergence rate for different widths $D$, depths $L$, and dropout probabilities $1-p$.
We then compare these measurements to the bounds on the convergence rates obtained in \Cref{sec:Results__Convergence_rate_of_GD_on_D_for_arborescences}.
We use \emph{Tensorflow}\footnote{\url{https://www.tensorflow.org/}} for the implementation.

\subsubsection{Setup}

\noindent
\emph{Datasets.}
We will consider three commonly used data sets of images: the \acrshort{MNIST}\footnote{\gls{MNIST}} \citep{lecun2010mnist}, \acrshort{CIFAR}-100-fine\footnote{ \gls{CIFAR}}, and \acrshort{CIFAR}-100-coarse datasets \citep{krizhevsky2009learning}.

\noindent
\emph{\gls{NN}
    Architecture.
}
We use as a base architecture a LeNet with 11 layers where the two dense layers have been substituted with $L$ fully-connected \gls{ReLU} layers of width $D$.
Each of these layers have dropout with dropout probability $1-p$.
While larger networks are commonly used in practice, a LeNet architecture is sufficient to test the effect of dropout on the convergence rate as we verify with the simulations.

\noindent
\emph{Loss.}
We use the cross-entropy loss, which is commonly used for classification.
For two distributions $p$ and $q$ with support on $[n]$ labels, the cross-entropy loss is defined as
\begin{equation}
    l(p,q) = -\sum_{i=1}^{n} q_i \log(p_i).
    \label{eqn:cross_entropy_loss}
\end{equation}

\noindent
\emph{Stopping criteria.}
In all experiments, we stop after $40$ epochs.

\noindent
\emph{Initialization.}
In order to see the convergence rate close to a minimum.
We use first a \emph{Gaussian initialization}, that is, we set every weight on the dense layers to $W_{ijk} \sim \mathrm{Normal}(0,1/\sqrt{D})$ in an independent manner, where $D$ is the width of the layer.
While this initialization is standard, we note that we cannot expect to compare convergence rates for different numbers of layers $L \in\{1,2,3\}$ and for different dropout probabilities $1-p$, since the loss functions are also different.
In the course of our experiments, we found that there are also many saddle points where \gls{SGD} remains stuck, which complicated the estimation of the convergence rate.
In order to start approximately at the same neighborhood where the iterates stay and continuously track minima across different choices of $p$, for each $L \in \{1,2,3\}$ we have used a two-step approach in order to avoid areas of the landscape with saddle points.
We first run \textrm{ADAM}\footnote{Adaptative Moment Estimation (See \cite{kingma2014Adam}).
} for $2$ epochs with $p =0.1$ and store the weights.
Secondly, for each $p
    \in P$ we then perform dropout \gls{SGD} with initialization given by the stored weights.
In this manner, we expect that we are approximately ``tracking'' the same local region across the optimization landscape when we change $p$.
Optimization with \textrm{ADAM} is less prone to remain in flat areas of the landscape since it uses a dynamic step size.
Hence, if after the dynamic step the iterates remain in a part of the landscape with no saddle points that smoothly changes with $p$, we also expect in this case to obtain comparable convergence rates for \gls{SGD} for each fixed $L$.

\noindent
\emph{Step size and batch size.}
In each experiment, the step size is given by $\eta = 10^{-5}$ and the batch size is $b = 1024$.

\noindent
\emph{Fitting procedure.}
We fix a set of probabilities $P \subset [0,1]$ and depths $L =\{1,2,3\}$ and for each pair $(p,l) \in P \times L$ we run the algorithm above.
From the value of the loss from all $T$ iterations of \gls{SGD} $\mathcal{L} = (l_t)_{t=0}^{T}$ in one run, we compute a moving average $a(\mathcal{L})_{t=0}^{T}$, where we average the loss across a window with size given by the number of batches $n_b$ required to complete one epoch.
In this manner we obtain an average convergence rate and diminish the stochasticity from the dataset.
We then fit the averaged loss of the iterates $a(\mathcal{L})_{t=0}^{T}$ for each $p$ and $l$ to the function
\begin{equation}
    f(\alpha_{p,l}, \beta_{p,l}, \gamma_{p,l}) = \alpha_{p,l} \exp(-\beta_{p,l} t) + \gamma_{p,l}.
\end{equation}
We run the experiment $R=10$ times for each $(p,l)$ and obtain an average convergence exponent $(\tilde{\beta}_{p,l})_{(p,l) \in P \times L}$.

\noindent

\begin{figure}[h]
    \captionsetup[subfigure]{justification=centering}
    \centering
    \begin{subfigure}{0.32\textwidth}
        \centering
        \includegraphics[height=0.85\linewidth]{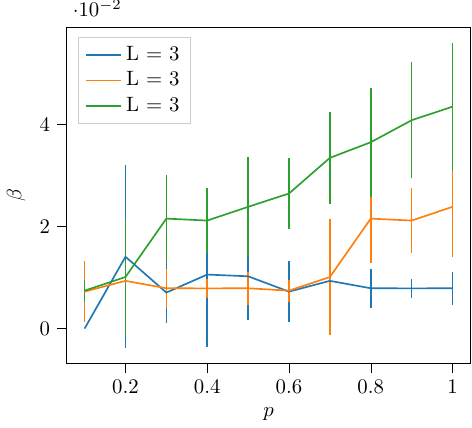}
        \caption*{(a)}
    \end{subfigure}
    \begin{subfigure}{0.32\textwidth}
        \centering
        \includegraphics[height=0.85\linewidth]{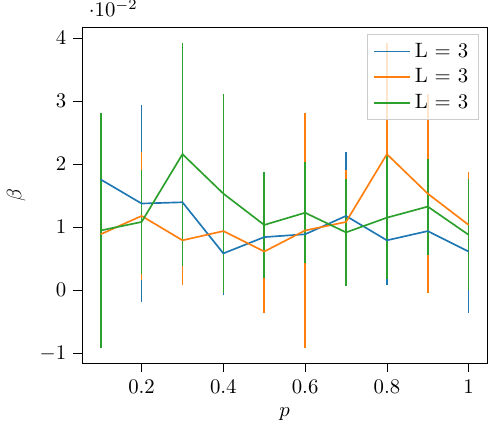}
        \caption*{(b)}
    \end{subfigure}
    \begin{subfigure}{0.32\textwidth}
        \centering
        \includegraphics[height=0.85\linewidth]{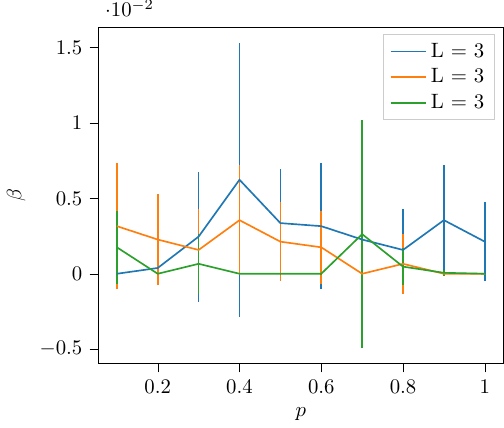}
        \caption*{(c)}
    \end{subfigure}

    \begin{subfigure}{0.32\textwidth}
        \centering
        \includegraphics[height=0.85\linewidth]{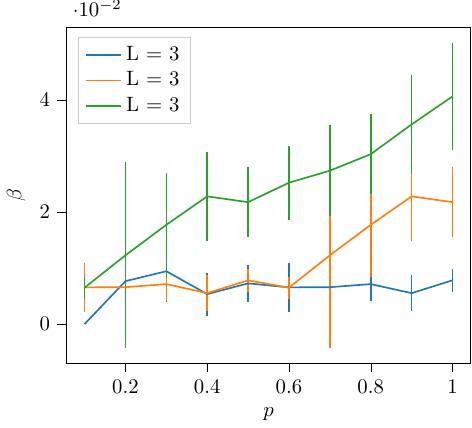}
        \caption*{($a^{\prime}$)}
    \end{subfigure}
    \begin{subfigure}{0.32\textwidth}
        \centering
        \includegraphics[height=0.85\linewidth]{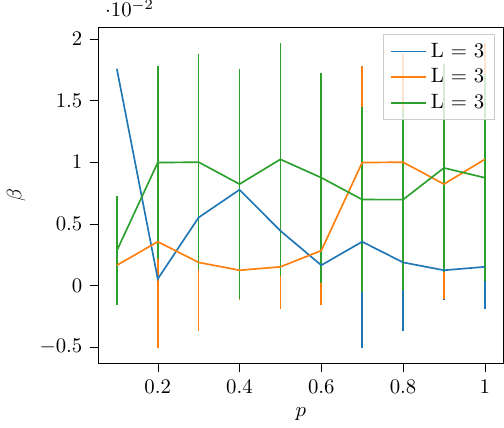}
        \caption*{($b^{\prime}$)}
    \end{subfigure}
    \begin{subfigure}{0.32\textwidth}
        \centering
        \includegraphics[height=0.85\linewidth]{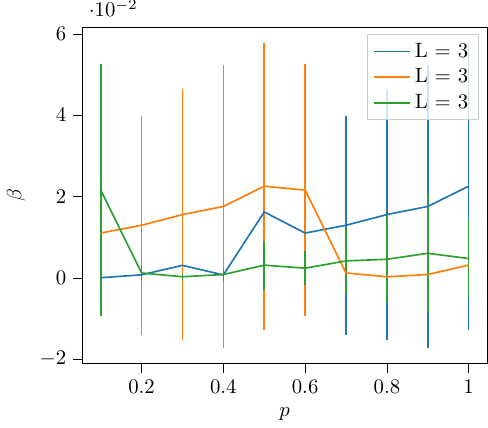}
        \caption*{($c^{\prime}$)}
    \end{subfigure}
    \vspace{-0.5em}
    \caption{The fit $\tilde{\beta}_{p,l}$ for $p \in \{i \times 10^{-1}: i \in [10]\}$ and $l \in \{1,2,3\}$ for LeNet with different widths $D$ and different datasets.
        Here \emph{(a)} \gls{MNIST} with $D=50$; \emph{($a^{\prime}$)} \gls{MNIST} with $D=100$; \emph{(b)} \gls{CIFAR}-100-fine labels with $D=50$; \emph{($b^{\prime}$)} \gls{CIFAR}-100-fine labels with $D=100$; \emph{(c)} \gls{CIFAR}-100-coarse labels with $D=50$; ($c^{\prime}$) \gls{CIFAR}-100-coarse labels with $D=100$.
        While for the \gls{MNIST} dataset there seems to be an increasing dependence of dropout on the convergence rate with the depth $L$, for \gls{CIFAR} no such dependence is observed.
        We remark, however, that in the \gls{CIFAR} datasets encountering saddle points was more common.
        For those areas the loss profile is flat and so we expect the fits to be biased towards the origin in some cases.
    }	\label{fig:deep_NN}
\end{figure}

\subsubsection{Results}

In \Cref{fig:deep_NN} we can see the plots of $\tilde{\beta}_{p,l}$.
As suspected from the heuristic argument, we do not see an increasingly large dependence on $p$ for $L=1,2$ or $3$ when $D \in \{50,100\}$.
For the \gls{MNIST} dataset some dependence on the depth is appreciated, but this may be due to other factors that affect the convergence rate, like initialization issues.
For the \gls{CIFAR} datasets, convergence is greatly affected by saddlepoints despite the use of dropout.
This is, however, common when using \gls{SGD} with small constant stepsizes.
In particular, in practical scenarios other schemes that adjust the stepsize, like e.g.\ \textrm{ADAM}, may be more appropriate when dealing with deep networks with dropout in different layers.
From the experiments it is concluded that despite the stochasticity provided by dropout, the convergence rate is not affected much by a varying dropout probability $1-p$ in wide networks with just few dropout layers.
 \section{Conclusion}
\label{sec:Conclusion}

In this paper we have shown with a probability theoretical proof that a large class of dropout algorithms for neural networks converge almost surely to a unique stationary set of a projected system of \glspl{ODE}.
The result gives a formal guarantee that these dropout algorithms are well-behaved for a wide range of \glspl{NN} and activation functions, and will at least asymptotically not suffer from issues because of the connection to bond percolation.
We leave the extension of this result for nonsmooth activation functions such as \textrm{ReLU} for future work.
Additionally, we established bounds for the sample complexity of \gls{SGD} with dropout to converge to an $\epsilon$-stationary point of a generic nonconvex function.
An upper bound to the rate of convergence of \gls{GD} on the limiting \gls{ODE} of dropout algorithms was established as well for arborescences of arbitrary depth with linear activation functions.
While \gls{GD} on the limiting \gls{ODE} is not strictly a dropout algorithm, the result is a necessary step towards analyzing the convergence rate of the actual stochastic implementations of dropout algorithms.
Finally, Proposition~\ref{prop:estimate_nu_dependent_p} specifically implies that \emph{Dropout} and \emph{Dropconnect} can impair the convergence rate by as much as an exponential factor in the number of layers of thin but deep networks.
We have theoretically and experimentally verified this claim in experiments with a path network.
This fact is in contrast to wide networks with a few dropout layers where a strong dependence on the dropout probability $p$ is not experimentally observed.
These two observations together imply that there is a change of regime in the convergence rate from networks that are wide with a few dropout layers to thin networks with many dropout layers.
 
\acks{We thank the anonymous referees for their feedback.
  Their suggestions have led to an improved paper.
}

\bibliography{biblio}

\appendix

\newpage

\begin{center}
  \Large{\bf{Appendix}}
\end{center}

\section{Backpropagation Algorithm}
\label{sec:Appendix_Backpropagation}

We define the backpropagation algorithm used in \Cref{sec:Preliminaries} to compute the estimate of the gradient.

\begin{definition}
    \label{definition:Backpropagation}

    Assume $\sigma \in C^1(\realNumbers)$.
    Given weights $W \in \mathcal{W}$ and input--output pair $(x,y) \in \realNumbers^{d_0} \times \realNumbers^{d_L}$, the tensor
    $
        \mathrm{B}_W(x,y) \in \realNumbers^{d_{L} \times d_{L-1}} \times \cdots \times \realNumbers^{d_1 \times d_0}
    $
    is calculated iteratively by:
    \begin{enumerate}[topsep=0pt,itemsep=-1ex,partopsep=1ex,parsep=1ex]
        \item Computing $A_1, \ldots, A_L$ using Definition~\ref{def:NN__Feedforward_recursion}.
        \item Calculating
for $i = L-1, \ldots, 1$,
              \begin{align}
                  R_{L}
                   &
                  = A_L =
                  (y - W_L A_{L-1})
                  \in \R^{d_L},
                  \quad
                  \nonumber \\
                  R_{i}
                   &
                  = (W_{i+1}^{\mathrm{T}} R_{i+1}) \odot (\sigma^{\prime}( W_i A_{i-1}))
                  \in \R^{d_{i}}.
                  \label{eqn:Backpropagation_variables}
              \end{align}
        \item Setting for $i \in [L]$,
              $
                  \bigl( \mathrm{B}_W(x,y) \bigr)_i
                  = -2R_{i}
                  A_{i-1}^{\mathrm{T}}.
              $
\end{enumerate}
\end{definition}

\noindent
Definition~\ref{definition:Backpropagation} is essentially a computationally efficient manner of calculating the gradient $\nabla  l(\Psi_W(x),y)$ in \eqref{eqn:Oracle_risk_function}, leveraging the \gls{NN}'s layered structure together with the chain rule of differentation to come to a recursive computation of the partial derivatives.

\section{ODE method}
\label{sec:appendix_ode_method}
Regarding our second result in Proposition~\ref{prop:convergence_of_dropout_on_an_expanding_tree}, observe that \gls{GD} on a limiting \gls{ODE} is not exactly a dropout algorithm.
Analyzing \gls{GD}'s convergence rate however is an important stepping stone towards analyzing the convergence rate of dropout algorithms.
To see the mathematical relation, consider that any dropout algorithm updates the weights
\begin{equation}
    \itr{W}{n+1}
    = \itr{W}{n} + \itrd{\alpha}{n} \itr{\Delta}{n+1}
    \label{eqn:Dropout_algorithm__Unprojected}
\end{equation}
randomly for $n = 0, 1, 2, \cdots$.
Here, the $\itrd{\alpha}{n}$ denote the step sizes of the algorithm, and the $\itr{\Delta}{n+1}$ represent the random directions that result from the act of dropping weights.
As we will show in this paper under assumptions of independence, these random directions satisfy
\begin{align}
    \expectation{ \itr{\Delta}{n+1} \mid \itr{W}{0}, \ldots, \itr{W}{n} }
     &
    = - \nabla \mathcal{D}( \itr{W}{n} )
    \label{eqn:dropout_regular_intro}
\end{align}
for some continuous, differentiable function $\mathcal{D}(W)$.
Observe that the algorithm in \eqref{eqn:Dropout_algorithm__Unprojected} satisfies
$
    \itr{W}{n+1}
    =
    \itr{W}{n}
    +
    \itrd{\alpha}{n}
    (
    - \nabla \mathcal{D}( \itr{W}{n} ) + \itr{M}{n+1}
    )
$
where
$
    \itr{M}{n+1}
    =
    \expectation{ \itr{\Delta}{n+1} \mid \itr{W}{0}, \ldots, \itr{W}{n} } - \itr{\Delta}{n+1}
$
describes a \emph{martingale difference} sequence.
This martingale difference sequence's expectation with respect to the past $\itr{W}{0}, \ldots, \itr{W}{n}$ is zero.

For diminishing step sizes $\itrd{\alpha}{n}$, we can consequently view dropout algorithms as in \eqref{eqn:Dropout_algorithm__Unprojected} as being noisy discretizations of the ordinary differential equation
\begin{equation}
    \frac{\d{W}}{\d{t}}
    = - \nabla \mathcal{D}( W(t) ).
    \label{eqn:Gradient_flow_introduction}
\end{equation}
In fact, we employ the so-called \emph{ordinary differential equation method} \citep{kushner2003stochastic,borkar2009stochastic}, which formally establishes that the random iterates in \eqref{eqn:Dropout_algorithm__Unprojected} follow the trajectories of the gradient flow in \eqref{eqn:Gradient_flow_introduction}.
Hence, after sufficiently many iterations $n$ and for a sufficiently small step size $\alpha$, the convergence rate of the deterministic \gls{GD} algorithm
\begin{equation}
    \itrd{W}{n+1}
    = \itrd{W}{n} - \alpha \nabla \mathcal{D}(\itrd{W}{n})
\end{equation}
gives insight into the convergence rate of the stochastic dropout algorithm in \eqref{eqn:Dropout_algorithm__Unprojected}.

\section{Projection operator}
\label{sec:Projection_operator}

We define here the projection operator $\pi$ used in \Cref{sec:Results__Almost_sure_convergence_of_projected_dropout_algorithms}.
Say that $\mathcal{H}$ is defined by $l$ smooth constraints \albert{$q_{i}: \mathcal{W} \to \R$, $i=1,\ldots, l$ satisfying $q_1(W) \leq 0, \ldots, q_l(W) \leq 0$, i.e.,} $\mathcal{H} = \lbrace W \in \mathcal{W} \albert{~ : ~ } q_i(W) \leq 0 \ \forall i \in [l] \rbrace$.
Denote by $\nabla \mathcal{D}|_{\mathcal{H}}(W)$ the gradient of $\mathcal{D}(W)$ restricted to $\mathcal{H}$ and let $\rm{T}_{W} \mathcal{W}$ be the tangent space of $\mathcal{W}$ at $W$.
Suppose that $\nabla q_i(W) \neq 0$ whenever $q_i(W) = 0$, and that these are linearly independent.
At any point $W \in \partial \mathcal{H}$, we define the outer normal cone
\begin{align}
    C(W)
    \triangleq \{ v \in \rm{T}_W \mathcal{W} \ \albert{~ : ~ } \  & \nabla q_i(W) v^T \geq 0
    \text{ for } i \in [l] \text{ s.t.
    }  q_i(W) = 0 \}
    .
    \label{eqn:cone_projection_manifold}
\end{align}
We also assume that $C(W)$ is upper semicontinuous, i.e., if $\tilde{W} \in B_{\mathcal{H}}(W, \delta)$, where $B_{\mathcal{H}}(W, \delta)$ is the ball of radius $\delta > 0$ centered at $W$ and intersected with $\mathcal{H}$, then $C(W)= \cap_{\delta > 0} \allowbreak \bigl( \cup_{\tilde{W} \in B_{\mathcal{H}}(W, \delta)} C( \tilde{W}) \bigr)$.
Let $\pi(W) \triangleq - t \indicator{ W \in \partial \mathcal{H}}$ with $t \in C(W)$ minimal to resolve the violated constraints of $\mathcal{D}|_{\mathcal{H}}(W)$ at $W \in \partial \mathcal{H}$ so that $\mathcal{D}|_{\mathcal{H}}(W) + \pi(W)$ points inside $\mathcal{H}$.
In particular, we have
\begin{equation}
    \pi(W)
    = -\sum_{i=1}^l \lambda_i(W) \nabla q_i(W) \in -C(W)
\end{equation}
where $\lbrace \lambda_i(W) \geq 0 \rbrace_{i=1}^l$ are functions such that $\lambda_i(W) = 0$ if $q_i(W) < 0$.

\section{Proof of \texorpdfstring{Proposition~\ref{prop:dropout_converges}}{of convergence}}
\label{appendix:Proof_that_dropout_converges}

The proof of Proposition~\ref{prop:dropout_converges} relies on the framework of stochastic approximation in \cite{kushner2003stochastic}.
Specifically, Proposition~\ref{prop:dropout_converges} follows from Theorem~2.1 on p.~127 if we can show that its conditions (A2.1)--(A2.6) on p.~126 are satisfied.
In the notation of Sections~\ref{sec:Preliminaries}, \ref{sec:Results__Almost_sure_convergence_of_projected_dropout_algorithms}, these conditions read:
\begin{itemize}
    \item[\textrm{(A2.1)}] $\sup_{t} \expectation{ \pnorm{\itr{\Delta}{t+1}}{\mathrm{F}} } < \infty$;
    \item[\textrm{(A2.2)}] there is a measurable function $\bar{g}(\cdot)$ of $W$ and there are random variables $\itr{\beta}{t+1}$ such that
        \begin{equation}
            \expectation{ \itr{\Delta}{t+1} \mid \mathcal{F}_t }
            =
            \bar{g}(\itr{W}{t}) + \itr{\beta}{t+1},
        \end{equation}
        where $\mathcal{F}_t$ denotes the smallest $\sigma$-algebra generated by $\cup_{s \leq t} \{ \itr{W}{0}, ( \itr{F}{s}, \itr{X}{s}, \itr{Y}{s} ) \}$;
    \item[\textrm{(A2.3)}] $\bar{g}(\cdot)$ is continuous;
    \item[\textrm{(A2.4)}] the step sizes satisfy
        \begin{gather}
            \label{eqn:Step_sizes_diverge_but_not_too_fast_appendix}
            \sum_{t=1}^\infty \itrd{\alpha}{t} = \infty,
            \itrd{\alpha}{n} \geq 0, \itrd{\alpha}{n} \to 0
            \textnormal{ for }
            n \geq 0
            \textnormal{ and }
            \itrd{\alpha}{n} = 0
            \textnormal{ for }
            n < 0
            ;
            \\
            \quad
            \sum_{t=1}^\infty ( \itrd{\alpha}{t} )^2 < \infty
            ;
        \end{gather}
    \item[\textrm{(A2.5)}] $\sum_{t} \itrd{\alpha}{t} \pnorm{ \itr{\beta}{t}}{\mathrm{F}} < \infty$ w.p.
        \ one;
    \item[\textrm{(A2.6)}] $\bar{g}(\cdot) = - \nabla \mathcal{D}(\cdot)$ for a continuously differentiable real-valued $\mathcal{D}(\cdot)$ and $\mathcal{D}(\cdot)$ is constant on each stationary set $S_i$.
\end{itemize}

We next also state for your convenience Theorem 2.1 by \cite{kushner2003stochastic} in the notation of this paper.
Their result does require some notation, as it characterizes the limiting behavior of the iterates of
\begin{equation}
    \itr{W}{n+1}
    =
    \mathcal{P}_{\mathcal{H}}\bigl( \itr{W}{n} - \alpha \itr{\Delta}{n+1} \bigr)
    \triangleq
    \itr{W}{n} - \alpha \itr{\Delta}{n+1} + \itr{Z}{n+1}
    .
    \label{eqn:Dropout_algorithm__Projected}
\end{equation}
For any sequence of step sizes $\itrd{\alpha}{n}$ satisfying
(A2.4),
define $t_0 = 0$ and $t_n = \sum_{i=0}^{n-1} \itrd{\alpha}{i}$.
Define the continuous-time interpolation
\begin{equation}
    W_0(t)
    =
    \begin{cases}
        \itr{W}{n} & \textnormal{for} \quad t_n \leq t < t_{n+1}, \\
        \itr{W}{0} & \textnormal{for} \quad t \leq 0,             \\
    \end{cases}
\end{equation}
as well as for $m \in \naturalNumbersZero$, the shifted processes $W_m(t) = W_0( t_m + t )$ for $t \in (-\infty,\infty)$.
Let furthermore
$
    o(t)
    =
    \inf \{ n \in \naturalNumbersZero : t_n \leq t < t_{n+1} \}
$
for $t \in [0,\infty)$,
and
$
    o(t)
    = 0
$
for $t \in (-\infty,\infty)$,
and define
\begin{equation}
    Z_0(t)
    =
    \begin{cases}
        \sum_{i=0}^{o(t)-1} \itrd{\alpha}{i} Z_i & \textnormal{for} \quad t \in [0,\infty),       \\
        0                                        & \textnormal{for} \quad t \in (-\infty,\infty), \\
    \end{cases}
\end{equation}
as well as for $m \in \naturalNumbersZero$, the shifted processes $Z_m(t) = \sum_{i=m}^{o(t_m+t)-1}$ for $t \in [0,\infty)$ and $Z_m(t) = - \sum_{i=o(t_m+t)}^{m-1} \itrd{\alpha}{i} Z_i$ for $t \in (-\infty,0)$.
The following now holds:

\begin{theorem}[A part of Theorem 2.1 by \cite{kushner2003stochastic}]
    \label{thm:Kuschner_adapted}
    Let conditions \allowbreak(A2.1)--(A2.5) hold for algorithm \eqref{eqn:Dropout_algorithm__Projected}, with the projection onto $\mathcal{H}$ being as described in \refAppendixSection{sec:Projection_operator}.
    Then there is a set $N$ of probability zero such that for $\omega \not\in N$, the set of functions $\{ W_m(\omega,\cdot), \allowbreak Z_m(\omega,\cdot), \allowbreak m < \infty \}$ is equicontinuous.
    Let $( W(\omega, \cdot), Z(\omega, \cdot) )$
    denote the limit of some convergent subsequence.
    Then this pair satisfies the
    projected \gls{ODE} \eqref{eqn:ODE_dropout}, and $\{ \itr{W}{n}(\omega) \}$ converges to some limit set of the \gls{ODE} in $\mathcal{H}$.
    Suppose that (A2.6) holds.
    Then, for almost all $\omega$, $\{ \itr{W}{n}(\omega) \}$ converges to a unique
    $S_i$.
\end{theorem}

In order to apply \refTheorem{thm:Kuschner_adapted} and arrive at Proposition~\ref{prop:dropout_converges}, we verify conditions (A2.1)--(A2.6) through Lemmas~\ref{lem:Boundedness_of_the_variance_of_stochastic_iterands}--\ref{lem:Dw_is_constant_on_Si} shown next in \refAppendixSection{sec:Appendix__Verification_of_conditions_A21_to_A26}.
These lemmas are proven in Appendices~\ref{secappendix:Boundedness_of_the_variance_of_stochastic_iterands}--\ref{secappendix:Dw_is_constant_on_Si}, respectively.

\subsection{Verification of conditions (A2.1)--(A2.6)}
\label{sec:Appendix__Verification_of_conditions_A21_to_A26}

First we assume conditions (N1)--(N3) and we prove that the variance of the random update direction in \eqref{eqn:Random_direction_Delta_as_function_of_Backpropagation} is finite.
This verifies condition (A2.1).
The proof can be found in \refAppendixSection{secappendix:Boundedness_of_the_variance_of_stochastic_iterands}:

\begin{lemma}
    \label{lem:Boundedness_of_the_variance_of_stochastic_iterands}
    Assume (N1)--(N3) from Proposition~\ref{prop:dropout_converges}.
    Then $\sup_{t \in \N} \expectation{ \pnorm{ \itr{ \Delta_i }{t+1} }{\textnormal{F}}^2 } < \infty$ for $i = 0, 1, \ldots, L$.
\end{lemma}

We prove next that if $\sigma \in C^r_{PB}(\R)$, then the random update direction in \eqref{eqn:Random_direction_Delta_as_function_of_Backpropagation}, conditional on all prior updates, has conditional expectation $\nabla \mathcal{D}(\itr{W}{t})$.
Lemma~\ref{lem:Expectation_cond_Ft_of_Delta_i} verifies conditions (A2.2), (A2.3), and (A2.5) (in particular, here $\itr{\beta}{t} = 0$).
The proof can be found in \refAppendixSection{secappendix:Expectation_cond_Ft_of_Delta_i}:

\begin{lemma}
    \label{lem:Expectation_cond_Ft_of_Delta_i}
    Assume (N2)--(N4) from Proposition~\ref{prop:dropout_converges}.
    Then
    $
        \expectation{ \itr{\Delta}{t+1} | \mathcal{F}_t }
        = \nabla \mathcal{D}(\itr{W}{t})
    $.
    Furthermore, $\nabla \mathcal{D}: \mathcal{W} \to \mathcal{W}$ is $r-1$ times continuously differentiable.
\end{lemma}

From these conditions the first part of Proposition~\ref{prop:dropout_converges} follows.
To prove the second part of Proposition~\ref{prop:dropout_converges}, we have to prove that the set of stationary points $S_\mathcal{H}$ is well-behaved in the sense that $\mathcal{D}|_{S_i}(W)$ is constant.
If an objective function is sufficiently differentiable, this is guaranteed by the Morse--Sard Theorem \citep{morse1939behavior,sard1942measure}.
In the present case however we must take into account the possibility of an intersection of the set of stationary points with the boundary $\partial \mathcal{H}$.
Assuming (N4) and (N5) provides sufficient conditions.
The proof of Lemma~\ref{lem:Dw_is_constant_on_Si} can be found in \refAppendixSection{secappendix:Dw_is_constant_on_Si}:

\begin{lemma}
    \label{lem:Dw_is_constant_on_Si}
    If (N2)--(N5) hold, then $\mathcal{D}(W)$ is constant on each $S_i$.
\end{lemma}

Since Conditions (A2.1)--(A2.6) of Thm.~2.1 on p.~127 in \cite{kushner2003stochastic} are now proven satisfied, the proof of Proposition~\ref{prop:dropout_converges} is now completed.
\BlackBox

\subsubsection{Boundedness of \texorpdfstring{$\itr{\Delta}{t+1}$}{the iterates} in expectation -- Proof \texorpdfstring{of Lemma~\ref{lem:Boundedness_of_the_variance_of_stochastic_iterands}}{}}
\label{secappendix:Boundedness_of_the_variance_of_stochastic_iterands}

We need to carefully track all sequences of random variables created by a dropout algorithm throughout this proof, which we state here first explicitly.

\begin{definition}[Dropout iterates]
    \label{definition:dropout_sgd}
    During its $(t+1)$-st \emph{feedforward step}, the algorithm iteratively calculates
    \begin{align}
        \itr{A_0}{t+1}
= \itr{X}{t+1},
        \quad
\itr{A_i}{t+1}
= \sigma( (\itr{W_i}{t} \odot \itr{F_i}{t+1}) \itr{A_{i-1}}{t+1} )
        \label{eqn:Feedforward_variables_of_Dropout}
    \end{align}
    for $i = 1, 2, \ldots, L-1$, to output
    \begin{equation}
        \Psi_{\itr{F}{t+1} \odot \itr{W}{t}}^{}(\itr{X}{t+1} )
        = (\itr{W_L}{t} \odot \itr{F_L}{t+1}) \itr{A_{L-1}}{t+1} = \itr{A_{L}}{t+1}.
    \end{equation}
    Subsequently for its $(t+1)$-st \emph{backpropagation step} the algorithm calculates
    \begin{align}
        \itr{R_{L}}{t+1}
         & =
        (\itr{Y}{t+1} - (\itr{W_{L}}{t} \odot \itr{F_{L}}{t+1}) \itr{A_{L-1}}{t+1})
        \in \R^{d_L} ,
        \nonumber                                                                                                                                                    \\
        \itr{R_{j}}{t+1}
         & = ((\itr{W_{j+1}}{t} \odot \itr{F_{j+1}}{t+1})^T \itr{R_{j+1}}{t+1}) \odot (\sigma^{\prime}( (\itr{W_{j}}{t} \odot \itr{F_{j}}{t+1}) \itr{A_{j-1}}{t+1}))
        \in \R^{d_{i}},
        \label{eqn:Backpropagation_variables_of_Dropout}
    \end{align}
    iteratively for $j = L-1, \ldots, 1$.
    The algorithm then calculates
    \begin{align}
        \itr{\Delta_i}{t+1}
        = -2\itr{F_i}{t+1} \odot (\itr{R_{i}}{t+1}(\itr{A_{i-1}}{t+1})^{\mathrm{T}})
        \label{eqn:Random_update_direction_of_Dropout}
    \end{align}
    for $i = 1, \ldots, L$, and finally updates all weights according to \eqref{update_sgd}.
    \label{def:Dropout_SGD_algorithm}
\end{definition}

The idea of the proof of Lemma~\ref{lem:Boundedness_of_the_variance_of_stochastic_iterands} is to expand the terms in $\itr{ \Delta_{i} }{t+1}$ defined in Definition~\ref{def:Dropout_SGD_algorithm} recursively, and identify a polynomial in variables $\{ \pnorm{Y}{2}^n \pnorm{X}{2}^m \}_{m \in \naturalNumbersZero}$ and $n=0,1,2$.
We will use several bounds that pertain to the Frobenius norm, written down in Lemma~\ref{lemma:inequalities_norm_matrices} in \refAppendixSection{appendix:inequalities_norm_matrix}, and we will iterate these in a moment.

First, we will prove two bounds on the activation function applied to an arbitrary matrix $A$.
Recall that $\sigma \in C_{PB}^2(\R)$ by assumption (N1).
There thus (i) exists some $C_0, k_0 > 0$ such that $| \sigma(z) | \leq C_0(1+z^2)^{k_0}$ for all $z \in \realNumbers$, and there exists some $C_1, k_1 > 0$ such that $| \sigma^{\prime}(z) | \leq C_1 (1+z^2)^{k_1}$ for all $z \in \realNumbers$.
Let $k = \max \{ 1, k_0, k_1 \}$.
Then
\begin{equation}
    \pnorm{\sigma(A)}{\mathrm{F}}^2
    = \sum_{i,j} | \sigma( A_{ij} ) |^2
    \eqcom{i}\leq C_0 \sum_{i,j} ( 1 + A_{ij}^2 )^{k}
    \eqcom{Lemma~\ref{lemma:inequalities_norm_matrices}}\leq C_2 ( 1 + \pnorm{A}{\mathrm{F}} )^{2k}
    \label{eqn:Bound_on_Frobenius_norm_of_poly_bounded_sigma}
\end{equation}
for some constant $C_2 > 0$.
Similarly there exists some $C_3 > 0$ such that $\norm{\sigma^{\prime}(A)}_F \leq C_3 (1 + \pnorm{A}{\mathrm{F}} )^{k}$.
Note furthermore that (ii) for all $l \geq 0$, by submultiplicativity of the Frobenius norm,
\begin{align}
    ( 1 + \pnorm{ A \sigma(B) }{\mathrm{F}} )^l
     &
    \eqcom{ii}\leq ( 1 + \pnorm{A}{\mathrm{F}} \pnorm{\sigma(B)}{\mathrm{F}} )^l
    \nonumber \\  &
       \eqcom{\ref{eqn:Bound_on_Frobenius_norm_of_poly_bounded_sigma}}\leq \bigl( 1 + C_2^{1/2} \pnorm{A}{\mathrm{F}} (1 + \pnorm{B}{\mathrm{F}})^k \bigr)^l
       \leq C_4 ( 1 + \pnorm{A}{\mathrm{F}})^l (1 + \pnorm{B}{\mathrm{F}})^{kl}
       \label{eqn:Bound_on_Frobenius_norm_plus_one_of_A_sigma_B}
\end{align}
for $C_4 = \max \{ 1, C_2^{l/2} \} > 0$.
Again, a similar bound holds for $\sigma'$.

Next, note that we have by (i) submultiplicativity and Lemma~\ref{lemma:inequalities_norm_matrices} that
\begin{equation}
    \pnorm{\itr{\Delta_i}{t+1}}{\mathrm{F}}
    = \pnorm{ \itr{F_i}{t+1} \odot ( \itr{R_i}{t+1} ( \itr{A_{i-1}}{t+1} )^{\mathrm{T}} ) }{\mathrm{F}}
    \eqcom{i} \leq \pnorm{\itr{F_i}{t+1}}{\mathrm{F}} \pnorm{\itr{R_{i}}{t+1}}{\mathrm{F}} \pnorm{\itr{A_{i-1}}{t+1}}{\mathrm{F}}.
\end{equation}
The first term is bounded with probability one: $\itr{F_{i,r,l}}{t} \in \{ 0,1 \}$ for all $i,r,l,t$.
For the second term, consider the following bound:
\begin{align}
    \pnorm{\itr{R_{i}}{t+1}}{\mathrm{F}}
     &
    \eqcom{\ref{eqn:Backpropagation_variables_of_Dropout}}= \pnorm{ (\itr{W_{i+1}}{t} \odot \itr{F_{i+1}}{t+1})^{\mathrm{T}} \itr{R_{i+1}}{t + 1} \odot \sigma^{\prime} \bigl( ( \itr{W_{i}}{t} \odot \itr{F_{i}}{t+1} ) \itr{A_{i-1}}{t+1} \bigr)}{\mathrm{F}}
    \nonumber \\  &
       \eqcom{Lemma~\ref{lemma:inequalities_norm_matrices}}\leq \pnorm{ \itr{W_{i+1}}{t} \odot \itr{F_{i+1}}{t+1} }{\mathrm{F}} \pnorm{ \sigma^{\prime} \bigl( ( \itr{W_{i}}{t} \odot \itr{F_{i}}{t+1} ) \itr{A_{i-1}}{t+1} \bigr) }{\mathrm{F}} \pnorm{\itr{R_{i+1}}{t + 1}}{\mathrm{F}}
       \label{eqn:Bounded_variance_recursive_inequalities__Nr_1}
\end{align}
for $1 \leq i \leq L$, where we have also used the submultiplicative property.
For the third term, consider the next bound: (i) recursing \eqref{eqn:Bound_on_Frobenius_norm_plus_one_of_A_sigma_B} with $A = I$ and $B = (\itr{W_j}{t} \odot \itr{F_j}{t+1} ) \itr{A_{j-1}}{t+1}$ etc, we obtain that there exists some $C_5 > 0$, say, so that
\begin{align}
    \pnorm{\itr{A_{j}}{t+1}}{\mathrm{F}}
     &
    \eqcom{\ref{eqn:Feedforward_variables_of_Dropout}}= \pnorm{\sigma((\itr{W_j}{t} \odot \itr{F_j}{t+1}) \itr{A_{j-1}}{t+1})}{\mathrm{F}}
    \eqcom{\ref{eqn:Bound_on_Frobenius_norm_of_poly_bounded_sigma}} \leq C_2 ( 1 + \pnorm{(\itr{W_j}{t} \odot \itr{F_j}{t+1}) \itr{A_{j-1}}{t+1} }{\mathrm{F}})^k
    \label{eqn:Bounded_variance_recursive_inequalities__Nr_2}
    \\  &
       \eqcom{Lemma~\ref{lemma:inequalities_norm_matrices}}\leq C_2 ( 1 + \pnorm{ \itr{W_j}{t} \odot \itr{F_j}{t+1} }{\mathrm{F}} )^k ( 1 + \pnorm{\itr{A_{j-1}}{t+1} }{\mathrm{F}} )^k
    \nonumber \\  &
       \eqcom{i}\leq C_5 \bigl( 1 + \pnorm{\itr{X}{t+1}}{2} \bigr)^{k^{j}} \prod_{l=1}^{j-1} \bigl( 1 + \pnorm{ \itr{W_l}{t} \odot \itr{F_l}{t+1} }{\mathrm{F}} \bigr)^{k^{j-l}}
       \nonumber
\end{align}
for $j = 1, 2, \ldots, L-1$.
Similar to the derivation in \refEquation{eqn:Bounded_variance_recursive_inequalities__Nr_2}, we obtain instead with $\sigma^{\prime}$ that there exists some $C_6 > 0$ such that
\begin{equation}
    \pnorm{\sigma^{\prime}((\itr{W_j}{t} \odot \itr{F_j}{t+1}) \itr{A_{j-1}}{t+1})}{\mathrm{F}} \leq C_6 \bigl( 1 + \pnorm{\itr{X}{t+1}}{2} \bigr)^{k^{j}} \prod_{l=1}^{j-1} \bigl( 1 + \pnorm{ \itr{W_l}{t} \odot \itr{F_l}{t+1} }{\mathrm{F}} \bigr)^{k^{j-l}}.
    \label{eqn:Bounded_variance_recursive_inequalities__Nr_3}
\end{equation}
Recall that $
    \pnorm{\itr{\Delta_i}{t+1}}{\mathrm{F}}
    \leq \pnorm{\itr{F_i}{t+1}}{\mathrm{F}} \pnorm{\itr{R_{i}}{t+1}}{\mathrm{F}} \pnorm{\itr{A_{i-1}}{t+1}}{\mathrm{F}}
$.
This, together with using \refEquation{eqn:Bounded_variance_recursive_inequalities__Nr_1} repeatedly for $j = i, \ldots, L-1$, and \eqref{eqn:Bounded_variance_recursive_inequalities__Nr_2}, \eqref{eqn:Bounded_variance_recursive_inequalities__Nr_3}, yields the following inequality
\begin{align}
     &
    \pnorm{\itr{\Delta_i}{t+1}}{\mathrm{F}}
    \eqcom{\ref{eqn:Bounded_variance_recursive_inequalities__Nr_1}}\leq \pnorm{\itr{F_i}{t+1}}{\mathrm{F}} \pnorm{\itr{R_L}{t+1}}{\mathrm{F}} \pnorm{\itr{A_{i}}{t+1}}{\mathrm{F}} \prod_{j=i}^{L-1} \pnorm{ \itr{W_{j+1}}{t} \odot \itr{F_{j+1}}{t+1} }{\mathrm{F}} \pnorm{ \sigma^{\prime} \bigl( ( \itr{W_{j}}{t} \odot \itr{F_{j}}{t+1} ) \itr{A_{j-1}}{t+1} \bigr) }{\mathrm{F}}
    \nonumber \\  &
       \eqcom{\ref{eqn:Bounded_variance_recursive_inequalities__Nr_2}}\leq C_5 \pnorm{\itr{F_i}{t+1}}{\mathrm{F}} \bigl( 1 + \pnorm{\itr{X}{t+1}}{2} \bigr)^{k^{i}} \prod_{l=1}^{i-1} \bigl( 1 + \pnorm{ \itr{W_l}{t} \odot \itr{F_l}{t+1} }{\mathrm{F}} \bigr)^{k^{i-l}}
    \nonumber \\  &
       \phantom{\leq} \times \pnorm{\itr{R_L}{t+1}}{\mathrm{F}} \prod_{j=i}^{L-1} \pnorm{ \itr{W_{j+1}}{t} \odot \itr{F_{j+1}}{t+1} }{\mathrm{F}} \pnorm{ \sigma^{\prime} \bigl( ( \itr{W_{j}}{t} \odot \itr{F_{j}}{t+1} ) \itr{A_{j-1}}{t+1} \bigr) }{\mathrm{F}}
    \nonumber \\  &
       \eqcom{\ref{eqn:Bounded_variance_recursive_inequalities__Nr_3}}\leq C_7 \pnorm{\itr{F_i}{t+1}}{\mathrm{F}} \bigl( 1 + \pnorm{\itr{X}{t+1}}{2} \bigr)^{k^{i}} \prod_{l=1}^{i-1} \bigl( 1 + \pnorm{ \itr{W_l}{t} \odot \itr{F_l}{t+1} }{\mathrm{F}} \bigr)^{k^{i-l}}
    \nonumber \\  &
       \phantom{\leq} \times \pnorm{\itr{R_L}{t+1}}{\mathrm{F}} \prod_{j=i}^{L-1} \pnorm{ \itr{W_{j+1}}{t} \odot \itr{F_{j+1}}{t+1} }{\mathrm{F}} \bigl( 1 + \pnorm{\itr{X}{t+1}}{2} \bigr)^{k^{j}} \prod_{l=1}^{j} \bigl( 1 + \pnorm{ \itr{W_l}{t} \odot \itr{F_l}{t+1} }{\mathrm{F}} \bigr)^{k^{j-l}}
    \nonumber \\  &
       \leq C_7 \pnorm{\itr{F_i}{t+1}}{\mathrm{F}} \pnorm{\itr{R_{L}}{t+1}}{\mathrm{F}} \Bigl( \prod_{j=i}^{L-1} \pnorm{ \itr{W_{i+1}}{t} \odot \itr{F_{i+1}}{t+1} }{\mathrm{F}} \Bigr)
    \nonumber \\  &
       \phantom{\leq}\times \Bigl( \prod_{j=i}^{L-1} \bigl( 1 + \pnorm{\itr{X}{t+1}}{2} \bigr)^{2k^{j}} \prod_{l=1}^{j} \bigl( 1 + \pnorm{ \itr{W_l}{t} \odot \itr{F_l}{t+1} }{\mathrm{F}} \bigr)^{2k^{j-l}} \Bigr)
    \nonumber \\  &
       = C_7 \pnorm{\itr{F_i}{t+1}}{\mathrm{F}} \pnorm{\itr{R_{L}}{t+1}}{\mathrm{F}} \Bigl( \prod_{j=i}^{L-1} \pnorm{ \itr{W_{i+1}}{t} \odot \itr{F_{i+1}}{t+1} }{\mathrm{F}} \Bigr)
    \nonumber \\  &
       \phantom{\leq}\times \bigl( 1 + \pnorm{\itr{X}{t+1}}{2} \bigr)^{ \sum_{j=i}^{L-1} 2 k^{j}} \Bigl( \prod_{j=i}^{L-1} \prod_{l=1}^{j} \bigl( 1 + \pnorm{ \itr{W_l}{t} \odot \itr{F_l}{t+1} }{\mathrm{F}} \bigr)^{2k^{j-l}} \Bigr).
       \label{eqn:Bounded_variance_recursive_inequalities__Nr_4}
\end{align}
Lastly, we bound $\pnorm{\itr{R_{L}}{t+1}}{\mathrm{F}}$.
By applying (i) subadditivity of the norm $\norm{A + B}_{F} \leq \norm{A}_{F} + \norm{B}_{F}$ and then using the elementary bound $(a+b)^2 \leq 2 ( a^2 + b^2 )$ as well as submultiplicativity, we obtain
\begin{align}
     &
    \pnorm{\itr{R_{L}}{t+1}}{\mathrm{F}}
    \eqcom{\ref{eqn:Backpropagation_variables_of_Dropout}}= \pnorm{\itr{Y}{t+1} - ( \itr{W_L}{t} \odot \itr{F_L}{t+1} ) \itr{A_{L-1}}{t+1}}{\mathrm{F}}
    \label{eqn:Bounded_variance_recursive_inequalities__Nr_5}
    \\  &
       \eqcom{i}\leq \pnorm{\itr{Y}{t+1}}{2}^2 + \pnorm{\itr{W_L}{t} \odot \itr{F_L}{t+1} }{\mathrm{F}} \pnorm{\itr{A_{L-1}}{t+1}}{{\mathrm{F}}}
    \nonumber \\  &
       \eqcom{\ref{eqn:Bounded_variance_recursive_inequalities__Nr_2}}\leq \pnorm{\itr{Y}{t+1}}{2} + \pnorm{ \itr{W_L}{t} \odot \itr{F_L}{t+1} }{\mathrm{F}} \bigl( 1 + \pnorm{\itr{X}{t+1}}{2} \bigr)^{k^{L-1}} \prod_{l=1}^{L-1} \bigl( 1 + 2 \pnorm{ \itr{W_l}{t} \odot \itr{F_l}{t+1} }{\mathrm{F}} \bigr)^{k^{L-l}}.
       \nonumber
\end{align}
By combining inequalities \refEquation{eqn:Bounded_variance_recursive_inequalities__Nr_4}, \refEquation{eqn:Bounded_variance_recursive_inequalities__Nr_5}, and upper bounding the exponent $k^{L-1}$ of the term $1 + \pnorm{\itr{X}{t+1}}{\mathrm{F}}$ in \refEquation{eqn:Bounded_variance_recursive_inequalities__Nr_5} by $2\sum_{j=1}^{L-1} k^{j}$, we conclude that
\begin{align}
     &
    \pnorm{\itr{\Delta_i}{t+1}}{\mathrm{F}}
    \nonumber \\  &
       \leq C_8 \pnorm{\itr{Y}{t+1}}{2} \bigl( 1 + \pnorm{\itr{X}{t+1}}{2} \bigr)^{ 2 \sum_{j=1}^{L-1} k^{j}} \pnorm{\itr{F_i}{t+1}}{\mathrm{F}} P_1\bigl( \pnorm{ \itr{W_1}{t} \odot \itr{F_1}{t+1} }{\mathrm{F}}, \ldots, \pnorm{ \itr{W_L}{t} \odot \itr{F_L}{t+1} }{\mathrm{F}} \bigr)
    \nonumber \\  &
       \phantom{\leq}+ C_9 \bigl( 1 + \pnorm{\itr{X}{t+1}}{2} \bigr)^{ 2 \sum_{j=1}^{L} k^{j}} \pnorm{\itr{F_i}{t+1}}{\mathrm{F}} P_2( \pnorm{ \itr{W_1}{t} \odot \itr{F_1}{t+1} }{\mathrm{F}}, \ldots, \pnorm{ \itr{W_L}{t} \odot \itr{F_L}{t+1} }{\mathrm{F}} )
       \label{eqn:Intermediate_bound_on_Delta_irlt__Nr1}
\end{align}
for $i=1, \ldots, L$ and some constants $C_8, C_9$ and polynomials $P_1(z_1,\ldots,z_L), P_2(z_1,\ldots,z_L)$, say, the latter both in $L$ variables.
Because of the projection and by definition of $\mathcal{H}$, there exists a constant $M$ such that $\pnorm{\itr{W_i}{t}}{\mathrm{F}} \leq M$ with probability one for all $i = 1, \ldots, L$, $t \in \naturalNumbersPlus$.
Furthermore, $\pnorm{ \itr{F_i}{t} }{\mathrm{F}} \leq \max_{i=0,\ldots,L-1} \sqrt{ d_i d_{i+1} }$ with probability one for all $i = 1, \ldots, L$, $t \in \naturalNumbersPlus$.
These two bounds, together with \refEquation{eqn:Intermediate_bound_on_Delta_irlt__Nr1} and the fact that $P_1, P_2$ are polynomials, as well as the hypothesis that $\expectation{ \pnorm{Y}{2}^m \pnorm{X}{2}^n } < \infty \, \forall m \in \{ 0, 1, 2 \}, n \in \naturalNumbersZero$, implies the result.
\BlackBox

\subsubsection{Conditional expectation of \texorpdfstring{$\itr{\Delta}{t+1}$}{the iterates} -- Proof \texorpdfstring{of Lemma~\ref{lem:Expectation_cond_Ft_of_Delta_i}}{}}
\label{secappendix:Expectation_cond_Ft_of_Delta_i}

Let $i \in \{ 1, \ldots, L\}$, $r \in \{ 1, \ldots, d_{i+1} \}$ and $l \in \{ 1, \ldots, d_{i} \}$.
Recall that $\mathcal{F}_t$ is the smallest $\sigma$-algebra generated by $ \{ \itr{W}{0}, ( \itr{F}{s}, \itr{X}{s}, \itr{Y}{s}) \}_{s \leq t}$, and note that $\itr{W}{t}$ is $\mathcal{F}_t$-measurable.
The (i) $\mathcal{F}_t$-measurability of $\itr{W}{t}$ together with the (ii) hypothesis that the sequences of random variables $\process{ ( \itr{F}{s}, \itr{X}{s}, \itr{Y}{s} ) }{ s \in \naturalNumbersPlus}$ is i.i.d.
\  implies that
\begin{align}
    \expectation{ \itr{\Delta_{i,r,l}}{t} | \mathcal{F}_t }
     &
    \eqcom{\ref{eqn:Random_direction_Delta_as_function_of_Backpropagation}}= \expectationBig{ \bigl( \itr{F_{i, r, l}}{t+1}  \mathrm{B}_{ \itr{F}{t+1} \odot \itr{W}{t} }( \itr{X}{t+1}, \itr{Y}{t+1} ) \bigr)_{i,r,l} \Big| \mathcal{F}_t }
    \nonumber \\                &
                     \eqcom{i,ii}= \int  F_{i, r, l} \mathrm{B}_{ F \odot \itr{W}{t} }( X, Y )_{i,r,l}  \d{ \probability{ \itr{F}{t+1} = F, \itr{X}{t+1} = X, \itr{Y}{t+1} = Y } }
    \nonumber \\                &
                     \eqcom{\ref{eqn:Backpropagation_calculates_the_gradient_operator_pointwise}}= \int \Bigl( F_{i, r, l}  \frac{ \partial l( \Psi_{F \odot \itr{V}{t}}(X), Y ) }{ \partial(F_{i, r, l}
                         V_{i, r, l}) } \Bigr)(\itr{W}{t}) \d{ \probability{ \itr{F}{t+1} = F, \itr{X}{t+1} = X, \itr{Y}{t+1} = Y } }
    \nonumber \\  &
       = \int \frac{ \partial l( \Psi_{F \odot \itr{V}{t}}(X), Y ) }{\partial V_{i, r, l} }(\itr{W}{t}) \d{ \probability{ \itr{F}{t+1} = F, \itr{X}{t+1} = X, \itr{Y}{t+1} = Y } }.
       \label{eqn:Prior_to_interchange_of_the_derivative_and_expectation}
\end{align}
Next, we need to check that we can exchange the derivative and expectation.
Note that we have the same assumptions $\expectation{ \pnorm{Y}{2}^m \pnorm{X}{2}^n } < \infty \, \forall m \in \{ 0, 1, 2 \}, n \in \naturalNumbersPlus$ as for Lemma~\ref{lem:Boundedness_of_the_variance_of_stochastic_iterands}.
as well as that $\sigma \in C^r_{\mathrm{PB}}(\realNumbers)$.
Therefore, by \eqref{eqn:Intermediate_bound_on_Delta_irlt__Nr1} in Lemma~\ref{lem:Boundedness_of_the_variance_of_stochastic_iterands} we have that $|\itr{\Delta_{i,r,l}}{t+1}|$ is upper bounded and moreover $\expectation{\itr{\Delta_{i,r,l}}{t+1}} \leq C_{\mathcal{H}}$ for some $C_{\mathcal{H}} \leq \infty$ only dependent on $\mathcal{H}$.
The interchange is then warranted by the dominated convergence theorem.
Hence continuing from \eqref{eqn:Prior_to_interchange_of_the_derivative_and_expectation}, we obtain
\begin{align}
    \expectation{ \itr{\Delta_{i,r,l}}{t} | \mathcal{F}_t }
     &
    = \frac{ \partial }{ \partial W_{i,r,l} } \int l( \Psi_{F \odot \itr{W}{t}}(X), Y ) \d{ \probability{ \itr{F}{t+1} = F, \itr{X}{t+1} = X, \itr{Y}{t+1} = Y } }
    \nonumber \\  &
       \eqcom{\ref{eqn:Dropouts_emperical_risk_function}}= \frac{ \partial \mathcal{D}(\itr{W}{t}) }{ \partial W_{i,r,l} }.
       \nonumber
\end{align}
If $\sigma \in C^r_{\mathrm{PB}}(\realNumbers)$, then for any multi-index $s$ on the set of weights, a bound similar to \eqref{eqn:Intermediate_bound_on_Delta_irlt__Nr1} holds by the chain rule:
\begin{align}
    \abs{ \partial^{s} l(Y, \Psi_{W \odot F} (X))}
     &
    \leq \pnorm{Y}{\mathrm{F}} P_{1,s}(\norm{W_1}_{\mathrm{F}}, \ldots, \norm{W_L}_{\mathrm{F}}, \{ \pnorm{X}{2}^j \}_{j=1}^{n_{s,1}})
    \nonumber \\  & \phantom{\leq}
       + P_{2,s}(\norm{W_1}_{\mathrm{F}}, \ldots, \norm{W_L}_{\mathrm{F}}, \{ \pnorm{X}{2}^j \}_{j=1}^{n_{s,2}} )
\end{align}
where $P_{1,s}, P_{2,s}$ are polynomials and $n_{s,1}, n_{s,2}$ are the top exponents in the expansion in $\pnorm{X}{\mathrm{F}}$.
Hence, using the assumption $\expectation{ \pnorm{Y}{2}^m \pnorm{X}{2}^n } < \infty \, \forall m \in \{ 0, 1, 2 \}, n \in \naturalNumbersPlus$, we obtain for any $W \in \mathcal{K} \subset \mathcal{W}$ a compact set that $\expectation{ \abs{\partial^{s} l(Y, \Psi_{W \odot F} (X))}} \leq C_{\mathcal{K}}$ .
In particular we can apply the dominated convergence theorem and conclude $\mathcal{D}(W) \in C^{r-1}(\mathcal{W})$ with $\partial^{s} \mathcal{D}(W) = \expectation{\partial^{s} l(Y, \Psi_{W \odot F} (X))}$.
\BlackBox

\subsubsection{Constant \texorpdfstring{$\mathcal{D}(W)$}{D(W)} on a critical set -- Proof \texorpdfstring{of Lemma~\ref{lem:Dw_is_constant_on_Si}}{}}

\label{secappendix:Dw_is_constant_on_Si}

\noindent
We use Sard's theorem \citep{sard1942measure} to prove Lemma~\ref{lem:Dw_is_constant_on_Si}, which gives sufficient conditions for condition (A2.6):

\begin{proposition}
    \citep{sard1942measure}
    Let $f:M \to N$ be a $f \in C^r$ map between manifolds with $\dim(M) = m$, $\dim(N) = n$.
    Let $\mathrm{Crit}(f)= \{ x \in M \albert{~:~} \nabla f(x) =0 \}$ be the set of critical points of $f$.
    If $r > m/n - 1$, then $f(\mathrm{Crit}(f))$ has measure zero.
    \label{prop:Sard}
\end{proposition}

\noindent
\emph{Proof of Lemma~\ref{lem:Dw_is_constant_on_Si}.}
By Lemma~\ref{lem:Expectation_cond_Ft_of_Delta_i}, we have $\mathcal{D}(W) \in C^r(\mathcal{W})$.
By assumption (N5) we have that if $W \in \partial \mathcal{H}$ and $\mathcal{D}(W) + \pi(W) = 0$, then $\mathcal{D}(W)=0$.
Furthermore $W \in S_j$ for some $j$, i.e., the critical points of $\mathcal{D}(W) + \pi(W)$ are $\{ W \in \mathcal{W} \mid \nabla \mathcal{D}(W) = 0\} \cap \mathcal{H}$.
We apply Sard's theorem (Proposition~\ref{prop:Sard}) to $\mathcal{D}(W) $.
We have that if $r \geq \dim(\mathcal{W})$, then $\mathcal{D}(S_i) \subseteq \R$ has measure zero.
Since $S_i$ is connected there is a continuous path $z_{a,b}:[0,1] \to S_i$ joining any two points $a,b \in S_i$.
By continuity of $\mathcal{D}(W)$ we must have then $\mathcal{D}(a)=\mathcal{D}(b)$, since otherwise we would have $[\mathcal{D}(a),\mathcal{D}(b)] \subseteq \mathcal{D}(S_i)$ which has positive measure in $\R$.
Therefore $\mathcal{D}(S_i)$ must be a constant.
\BlackBox

Remark that in Lemma~\ref{lem:Dw_is_constant_on_Si} the condition $r \geq \dim(\mathcal{W})$ cannot immediately be eliminated.
When $r < \dim(\mathcal{W})$, there are examples of functions which are not constant on their connected critical sets, see e.g.\ \cite{hajlasz2003whitney}.

\section{Proof of Propositions~\ref{prop:convergence_rate_generic_dropout} and \ref{prop:convergence_rate_generic_dropout_scaled}}
\label{sec:appendix_proof_generic_rate_dropout}
We use standard tools for proving convergence to an $\epsilon$-stationary point (for a reference, see \cite{bottou2018optimization}).
We require first the following bounds on the variance induced by dropout.

\begin{lemma}
    Assume that $F$ is a random variable satisfying (Q4).
    If $f$ is a vector of random variables with distribution $F$, then
    \begin{itemize}
        \item[(i)] $\expectationBig{\pnorm{f - \E[f]}{1}} = 2Np(1-p)$.
        \item[(ii)] $\expectationBig{\pnorm{f - \E[f]}{1}^2} = 2N^2p(1-p)$.
    \end{itemize}
    \label{lem:first_second_moment_bernoulli}
\end{lemma}
\begin{proof}
    We prove first (i).
    Recall that $f \in \{0,1\}^{N}$.
    If we denote by $f_i$ the $i$th entry of $f$, then note that from (Q4) $\probability{f_i = 1} = p$ and so $\expectation{|f_i - \E[f_i]|} = \expectation{|f_i - p|} = 2p(1-p)$.
    From linearity (i) follows.
    For (ii), we have
    \begin{align}
        \expectationBig{\pnorm{f - \E[f]}{1}^2} & = \sum_{i} \expectationBig{|f_i - p|^2} + \sum_{i \neq j} \expectationBig{|f_i - p||f_j - p|} \nonumber \\
                                                & \leq 2Np(1-p) + 2N(N-1)p(1-p) = 2N^2p(1-p),
    \end{align}
    where in the last inequality we have used the Cauchy--Schwartz inequality.
\end{proof}

In order to prove both Propositions~\ref{prop:convergence_rate_generic_dropout} and \ref{prop:convergence_rate_generic_dropout_scaled} simultaneously, we will temporarily redefine in this section $\mathsf{D}$ as
\begin{equation}
    \nabla \mathsf{D}(w) = \E[ cF \odot \nabla \mathsf{U}(W \odot cF)],
    \label{eqn:definition_D_generic_alternative}
\end{equation}
where $c > 0$ is a constant.
Later on we will specify both $c=1$ for Proposition~\ref{prop:convergence_rate_generic_dropout} and $c=1/p$ for Proposition~\ref{prop:convergence_rate_generic_dropout_scaled}, respectively.

\begin{lemma}
    Assume (Q3) and (Q4), that is, $\nabla \mathsf{U}$ is $\ell$-Lipschitz and the distribution of the filters is $\{0,1\}$-valued.
    Then, $\nabla \mathsf{D}$ is also $c^2\ell$-Lipschitz.
    \label{lem:lipschitz_D}
\end{lemma}
\begin{proof}
    Using (i) Jensen's inequality with the norm, we have for a fixed $w, s \in \mathcal{W}$ that
    \begin{align}
        \pnorm{\nabla \mathsf{D}(w) - \nabla \mathsf{D}(s)}{2} & = \pnorm{\expectationWrt{cf \odot \nabla \mathsf{U}(w \odot cf) - cf \odot \nabla \mathsf{U}(s \odot cf)}{f}}{2} \nonumber               \\
                                                               & \eqcom{i}\leq c\expectationBigWrt{\pnorm{f \odot \nabla \mathsf{U}(w \odot cf) - f \odot \nabla \mathsf{U}(s \odot cf)}{2}}{f} \nonumber \\
                                                               & \eqcom{ii} \leq c\expectationBigWrt{\pnorm{\nabla \mathsf{U}(w \odot cf) - \nabla \mathsf{U}(s \odot cf)}{2}}{f} \nonumber               \\
                                                               & \eqcom{iii} \leq c\ell \expectationBigWrt{\pnorm{w \odot cf - s \odot cf}{2}}{f} \nonumber                                               \\
                                                               & \eqcom{ii} \leq c^2\ell \expectationBigWrt{\pnorm{w - s}{2}}{f} = c^2\ell \pnorm{w - s}{2}
    \end{align}
    where we have also used (ii) the fact that for a vector $u$ and $\{0,1\}$-valued vector $f$ we have $\pnorm{f \odot u}{2} \leq \pnorm{u}{2}$, (iii) $\nabla \mathsf{U}$ is $\ell$-Lipschitz.
\end{proof}

The proof of the following lemma can be found in \refAppendixSection{sec:Appendix__proof_bound_variance}.
\begin{lemma}
    Assume (Q1)--(Q4), then for any $w \in \mathcal{W}$ with $\pnorm{w}{2} < R$, we have
    \begin{align}
        \expectationBig{\pnorm{\nabla \mathsf{D}(w) - cf \odot \nabla \mathsf{U}(w \odot cf)}{2}^2} & \leq c^2Np(1-p)\bigl(4S^2 + 6cN^2(\ell^2R^2 + 2c\ell R)\bigr).
    \end{align}
    \label{lem:bound_variance_generic_point}
\end{lemma}
We obtain in the next lemma a simple bound for the variance of the gradient that depends on the data.
\begin{lemma}
    Assume (Q1)--(Q4), then for any $w \in \mathcal{W}$, we have
    \begin{align}
        \expectationBigWrt{\pnorm{cf \odot \nabla \mathsf{U}(w \odot cf) - cf \odot \nabla r(w \odot cf, z)}{2}^2}{z,f} & \leq 4c^2pNS^2.
    \end{align}
    \label{lem:bound_variance_generic_point_data}
\end{lemma}
\begin{proof}
    We use first the definition of $\mathsf{U}$ as an expectation.
    We have
    \begin{align}
        \expectationBigWrt{\pnorm{cf \odot \nabla \mathsf{U}(w \odot cf) & - cf \odot \nabla r(w \odot cf, z)}{2}^2}{z,f} \nonumber                                                                                       \\
                                                                         & =c^2\expectationBigWrt{\pnorm{\expectationWrt{f \odot \nabla r(w \odot cf, z_1)}{z_1} - f \odot \nabla r(w \odot cf, z)}{2}^2}{z,f} \nonumber  \\
                                                                         & \leq c^2\expectationBigWrt{\expectationBigWrt{\pnorm{f \odot (\nabla r(w \odot cf, z_1) - \nabla r(w \odot cf, z))}{2}}{z_1}^2}{z,f} \nonumber \\
                                                                         & \eqcom{i} \leq c^2\expectationBigWrt{\expectationBigWrt{2S\pnorm{f}{2}}{z_1}^2}{z,f} \nonumber                                                 \\
                                                                         & \eqcom{ii}\leq c^2\expectationBigWrt{4S^2\pnorm{f}{2}^2}{z,f} = 4c^2pNS^2,
    \end{align}
    where in (i) we have used the upper bound for $\pnorm{\nabla r(w \odot cf, z)}{2}$ from (Q2) and in (ii) that since $f_i \in \{0,1\}$ for all $i \in [N]$, we have $\pnorm{f}{2}^2 = \pnorm{f}{1}$ so using linearity with (Q4) the bound follows.
\end{proof}

By (Q3)--(Q4) and Lemma~\ref{lem:lipschitz_D}, $\nabla \mathsf{D}$ is $c^2\ell$-Lipschitz.
In this case, we can then use the following common argument: if $\nabla \mathsf{D}$ is $c^2\ell$-Lipschitz then we have the inequality
\begin{equation}
    \mathsf{D}(\itr{W}{t+1}) \leq \mathsf{D}(\itr{W}{t}) + \langle \nabla \mathsf{D}(\itr{W}{t}), \itr{W}{t+1} - \itr{W}{t}\rangle + \frac{c^2\ell}{2}\pnorm{\itr{W}{t+1} - \itr{W}{t}}{2}^2.
\end{equation}
We can then use the definition of $\itr{W}{t+1}$ to write
\begin{align}
    \mathsf{D}(\itr{W}{t+1})
     & \leq \mathsf{D}(\itr{W}{t}) -\itrd{\alpha}{t} \langle \nabla \mathsf{D}(\itr{W}{t}), c\itr{F}{t+1} \odot \nabla r(\itr{W}{t} \odot c\itr{F}{t+1}, \itr{Z}{t+1})\rangle \nonumber \\
     & + \frac{c^2\ell(\itrd{\alpha}{t})^2}{2}\pnorm{c\itr{F}{t+1} \odot \nabla r(\itr{W}{t} \odot c\itr{F}{t+1}, \itr{Z}{t+1})}{2}^2.
    \label{eqn:generic_inter1}
\end{align}
Let $\mathcal{F}_t$ be the $\sigma$-algebra of $(\itr{W}{0}, \itr{F}{1}, \itr{Z}{1}, \ldots, \itr{W}{t}, \itr{F}{t}, \itr{Z}{t})$.
Conditional on $\mathcal{F}_t$, $\itr{F}{t+1} \odot \nabla r(\itr{W}{t} \odot \itr{F}{t+1}, \itr{Z}{t+1})$ is an unbiased estimator of $\nabla \mathsf{D}(\itr{W}{t})$ so that by linearity
\begin{align}
    \expectationBig{ \Bigl \langle \nabla \mathsf{D}(\itr{W}{t}), c\itr{F}{t+1} \odot \nabla r(\itr{W}{t} \odot c\itr{F}{t+1}, \itr{Z}{t+1}) \Bigr \rangle | \mathcal{F}_t} = \pnorm{\nabla \mathsf{D}(\itr{W}{t})}{2}^2.
    \label{eqn:generic_inter2}
\end{align}
Similarly to \eqref{eqn:generic_inter2}, we can decompose
\begin{align}
      & \expectationBig{\pnorm{c\itr{F}{t+1} \odot \nabla r(\itr{W}{t} \odot c\itr{F}{t+1}, \itr{Z}{t+1})}{2}^2 \Big| \mathcal{F}_t} \nonumber                                                                                                       \\
    = & \expectationBig{\pnorm{c\itr{F}{t+1} \odot \nabla r(\itr{W}{t} \odot c\itr{F}{t+1}, \itr{Z}{t+1}) - c\itr{F}{t+1} \odot \nabla \mathsf{U}(\itr{W}{t} \odot c\itr{F}{t+1}) \nonumber                                                          \\
      & \phantom{\itr{F}{t+1} \odot \nabla r(\itr{W}{t} \odot \itr{F}{t+1}, \itr{Z}{t+1}) +itr{Z}{t+1}) +} + c\itr{F}{t+1} \odot \nabla \mathsf{U}(\itr{W}{t} \odot c\itr{F}{t+1})}{2}^2 \Big| \mathcal{F}_t} \nonumber                              \\
      & = \expectationBig{\pnorm{c\itr{F}{t+1} \odot \nabla \mathsf{U}(\itr{W}{t} \odot c\itr{F}{t+1})}{2}^2 \Big| \mathcal{F}_t} \nonumber                                                                                                          \\
      & +\expectationBig{\pnorm{c\itr{F}{t+1} \odot \nabla \mathsf{U}(\itr{W}{t} \odot c\itr{F}{t+1}) - c\itr{F}{t+1} \odot \nabla r(\itr{W}{t} \odot c\itr{F}{t+1}, \itr{Z}{t+1})}{2}^2 \Big| \mathcal{F}_t} \nonumber                              \\
      & + 2\expectationBig{ \Bigl \langle c\itr{F}{t+1} \odot \nabla \mathsf{U}(\itr{W}{t} \odot c\itr{F}{t+1}) - c\itr{F}{t+1} \odot \nabla r(\itr{W}{t} \odot c\itr{F}{t+1},\itr{Z}{t+1}), \nonumber                                               \\
      & \phantom{\Bigl \langle \itr{F}{t+1} \odot \nabla \mathsf{U}(\itr{W}{t} \odot \itr{F}{t+1}) - \itr{F}{t+1} - \itr{F}{t+1}} c\itr{F}{t+1} \odot \nabla \mathsf{U}(\itr{W}{t} \odot c\itr{F}{t+1}) \Bigr \rangle \Big| \mathcal{F}_t} \nonumber \\
      & = \expectationBig{\pnorm{c\itr{F}{t+1} \odot \nabla \mathsf{U}(\itr{W}{t} \odot c\itr{F}{t+1})}{2}^2 \Big| \mathcal{F}_t} \nonumber                                                                                                          \\
      & +\expectationBig{\pnorm{c\itr{F}{t+1} \odot \nabla \mathsf{U}(\itr{W}{t} \odot c\itr{F}{t+1}) - c\itr{F}{t+1} \odot \nabla r(\itr{W}{t} \odot c\itr{F}{t+1}, \itr{Z}{t+1})}{2}^2 \Big| \mathcal{F}_t},
    \label{eqn:generic_inter2andand}
\end{align}
where in the last step the cross-term vanishes since, by using the independence assumption of $\itr{Z}{t+1}$ and $\itr{F}{t}$.
If we take the expectation with respect to $\itr{Z}{t+1}$ first, then we find
\begin{equation}
    \expectationWrt{c\itr{F}{t+1} \odot \nabla r(\itr{W}{t} \odot c\itr{F}{t+1}, \itr{Z}{t+1})| \mathcal{F}_t}{\itr{Z}{t+1} } = c\itr{F}{t+1} \odot \nabla \mathsf{U}(\itr{W}{t} \odot c\itr{F}{t+1}).
\end{equation}
Similarly, we can add and substract $\nabla \mathsf{D}(\itr{W}{t})$ in the first term and repeat the argument with the definitions of $\nabla \mathsf{U}$ and $\nabla \mathsf{D}$ in \eqref{eqn:definition_D_generic_alternative}, where we take the expectation of \eqref{eqn:generic_inter2andand} with respect to $\itr{F}{t+1}$ instead.
A similar cross-term vanishes.
We then obtain
\begin{align}
    \expectationBig{\pnorm{c\itr{F}{t+1} & \odot \nabla r(\itr{W}{t} \odot c\itr{F}{t+1}, \itr{Z}{t+1})}{2}^2 \Big| \mathcal{F}_t} \leq \pnorm{\nabla \mathsf{D}(\itr{W}{t})}{2}^2 \nonumber                                                      \\
                                         & +\expectationBig{\pnorm{\nabla \mathsf{D}(\itr{W}{t}) - c\itr{F}{t+1} \odot \nabla \mathsf{U}(\itr{W}{t} \odot c\itr{F}{t+1})}{2}^2 \Big| \mathcal{F}_t} \label{eqn:generic_inter2and}                 \\
                                         & +\expectationBig{\pnorm{c\itr{F}{t+1} \odot \nabla \mathsf{U}(\itr{W}{t} \odot c\itr{F}{t+1}) - c\itr{F}{t+1} \odot \nabla r(\itr{W}{t} \odot c\itr{F}{t+1}, \itr{Z}{t+1})}{2}^2 \Big| \mathcal{F}_t}.
    \nonumber
\end{align}
Define the constant $J_c = S^2 + \frac{3}
    {2}N^2c(\ell^2 R^2 + 2c \ell R)$ depending on $c$.
Using the bounds of Lemma~\ref{lem:bound_variance_generic_point} together with assumption (Q5) and Lemma~\ref{lem:bound_variance_generic_point_data} in \eqref{eqn:generic_inter2and} we obtain
\begin{equation}
    \expectationBig{\pnorm{\itr{F}{t+1} \odot \nabla r(\itr{W}{t} \odot \itr{F}{t+1}, \itr{Z}{t+1})}{2}^2 \Big| \mathcal{F}_t} \leq \pnorm{\nabla \mathsf{D}(\itr{W}{t})}{2}^2 + 4c^2pNS^2 + 4c^2Np(1-p)J_c.
    \label{eqn:generic_inter2andandand}
\end{equation}

Substitute now \eqref{eqn:generic_inter2} and \eqref{eqn:generic_inter2andandand} in \eqref{eqn:generic_inter1}.
After taking the expectation, we can use a telescopic sum in \eqref{eqn:generic_inter1} with the previous bounds, which yields
\begin{align}
    \sum_{t=1}^{T} \itrd{\alpha}{t}\Bigl( 1-\frac{c^2\ell \itrd{\alpha}{t}}{2} \Bigr) \expectationBig{\pnorm{\nabla \mathsf{D}(\itr{W}{t})}{2}^2} \leq \expectation{\mathsf{D}(\itr{W}{0} & )} -\expectation{\mathsf{D}(\itr{W}{T})} \nonumber                 \\
                                                                                                                                                                                          & + 2 c^4\ell Np(S^2 + (1-p)J_c)\sum_{t=1}^{T} (\itrd{\alpha}{t})^2.
\end{align}
By (Q1) we have $\expectation{\mathsf{D}(\itr{W}{0})} - \expectation{\mathsf{D}(\itr{W}{T})} \leq 2M$ independently of $c$.
Assuming that $\itrd{\alpha}{t} < \frac{1}{c^2\ell}$ for all $t \in [T]$, we then have
\begin{equation}
    \min_{t \in [T]} \expectationBig{\pnorm{\nabla \mathsf{D}(\itr{W}{t})}{2}^2} \leq \frac{4M + 4 \ell c^4 Np(S^2 + (1-p)J_c)\sum_{t=1}^{T} (\itrd{\alpha}{t})^2}{\sum_{t=1}^{T} \itrd{\alpha}{t}}.
    \label{eqn:generic_inter3}
\end{equation}

We not proceed with proving Propositions~\ref{prop:convergence_rate_generic_dropout} and \ref{prop:convergence_rate_generic_dropout_scaled}.
\\
\noindent
\emph{Proof of Propositions~\ref{prop:convergence_rate_generic_dropout} (a) and \ref{prop:convergence_rate_generic_dropout_scaled} :}
If $\itrd{\alpha}{t}= \eta$ is a constant in \eqref{eqn:generic_inter3}, we find
\begin{equation}
    \min_{t \in [T]} \expectationBig{\pnorm{\nabla \mathsf{D}(\itr{W}{t})}{2}^2} \leq \frac{4M + 4\eta^2\ell c^4 Np(S^2 + (1-p)J_c)}{T\eta}.
\end{equation}
Minimizing the bound over $\eta$ yields that the minimum occurs at $\eta^2 = M/(\ell Nc^4p(S^2 + (1-p)J_c)T)$.
For this $\eta$, the bound reads
\begin{equation}
    \min_{t \in [T]} \expectationBig{\pnorm{\nabla \mathsf{D}(\itr{W}{t})}{2}^2} \leq 4\sqrt{c^4p(S^2 + (1-p)J_c)}\sqrt{\frac{M \ell N}{T}}.
    \label{eqn:generic_final_1}
\end{equation}

For proving Proposition~\ref{prop:convergence_rate_generic_dropout} (a), set $c=1$ in \eqref{eqn:generic_final_1} as well as $J_c = J = S^2 + \frac{3}
    {2}N^2(\ell^2 R^2 + 2\ell R)$.
Note finally that the condition $\eta < 1/(c^2\ell) = 1/\ell$ is satisfied, for example, if $p > M\ell /(NS^2T)$.

For proving Propostion~\ref{prop:convergence_rate_generic_dropout_scaled}, set $c=1/p$ in \eqref{eqn:generic_final_1} as well as $J_{1/p} = S^2 + \frac{3}
    {2}N^2(\ell^2 R^2 + 2\ell/p)/p$.
Note finally that the condition $\eta < 1/(c^2\ell) = p^2/\ell$ is satisfied, for example, if $p > M\ell/(NS^2T)$.

\noindent
\emph{Proof of Proposition~\ref{prop:convergence_rate_generic_dropout} (b):}
Let $c=1$ and denote $J_1 = J$ in \eqref{eqn:generic_inter3}.
We can also set $\itrd{\alpha}{t} = 1/(\ell \sqrt{t})$.
It is easily verified that for $T \geq 4$:
\begin{equation}
    \sum_{t=1}^{T} \itrd{\alpha}{t}
    >
    \frac{\sqrt{T}}{\ell}
    \quad
    \textnormal{and}
    \quad
    \sum_{t=1}^{T} (\itrd{\alpha}{t})^2
    <
    \frac{\log(T)}{\ell^2}
    .
\end{equation}
Substituting these bounds in \eqref{eqn:generic_inter3} yields the result.
\BlackBox

\subsection{Proof of \texorpdfstring{Lemma~\ref{lem:bound_variance_generic_point}}{bounded variance}}
\label{sec:Appendix__proof_bound_variance}

Noting that we have temporarily the definition $\nabla \mathsf{D}(w)=\expectation{cf \odot \nabla \mathsf{U}(w \odot cf)}$ we can write
\begin{align}
     & \expectationBig{\pnorm{\nabla \mathsf{D}(w) - cf \odot \nabla \mathsf{U}(w \odot cf)}{2}^2} = \expectationBigWrt{\pnorm{\expectationWrt{cf_2 \odot \nabla \mathsf{U}(w \odot cf_2) - cf_1 \odot \nabla \mathsf{U}(w \odot cf_1)}{f_2}}{2}^2}{f_1} \nonumber \\
     & = \expectationBigWrt{\pnorm{\expectationWrt{cf_2 \odot \nabla \mathsf{U}(w \odot cf_2) - cf_2 \odot \nabla \mathsf{U}(w \odot cf_1) + \label{eqn:generic_lemma_inter1}                                                                                      \\
     & \phantom{------ f_2 \odot \nabla \mathsf{U}(w \odot f_1) +}+ cf_2 \odot \nabla \mathsf{U}(w \odot cf_1) - cf_1 \odot \nabla \mathsf{U}(w \odot cf_1)}{f_2}}{2}^2}{f_1} \nonumber                                                                            \\
     & \eqcom{i} \leq \expectationBigWrt{\expectationBigWrt{\pnorm{cf_2 \odot (\nabla \mathsf{U}(w \odot cf_2) - \nabla \mathsf{U}(w \odot cf_1)) + c(f_2 - f_1) \odot \nabla \mathsf{U}(w \odot cf_1)}{2}}{f_2}^2}{f_1} \nonumber                                 \\
     & \eqcom{ii} \leq \expectationBigWrt{\expectationBigWrt{\pnorm{cf_2 \odot (\nabla \mathsf{U}(w \odot cf_2) - \nabla \mathsf{U}(w \odot cf_1))}{2} + \pnorm{c(f_2 - f_1) \odot \nabla \mathsf{U}(w \odot cf_1)}{2}}{f_2}^2}{f_1}, \nonumber
\end{align}
where (i) we have used Jensen's inequality for a vector-valued random variable $v$, namely $\pnorm{\E[v]}{2} \leq \E[\pnorm{v}{2}]$, and (ii) the subadditivity of the norm $\pnorm{a + b}{2} \leq \pnorm{a}{2} + \pnorm{b}{2}$ for any $a,b \in \R^N$.
We now note that
\begin{align}
     &
    \pnorm{cf_2 \odot (\nabla \mathsf{U}(w \odot cf_2) - \nabla \mathsf{U}(w \odot cf_1))}{2}^2
    = \sum_{i} c^2f_2^i | \nabla_i \mathsf{U}(w \odot cf_2) - \nabla_i \mathsf{U}(w \odot cf_1)|^2
    \nonumber \\  &
       \eqcom{i}
       \leq
       \sum_{i} c^2f_2^i \ell^2\pnorm{w \odot cf_2 - w \odot cf_1}{2}^2
       \leq
       \sum_{i} c^4 f_2^i \ell^2\pnorm{f_2 - f_1}{2}^2 \pnorm{w}{2}^2
       \eqcom{ii}
       \leq
       c^4\pnorm{f_2}{1} \pnorm{f_2 - f_1}{1} \ell^2R^2
       ,
\end{align}
where we have used (i) the Lipschitzness assumption of $\nabla \mathsf{U}$ from (Q3), and (ii) the facts that $\pnorm{w}{2}^2 < R^2$ and $\pnorm{f_2}{2}^2 = \pnorm{f_2}{1}$.
The latter is true because for any vector $f$ with entries $\{-1,0,1\}$, $\pnorm{f}{2}^2 = \pnorm{f}{1}$.
We can reason similarly with $f_1-f_2$.

Using (Q2) we can also bound
\begin{equation}
    \pnorm{c(f_2 - f_1) \odot \nabla \mathsf{U}(w \odot f_1)}{2}^2 \leq c^2\pnorm{f_2 - f_1}{1} S^2.
\end{equation}
Hence, we have in \eqref{eqn:generic_lemma_inter1} that
\begin{align}
    \expectationWrt{\expectationWrt{\pnorm{cf_2 \odot (\nabla \mathsf{U}(w \odot cf_2) & - \nabla \mathsf{U}(w \odot cf_1))}{2} + \pnorm{c(f_2 - f_1) \odot \nabla \mathsf{U}(w \odot cf_1)}{2}}{f_2}^2}{f_1} \nonumber                         \\
                                                                                       & \leq \expectationWrt{\expectationWrt{c^2\pnorm{f_2}{1}^{1/2} \pnorm{f_2 - f_1}{1}^{1/2} \ell R + c\pnorm{f_2 - f_1}{1}^{1/2} S}{f_2}^2}{f_1} \nonumber \\
                                                                                       & \eqcom{i} \leq \expectationWrt{\expectationWrt{ c^2\pnorm{f_2 - f_1}{1}(c\pnorm{f_2}{1}^{1/2} \ell R + S)^2}{f_2}}{f_1} \nonumber                      \\
                                                                                       & \leq \expectationWrt{c^2\pnorm{f_2 - f_1}{1}(c^2\pnorm{f_2}{1} \ell^2R^2 + c\pnorm{f_2}{1}^{1/2} 2S\ell R + S^2)}{f_1, f_2} \nonumber                  \\
                                                                                       & \eqcom{ii} \leq \expectationWrt{c^2\pnorm{f_2 - f_1}{1}(c\pnorm{f_2}{1}
        (\ell^2R^2 + 2c \ell R) + S^2)}{f_1, f_2},
    \label{eqn:generic_lemma_inter2}
\end{align}
where (i) for a random variable $v$ we have $\E[v]^2 \leq \E[v^2]$ and (ii) $\pnorm{f_2}{1}^{1/2} \leq \pnorm{f_2}{1}$ since either $\pnorm{f_2}{1} = 0$ or $\pnorm{f_2}{1} \geq 1$.
We can now add an expectation term in the norm $\pnorm{f_2 - f_1}{1} \leq \pnorm{f_2 - \E[f_2]}{1} + \pnorm{f_1 - \E[f_1]}{1}$ and $\pnorm{f_2}{1} \leq \pnorm{f_2 - \E[f_2]}{1} + \pnorm{\E[f_2]}{1}$.
Here, $\pnorm{\E[f_2]}{1} = \pnorm{\E[f_1]}{1} = pN$ by (Q4).
Hence, from \eqref{eqn:generic_lemma_inter2} onward we can write
\begin{align}
     & \expectationWrt{c^2\pnorm{f_2 - f_1}{1}(c\pnorm{f_2}{1}
    (\ell^2R^2 + 2c \ell R) + S^2)}{f_1, f_2} \nonumber                                                                                                                                             \\
     & \leq \expectationBigWrt{c^2\bigl(\pnorm{f_2 - \E[f_2]}{1} + \pnorm{f_1 - \E[f_1]}{1}\bigr) \bigl( c\pnorm{f_2 - \E[f_2]}{1}(\ell^2R^2 + 2 c \ell R)(1 + pN) + S^2\bigr)}{f_1, f_2} \nonumber \\
     & = \expectationBigWrt{c^3\bigl( \pnorm{f_2 - \E[f_2]}{1}^2 + \pnorm{f_1 - \E[f_1]}{1}\pnorm{f_2 - \E[f_2]}{1}\bigr) (\ell^2R^2 + 2c\ell R)(1 + pN)}{f_1, f_2} \nonumber                       \\
     & \phantom{(\pnorm{f_2 - \E[f_2]}{1}^2 + \pnorm{f_1 - \E[f_1]}{1}\pnorm{f_2 - \E[f_2]}{1})(1 + pN)} + 2c^2S^2\expectationBigWrt{\pnorm{f_2 - \E[f_2]}{1}}{f_2} \nonumber                       \\
     & \eqcom{i} \leq \expectationBigWrt{c^3\bigl( \pnorm{f_2 - \E[f_2]}{1}^2 + \pnorm{f_1 - \E[f_1]}{1}\pnorm{f_2 - \E[f_2]}{1}\bigr)(\ell^2R^2 + 2c\ell R)(1 + pN)}{f_1, f_2} \nonumber           \\
     & \phantom{(\pnorm{f_2 - \E[f_2]}{1}^2 + \pnorm{f_1 - \E[f_1]}{1}\pnorm{f_2 - \E[f_2]}{1})(1 + pN)} + 4c^2S^2Np(1-p) \nonumber                                                                 \\
     & \eqcom{ii} \leq \expectationBigWrt{c^3\bigl( \pnorm{f_2 - \E[f_2]}{1}^2 + 4N^2p^2(1-p)^2\bigr) (\ell^2R^2 + 2c\ell R)(1 + pN)}{f_1, f_2} + 4c^2S^2Np(1-p) \nonumber                         \\
     & \eqcom{iii} \leq c^3\bigl( 2N^2p(1-p) + 4N^2p^2(1-p)^2\bigr) \bigl( \ell^2R^2 + 2c\ell R\bigr)\bigl(1 + pN \bigr) + 4c^2S^2Np(1-p) \nonumber                                                 \\
     & = c^2Np(1-p)\bigl( c\bigl( 2N + 4Np(1-p) \bigr)(1 + pN)(\ell^2R^2 + 2c\ell R) + 4S^2 \Bigr) \nonumber                                                                                        \\
     & \eqcom{iv} \leq c^2Np(1-p)(4S^2 + 6cN^2(\ell^2R^2 + 2c\ell R)),
\end{align}
where we have used (i) Lemma~\ref{lem:first_second_moment_bernoulli}(i), (ii) independence of $f_1$ from $f_2$ and Lemma~\ref{lem:first_second_moment_bernoulli}(i) again, (iii) Lemma~\ref{lem:first_second_moment_bernoulli}(ii), and (iv) bounded $1 + pN < 2N$ and $p(1-p) \leq 1/4$.
\BlackBox

 \section{Path representation of \texorpdfstring{$\mathcal{D}(W)$}{D(W)} -- Proofs of Lemma~\ref{lem:Path_representation_of_DW} and \refCorollary{cor:Path_representation_of_DW__Tree_case_arborescence}}
\label{sec:Appendix__Proof_of_the_path_representation_of_DW__Tree_case}

\noindent
\emph{Proof of \eqref{eqn:Path_representation_of_DW__Tree_case}.}
Recall that $G_F = ( \mathcal{E}_F, \mathcal{V})$ is a random subgraph of $G = ( \mathcal{E}, \mathcal{V} )$ with edge set $\mathcal{E}_F = \{ e \in \mathcal{E} : F_e = 1 \}$.
By (i) the law of total expectation, and by (ii) independence of $F$ and $(X,Y)$,
\begin{align}
    \mathcal{D}(W)
     &
    =
    \expectationBig{ \sum_{i=1}^{d_L} \bigl( Y_f - \sum_{ \gamma \in \Gamma^i(G) } P_\gamma F_{\gamma} X_{\gamma_0}\bigr)^2 }
    \nonumber \\  &
       \eqcom{i}= \sum_{g \in \mathcal{G}} \expectationBig{ \sum_{f=1}^{d_L} \bigl( Y_f - \sum_{ \gamma \in \Gamma^f(G_F) } P_\gamma X_{\gamma_0}\bigr)^2 \Big| \{ G_F = g \} } \probability{ G_F = g }
    \nonumber \\  &
       \eqcom{ii}= \sum_{ g \in \mathcal{G} } \mu_g \expectationBig{ \sum_{f=1}^{d_L} \bigl( Y_f - \sum_{ \gamma \in \Gamma^f(g) } P_\gamma X_{\gamma_0}\bigr)^2}.
       \label{eqn:Path_representation_of_DW__Tree_case_Proof_1}
\end{align}

\noindent
\emph{Proof of \eqref{eqn:Regularization_term_in_terms_of_paths}.}
Expand \eqref{eqn:Path_representation_of_DW__Tree_case_Proof_1} to find
\begin{equation}
    \mathcal{D}(W)
    = \sum_{ g \in \mathcal{G} } \mu_g \expectationBig{ \sum_{f=1}^{d_L} \Bigl( Y_f^2 - 2 Y_f \sum_{ \gamma \in \Gamma^f(g) } P_\gamma X_{\gamma_0} + \sum_{ \gamma \in \Gamma^f(g) } \sum_{ \delta \in \Gamma^f(g) } P_\gamma X_{\gamma_0} P_\delta X_{\delta_0} \Bigr)}.
\end{equation}
Setting $\eta_\gamma = \sum_{ \{ g \in \mathcal{G} | \gamma \in \Gamma(g) \} } \mu_g$, we obtain
\begin{align}
     &
    \mathcal{D}(W)
    = \sum_{ g \in \mathcal{G} } \mu_g \expectationBig{\Bigl( \sum_{f=1}^{d_L} \sum_{ \gamma \in \Gamma^f(g) } \Bigl( \frac{ Y_f^2 }{ \cardinality{ \Gamma^f(g) } } - 2 Y_f P_\gamma X_{\gamma_0} \Bigr) + \sum_{ \gamma \in \Gamma(g) } \sum_{ \delta \in \Gamma^{\gamma_L}(g) } P_\gamma X_{\gamma_0} P_\delta X_{\delta_0} \Bigr)}
    \\  &
       = \sum_{\gamma \in \Gamma(G)} \eta_{\gamma} \expectationBig{\bigl( Y_{\gamma_L} - P_\gamma X_{\gamma_0} \bigr)^2 }
    \nonumber \\  & \phantom{=}
       - \sum_{ g \in \mathcal{G} } \mu_g \expectationBig{ \sum_{f=1}^{d_L} \sum_{ \gamma \in \Gamma^f(g) } \Bigl( \Bigl( 1 - \frac{1}{ \cardinality{ \Gamma^{f}(g) } } \Bigr) Y_{f}^2 - P_\gamma X_{\gamma_0} \sum_{ \delta \in \Gamma^{f}(g) \backslash \{ \gamma \} } P_\delta X_{\delta_0} \Bigr)}
       \nonumber
\end{align}
after rearranging terms.
This completes Lemma~\ref{lem:Path_representation_of_DW}'s proof after identifying $\mathcal{J}(W)$ and $\mathcal{R}(W)$ here as the left and right sum, respectively.

To prove \refCorollary{cor:Path_representation_of_DW__Tree_case_arborescence}, consider that since for an arborescence $\mathcal{R}(W) = 0$, we can write
\begin{align}
     &
    \sum_{\gamma \in \Gamma(G)} \eta_{\gamma} \expectationBig{ \bigl( Y_{\gamma_L} - P_\gamma X_{\gamma_0} \bigr)^2 }
    \\  &
       = \sum_{\gamma \in \Gamma(G)} \eta_{\gamma} \expectation{X_{\gamma_0}^2} \Bigl( \frac{\expectation{Y_{\gamma_L} X_{\gamma_0}}}{\expectation{X_{\gamma_0}^2}} - P_\gamma \Bigr)^2 + \sum_{\gamma \in \Gamma(G)} \eta_{\gamma} \Bigl( \expectation{Y_{\gamma_L}^2} - \frac{\expectation{Y_{\gamma_L} X_{\gamma_0}}^2}{\expectation{X_{\gamma_0}^2}} \Bigr)
    \nonumber \\  &
       \eqcom{iii}= \mathcal{I}(W) + \mathcal{D}(\criticalpoint{W}).
       \nonumber
       \label{eqn:dropout_loss_function_arborescence}
\end{align}
Here, (iii) follows because since $\mathcal{I}(W) \geq 0$ and $\mathcal{I}(W) = 0$ at $z_\gamma = P_\gamma$, what remains must be the optimum.
This completes the proofs of Lemma~\ref{lem:Path_representation_of_DW} and \refCorollary{cor:Path_representation_of_DW__Tree_case_arborescence}.
\BlackBox

\section{Conserved quantities -- Proof \texorpdfstring{of Lemma~\ref{lem:Conserved_quantities_in_a_tree}}{}}
\label{sec:Appendix__Proof_of_Conserved_quantities_in_a_tree}

For any edge $f \in \mathcal{E}$,
\begin{align}
    W_f \frac{ \partial \mathcal{D} }{ \partial W_f }
     &
    \eqcom{\ref{eqn:Path_representation_of_DW__Tree_case}}= \sum_{ g \in \mathcal{G} } \mu_g \expectationBig{\sum_{e=1}^d 2 \bigl( Y_e - \sum_{ \gamma \in \Gamma^e(g) } P_\gamma X_{\gamma_0} \bigr) \bigl( \sum_{ \delta \in \Gamma^e(g;f) } P_\delta X_{\delta_0} \bigr) }
    \nonumber \\  &
       = \sum_{ g \in \mathcal{G} } \mu_g \expectationBig{\sum_{ \delta \in \Gamma(g;f) } 2 \bigl( Y_{\delta_L} - \sum_{ \gamma \in \Gamma^{\delta_L} (g) } P_\gamma X_{\gamma_0} \bigr) P_\delta X_{\delta_0} }.
\end{align}
Note that $\Gamma(g;l) = \Gamma^l(g)$ for any leaf $l \in \mathcal{L}(G)$ and $g \in \mathcal{G}$, and therefore in particular
\begin{equation}
    W_l \frac{ \partial \mathcal{D} }{ \partial W_l }
    = \sum_{ g \in \mathcal{G} } \mu_g \sum_{ \delta \in \Gamma^l(g) } \expectationBig{ 2 \bigl( Y_{\delta_L} - \sum_{ \gamma \in \Gamma^{\delta_L}(g) } P_\gamma X_{\gamma_0} \bigr) P_\delta X_{\delta_L} }.
\end{equation}

Recall that $\mathcal{L}(G;f)$ is the set of leaves of the subtree of the base graph $G$ rooted at $f \in \mathcal{E}$.
By the fact that $\{ \Gamma^l(g;f) \}_{ l \in \mathcal{L}(G;f) }$ partitions $\Gamma(g;f)$ for any $g \in \mathcal{G}$, viz.,
\begin{equation}
    \Gamma(g;f)
    = \cup_{ l \in \mathcal{L}(G;f) } \Gamma^l(g;f),
    \quad
    \Gamma^{l_1}(g;f) \cap \Gamma^{l_2}(g;f)
    = \emptyset
    \,
    \text{ for all } l_1 \neq l_2,
    \,
    g \in \mathcal{G},
\end{equation}
it follows that
\begin{equation}
    \sum_{ l \in \mathcal{L}(G;f) } W_l \frac{ \partial \mathcal{D} }{ \partial W_l }
    = W_f \frac{ \partial \mathcal{D} }{ \partial W_f }.
\end{equation}
Note in fact that this proof works for \emph{any} base graph $G$ that has no cycles and only length-$L$ paths, so not just an arborescence.
This is why we make Assumption (N6') as opposed to the stronger Assumption (N6) in \refCorollary{cor:Path_representation_of_DW__Tree_case_arborescence}.
\BlackBox

\section{Proof of Proposition~ \ref{prop:convergence_of_dropout_on_an_expanding_tree}}
\label{sec:Appendix__Convergence_of_GD}

\noindent
The proof of Proposition~\ref{prop:convergence_of_dropout_on_an_expanding_tree} is by double induction on the statements $ A(t) \equiv \{ \mathcal{I}( \itrd{W}{s} ) \leq \mathcal{I}( \itrd{W}{s-1} ) \e{ - 2 \nu_{\min} \kappa \alpha }, \forall
    s \in [t] \}$ and $
    B(t) \equiv \{ \itrd{W}{s} \in K, \forall
    s \in [t] \}$
where $\kappa > 0$ is a free parameter and $K$ is a compact set which we will define.
Concretely, we prove that there exist $\alpha$ and $\kappa$ such that $A(t) \cap B(t) \Rightarrow B(t+1)$ and $A(t) \cap B(t+1) \Rightarrow A(t+1)$.
\Cref{sec:Appendix__Double_induction} describes in detail how the upcoming Lemmas~\ref{lem:Compactness__Expanding_tree}--\ref{lem:Conserved_quantities_remain_bounded} provide sufficient conditions for the induction step.
There we also maximize the upper bound on the convergence rate over $\kappa$, which gives the rate in \eqref{eqn:Convergence_of_Dropout_tree_step_size}.

\albert{We start by giving Lemmas~\ref{lem:Compactness__Expanding_tree}--\ref{lem:Conserved_quantities_remain_bounded}.
    Recall first the definition of the set $B(\epsilon,I)$ in \eqref{eqn:definition_compact_set_wrt_I}.
}
Here, with a minor abuse of notation, we define also
\begin{equation}
    B(\epsilon, \{ C_f \}_{f \in \mathcal{E}\backslash \mathcal{L}(G)})
    \triangleq
    \bigl\{
W \in \mathcal{W} \big|
    \mathcal{I}(W) \leq \epsilon,
W_f^2 - \sum_{l \in \mathcal{L}(G; f)} W_{\gamma^l}^2 = C_f
    \bigr\}
\label{eqn:definition_compact_set_appendix}
\end{equation}
where $\{ \gamma^l \} \triangleq \Gamma^l(G)$ for $l \in \mathcal{L}(G)$ if $G$ is an arborescence.

\begin{lemma}
    \label{lem:Compactness__Expanding_tree}
    Assume (N2) from Proposition~\ref{prop:dropout_converges} and (N6) from \refCorollary{cor:Path_representation_of_DW__Tree_case_arborescence}.
    Then:
    \begin{itemize}[topsep=0pt,itemsep=-1ex,partopsep=1ex,parsep=1ex]
        \item[(i)] If $\epsilon > 0$ and $\abs{C_f} < \infty$ for $f \in \mathcal{E} \backslash \mathcal{L}(G)$, then the set $B(\epsilon, \{ C_f \}_{f \in \mathcal{E}\backslash \mathcal{L}})$ is compact.
        \item[(ii)] If $\max_{\gamma \in \Gamma(G)} \abs{z_{\gamma}} \leq M^{L}$, then the function $\mathcal{I}(W)$ is $\beta$-smooth in $\mathcal{S}$ with $\beta = 6 \nu_{\max} \cdot \allowbreak \abs{\mathcal{E}(G)} \abs{\Gamma(G)}
                M^{2(L-1)}$.
    \end{itemize}
\end{lemma}
Lemma~\ref{lem:Compactness__Expanding_tree} implies that $B(\epsilon,I)$ is compact and that $\mathcal{D}(W)$ is $\beta$-smooth on the compact set $K = \mathcal{S} \cap B(\epsilon, I)$, i.e.,
\begin{equation}
    \mathcal{D}(W^{\prime}) - \mathcal{D}(W)
    \leq \transpose{ \nabla \mathcal{D}(W) }( W^{\prime} - W ) + \beta \pnorm{ W^{\prime} - W }{2}^2
\end{equation}
for $W, W^{\prime} \in K$.
Its proof is deferred to \Cref{sec:Appendix__Compactness}.

Next, Lemma~\ref{lem:PL_inequality__Expanding_tree} gives a lower bound on the curvature of $\mathcal{D}(W)$ on $K$ in the direction of $\nabla \mathcal{D}(W)$, in the form of a \gls{PL}-inequality \citep{karimi2016linear}.
Its proof is in \Cref{sec:Appendix__PL_inequality}.

\begin{lemma}
    \label{lem:PL_inequality__Expanding_tree}
    Assume (N2) from Proposition~\ref{prop:dropout_converges} and (N6) from \refCorollary{cor:Path_representation_of_DW__Tree_case_arborescence}.
    If $\itrd{W}{t} \in \mathcal{S} \cap B(\epsilon, I)$, then
    \begin{equation}
        \pnorm{ \nabla \mathcal{D}(\itrd{W}{t}) }{2}^2
        \geq 4 \nu_{\min} (\itrd{C_{\min}}{t})^{(L-1)} \bigl( \mathcal{D}( \itrd{W}{t}) - \mathcal{D}(\criticalpoint{W}) \bigr)
        .
    \end{equation}
\end{lemma}

Lemma~\ref{lem:Conserved_quantities_remain_bounded} proves that the conserved quantities of the gradient flow remain bounded under the \gls{GD} algorithm in \eqref{eqn:iterates_gradient_descent}.
This lemma allows us to keep track of the iterates in the compact set $K=\mathcal{S} \cap B(\epsilon, I)$ by relating them to conserved quantities and exploiting the fact that under \gls{GD} $| \itrd{C_f}{t+1} - \itrd{C_f}{t} |$ has order $O(\alpha^2)$.
\Cref{sec:Appendix__PL_inequality} contains its proof.

\begin{lemma}
    \label{lem:Conserved_quantities_remain_bounded}
    Assume (N2) from Proposition~\ref{prop:dropout_converges} and (N6) from \refCorollary{cor:Path_representation_of_DW__Tree_case_arborescence}.
    If $\itrd{W}{t} \in \mathcal{S}$, and $\itrd{C_f}{t} > 0$ for all $f \in \mathcal{E} \backslash \mathcal{L}(G)$, then
    $
        4 \alpha^2 \norm{\nu}_1 M^{2(L-1)} \bigl( \mathcal{D}( \itrd{W}{t}) - \mathcal{D}(\criticalpoint{W}) \bigr)
        \geq | \itrd{C_f}{t+1} - \itrd{C_f}{t} |.
    $
\end{lemma}

\noindent
\emph{A note on the exchange of derivative and expectation in this section.}
Whenever we make both Assumption (N2) in Proposition~\ref{prop:dropout_converges} and (N7) in Lemma~\ref{lem:Path_representation_of_DW}, the exchange of derivative and expectation is warranted.
This occurs several times throughout this section.
We refer to the proof of Lemma~\ref{lem:Expectation_cond_Ft_of_Delta_i} for the details.

\subsection{Compactness, and smoothness -- Proof of \texorpdfstring{Lemma~\ref{lem:Compactness__Expanding_tree}}{}}
\label{sec:Appendix__Compactness}

In the proof of Lemma~\ref{lem:Compactness__Expanding_tree}, we will upper bound the operator norm of the Hessian.
Recall that for a symmetric bilinear matrix $A$, $\norm{A}_{\mathrm{op}} \triangleq \sup_{\norm{v}_{2} = 1} |v^T A v|$.
\\

\noindent
\emph{Proof of (i)}.
By continuity of the conditions in \eqref{eqn:definition_compact_set_wrt_I}, the set $B(\epsilon, \{ C_f \}_{f \in \mathcal{E}\backslash \mathcal{L}})$ is closed.
We need to prove boundedness.
Let $W \in B(\epsilon, \{ C_f \}_{f \in \mathcal{E}\backslash \mathcal{L}})$, and suppose w.l.o.g.
\ that for some $f^{*} \in \mathcal{E}\backslash \mathcal{L}$ we have $\abs{W_{f^{*}}} > Q$, where $Q > \max_{j \in \mathcal{E} \backslash \mathcal{L}, \gamma \in \Gamma(G)} \{ \abs{C_j}, \abs{z_{\gamma}}\}$.
We want to find a path $\gamma \in \Gamma(G)$ such that $P_{\gamma}$ is large for a contradiction with the assumption that $\mathcal{I}(W) \leq \epsilon$.
By \eqref{eqn:definition_first_integrals}, we have the inequality $\sum_{l \in \mathcal{L}(G;f^*)} W_{l}^2 > Q^2 - \abs{C_{f^*}}$ so that for some $l^* \in \mathcal{L}(G;f^*)$ we must have $W_{l^*}^2 > (Q-\abs{C_{f^*}})/\abs{\mathcal{L}(G;f^*)}$.
Consequently, we have by \eqref{eqn:definition_first_integrals} that $\abs{W_e}^2 > (Q^2-\abs{C_{f^*}})/{\abs{\mathcal{L}(G;f^*)}} - \abs{C_{e}}$ for any edge $e \in \gamma$ in any path $\gamma \in \Gamma^{l^*}(G)$ except for the edge $f^*$ where we have $\abs{W_{f^{*}}} > Q$ by assumption.
In particular, we have the bound $\abs{W_e} > O(Q)$ for any edge $e \in \gamma$ for any path $\gamma \in \Gamma(G;f^*)$.
Therefore if we pick $\gamma \in \Gamma(G;f^*)$ we have
\begin{equation}
    \epsilon
    \eqcom{\ref{eqn:definition_compact_set_wrt_I}}\geq \mathcal{I}(W)
    \geq \nu_{\gamma}(z_\gamma - P_{\gamma})^2
    \geq \nu_{\gamma}(\abs{P_\gamma} - \abs{z_\gamma})^2
    > O(Q^{2L})
\end{equation}
for sufficiently large $Q$, which is a contradiction.
We must thus have $| W_{f^*} | \leq Q$ for some $Q < \infty$.
If on the other hand $| W_{l} | > Q$ for some $l \in \mathcal{L}(G;f^*)$, by (\ref{eqn:definition_first_integrals}) we must also have $(W_{f^*})^2 > Q^2 + C_{f^{*}} > O(Q^2)$ for sufficiently large $Q$.
This case is, thus, the same as before.

\noindent
\emph{Proof of (ii)}.
Using a regular upper bound to the entries of $\nabla^2 \mathcal{I}(W)$ when $W \in \mathcal{S}$ will suffice.
Element-wise, we have
\begin{align}
     &
    ( \nabla^2 \mathcal{I}(W) )_{i,j}
    \\  &
       =
       \begin{cases}
        2\sum_{\delta \in \Gamma(G;i) \cap \Gamma(G;j)} \nu_{\delta} \Bigl( \frac{P_{\delta}}{W_i}\frac{P_{\delta}}{W_j} - \frac{P_{\delta}}{W_i W_j}(z_{\gamma} - P_{\gamma}) \Bigr),
         & \text{if} \enskip i \neq j, \Gamma(G;i) \cap \Gamma(G;j) \neq \emptyset,
        \\
        2\sum_{\gamma \in \Gamma(G;i)} \nu_{\gamma} (\frac{P_{\gamma}}{W_i})^2
         & \text{if} \enskip i=j,
        \\
        0
& \text{otherwise}.
    \end{cases}
       \nonumber
\end{align}
Hence, noting that since we have $\abs{W_f} \leq M$ for all $f \in \mathcal{E}$ on $\mathcal{S}$, we can bound $\abs{P_{\gamma}/ W_f} \leq M^{L-1}$, $\abs{z_{\gamma}} \leq M^{L}$ and the other terms similarly.
We upper bound the number of terms in the sum over $\Gamma(G;i)$ and $\Gamma(G;i) \cap \Gamma(G;j)$ by $| \Gamma(G) |$ and $\nu_{\gamma} \leq \nu_{\max}$.
Adding all terms, we obtain that $6 \nu_{\max} \abs{\Gamma(G)} M^{2(L-1)}$ is an upper bound for each of the entries of $\nabla^2 \mathcal{I}(W)$.
This gives an upper bound $\pnorm{ \nabla^2 \mathcal{I}(W) }{\mathrm{op}} \leq 6 |\mathcal{E}| \nu_{\max} \abs{\Gamma(G)} M^{2(L-1)}$ in $\mathcal{S}$.
\BlackBox

\subsection{\texorpdfstring{\gls{PL}}{PL}-inequality on a compact set -- Proof \texorpdfstring{of Lemma~\ref{lem:PL_inequality__Expanding_tree}}{}}
\label{sec:Appendix__PL_inequality}

Recall the definition of a \gls{PL}-inequality:

\glsreset{PL}
\begin{definition}
    Let $u \in C^2(K, \R)$ where $K \subset \R^n$ is compact and $K \backslash \partial K \neq \emptyset$.
    Denote by $u^{*} = \min_{x \in K} u(x)$ and suppose that $u^{*} \in K \backslash \partial K$.
    We say that $u$ satisfies a \emph{\gls{PL}} inequality if there exist a $\tau_{K} > 0$ depending only on $K$ such that
    \begin{equation}
        \norm{\nabla u(x)}_{2}^2
        \geq \tau_K (u(x) - u^{*})
        \quad
        \text{for all}
        \quad
        x \in K.
        \label{eqn:PL_inequality_general}
    \end{equation}
\end{definition}

A \gls{PL}-inequality together with $\beta$-smoothness on a compact set will imply that $\mathcal{D}(\itrd{W}{t})-\mathcal{D}(\criticalpoint{W})$ decreases.
To see this, note that by (i) $\beta$-smoothness, and (ii) the update rule
\begin{align}
    \mathcal{D}(\itrd{W}{t+1}) - \mathcal{D}(\itrd{W}{t})
     &
    \eqcom{i}\leq \transpose{ \nabla \mathcal{D}(\itrd{W}{t}) }( \itrd{W}{t+1} - \itrd{W}{t} ) + \beta \pnorm{ \itrd{W}{t+1} - \itrd{W}{t} }{2}^2
    \nonumber \\  &
       \eqcom{ii}= \alpha \bigl( \beta \alpha - 1 \bigr) \pnorm{ \nabla \mathcal{D}(\itrd{W}{t}) }{2}^2
\end{align}
If furthermore $\alpha \leq 1/(2\beta)$, then also $\beta \alpha - 1 \leq -1/2$.
Together with \eqref{eqn:PL_inequality_general}, and after rearranging terms, one finds that
\begin{equation}
    \mathcal{D}(\itrd{W}{t+1}) - \mathcal{D}(\itrd{W}{t}) \leq \frac{ \alpha \tau_K }{2}(\mathcal{D}(\itrd{W}{t}) - \mathcal{D}(\criticalpoint{W}))
    \quad
    \textnormal{for all}
    \quad
    W \in K.
    \label{eqn:Exponential_decay_due_to_PL_inequality_plus_beta_smoothness}
\end{equation}
By (iii) $1+x \leq \e{x}$ for all $x \in \realNumbers$, we obtain \eqref{eqn:PL_inequality_upper_bounded_and_iterated}.
The strategy will now be to prove that there is a \gls{PL}-inequality in some compact set, that the iterates remain in that compact set, and that the function is $\beta$-smooth.

\noindent
\emph{Proof of Lemma~\ref{lem:PL_inequality__Expanding_tree}.}
First note that if $l \in \mathcal{L}(G)$ and $\gamma \in \Gamma(G;l)$, the indexes of the weights in the product $| \itrd{P_{\gamma}}{t}/\itrd{W_{l}}{t} |$ belong to the index set $\mathcal{E} \backslash \mathcal{L}(G)$.
The proof follows (i) by restricting the sum, and (ii) from the fact that for every path $\gamma \in \Gamma(G)$ in an arborescence $G$, there is exactly one leaf $l\in \mathcal{L}(G)$ such that $\gamma^l = \gamma$.
Thus
\begin{align}
     &
    \sum_{e \in \mathcal{E}} \Bigl| \frac{\partial}{\partial W_e} \mathcal{I}(\itrd{W}{t}) \Bigr|^2
= 4\sum_{e \in \mathcal{E}} \Bigl|\sum_{\gamma \in \Gamma(G;e)} \nu_{\gamma} \frac{\itrd{P_{\gamma}}{t}}{\itrd{W_e}{t}}(z_{\gamma} - \itrd{P_{\gamma}}{t} ) \Bigr|^2
\eqcom{i} \geq 4\sum_{l \in \mathcal{L}(G)} \Bigl| \nu_{\gamma^l}\frac{\itrd{P_{\gamma^l}}{t}}{\itrd{W_{l}}{t}}( z_{\gamma^l} - \itrd{P_{\gamma^l}}{t}) \Bigr|^2
    \nonumber \\  &
       \eqcom{ii} = 4\sum_{\gamma \in \Gamma(G)} \nu_{\gamma}^2 \Bigl| \frac{\itrd{P_{\gamma}}{t}}{\itrd{W_{\gamma_L}}{t}}( z_{\gamma} - \itrd{P_{\gamma}}{t})\Bigr|^2
\eqcom{iii} \geq 4 \nu_{\min} \bigl( \min_{f \in \mathcal{E} \backslash \mathcal{L}(G)} |\itrd{W_f}{t}|^2 \bigr)^{L-1} \mathcal{I}(\itrd{W}{t}),
       \label{eqn:PL_inequality_appendix}
\end{align}
where in (iii) we have used the bound $|\itrd{W_i}{t}| \geq \min_{e \in \mathcal{E} \backslash \mathcal{L}(G)} |\itrd{W_e}{t}|$ for all $i \in \mathcal{E} \backslash \mathcal{L}(G)$ and similarly with $\nu_{\gamma} \geq \nu_{\min}$ for $\gamma \in \Gamma(G)$.
Finally, by \eqref{eqn:definition_first_integrals}, we have $\min_{e \in \mathcal{E} \backslash \mathcal{L}(G)} |\itrd{W_e}{t}|^2 \geq \itrd{ C_{\min} }{t}$.
This completes the proof.
\BlackBox

\subsection{Conserved quantities remain bounded throughout \texorpdfstring{\gls{GD}}{GD} -- Proof of Lemma~\ref{lem:Conserved_quantities_remain_bounded}}
\label{sec:Conserved_quantities_remain_bounded}

\begin{proof}
    Pick $f \in \mathcal{E} \backslash \mathcal{L}(G)$.
    By (i) \refCorollary{cor:Path_representation_of_DW__Tree_case_arborescence}, and (ii) Lemma~\ref{lem:Conserved_quantities_in_a_tree}, we have
    \begin{align}
         &
        \itrd{C_f}{t+1}
        = (\itrd{W_f}{t+1})^2 - \sum_{l \in \mathcal{L}(G;i)} (\itrd{W_{l}}{t+1})^2
        \nonumber \\  & \eqcom{\ref{eqn:iterates_gradient_descent}}= \Bigl( \itrd{W_f}{t} - \alpha \frac{\partial}{\partial W_f} \mathcal{D}(\itrd{W}{t}) \Bigr)^2 - \sum_{l \in \mathcal{L}(G;f)} \Bigl( \itrd{W_{l}}{t} - \alpha \frac{\partial}{\partial W_{l}} \mathcal{D}(\itrd{W}{t}) \Bigr)^2
        \nonumber \\  &
           \eqcom{i}= \Bigl( \itrd{W_f}{t} - \alpha \frac{\partial}{\partial W_f} \mathcal{I}(\itrd{W}{t}) \Bigr)^2 - \sum_{l \in \mathcal{L}(G;f)} \Bigl( \itrd{W_{l}}{t} - \alpha \frac{\partial}{\partial W_{l}} \mathcal{I}(\itrd{W}{t}) \Bigr)^2
        \nonumber \\  &
           \eqcom{ii}=\itrd{C_f}{t} + \alpha^2 \Bigl( \Bigl( \frac{\partial}{\partial W_f} \mathcal{I}(\itrd{W}{t})\Bigr)^2 - \sum_{l \in \mathcal{L}(G;f)} \Bigl( \frac{\partial}{\partial W_{l}} \mathcal{I}(\itrd{W}{t}) \Bigr)^2 \Bigr)
        \nonumber \\  &
           = \itrd{C_i}{t} + 4 \alpha^2 \Bigl( \Bigl( \sum_{\gamma \in \Gamma(G;f)} \nu_{\gamma}\frac{\itrd{P_{\gamma}}{t}}{\itrd{W_f}{t}} (z_{\gamma} - \itrd{P_{\gamma}}{t}) \Bigr)^2 - \sum_{l \in \mathcal{L}(G;f)} \nu_{\gamma^l}^2 \Bigl( \frac{\itrd{P_{\gamma^l}}{t}}{\itrd{W_{l}}{t}} \Bigr)^2 (z_{\gamma^l} - \itrd{P_{\gamma^l}}{t})^2 \Bigr) \label{eqn:step_for_recurrence}
        \\  &
           \geq \itrd{C_f}{t} - 4 \alpha^2 \Bigl( \sum_{l \in \mathcal{L}(G;f)} \nu_{\gamma^l}^2 \Bigl(\frac{\itrd{P_{\gamma^l}}{t}}{\itrd{W_{l}}{t}} \Bigr)^2 (z_{\gamma^l} - \itrd{P_{\gamma^l}}{t})^2 \Bigr).
           \label{eqn:step_for_recurrence_3}
    \end{align}

    By Cauchy--Schwartz we also have
    \begin{equation}
        \Bigl( \sum_{\gamma \in \Gamma(G;f)} \nu_{\gamma} \frac{\itrd{P_{\gamma}}{t}}{\itrd{W_f}{t}} (z_{\gamma} - \itrd{P_{\gamma}}{t}) \Bigr)^2 \leq \Bigl( \sum_{\gamma \in \Gamma(G;f)} \nu_{\gamma} \Bigr) \sum_{\gamma \in \Gamma(G;f)} \nu_{\gamma} \Bigl(\frac{\itrd{P_{\gamma}}{t}}{\itrd{W_l}{t}} \Bigr)^2 (z_{\gamma} - \itrd{P_{\gamma}}{t})^2.
        \label{eqn:step_for_recurrence_2}
    \end{equation}

    If we have $\itrd{C_f}{t} > 0$, then $(\itrd{W_f}{t})^2 > (\itrd{W_{\gamma_L}}{t})^2$ for any $\gamma \in \Gamma(G;f)$.
    Thus, combining the estimate (\ref{eqn:step_for_recurrence}) with (\ref{eqn:step_for_recurrence_2}) we obtain
    \begin{equation}
        \itrd{C_f}{t+1} \leq \itrd{C_f}{t} + 4 \Bigl( \sum_{\gamma \in \Gamma(G;f)} \nu_{\gamma} \Bigr)\alpha^2 \Bigl( \sum_{l \in \mathcal{L}(G;f)} \nu_{\gamma^l} \Bigl(\frac{\itrd{P_{\gamma^l}}{t}}{\itrd{W_{l}}{t}} \Bigr)^2 (z_{\gamma^l} - \itrd{P_{\gamma^l}}{t})^2 \Bigr).
        \label{eqn:step_for_recurrence_4}
    \end{equation}
    Extending the sums in \eqref{eqn:step_for_recurrence_4} from $\Gamma(G;f)$ to $\Gamma(G)$ and from $\mathcal{L}(G;f)$ to $\mathcal{L}(G)$, respectively, yields
    \begin{equation}
        \itrd{C_f}{t+1} - \itrd{C_f}{t}
        \leq 4 \norm{\nu}_1 \alpha^2 \bigl( \max_{e \in \mathcal{E} \backslash \mathcal{L}(G)} |\itrd{W_e}{t}|^2 \bigr)^{L-1} \mathcal{I}(\itrd{W}{t}),
        \label{eqn:bounds_first_integral_above}
    \end{equation}
    where we have used the bound $\abs{W_f} \leq \max_{e \in \mathcal{E} \backslash \mathcal{L}(G)} \abs{W_e}$ for all $f \in \mathcal{E} \backslash \mathcal{L}(G)$.
    Similarly, using (\ref{eqn:step_for_recurrence_3}) and the trivial bound $\nu_{\gamma} \leq \norm{\nu}_1$ for any $\gamma \in \Gamma$, and by absorbing one $\nu_\gamma$-term into $\mathcal{I}(W)$'s expression, we obtain
    \begin{equation}
        \itrd{C_f}{t+1} \geq \itrd{C_f}{t} - 4 \norm{\nu}_1 \alpha^2 \bigl( \max_{e \in \mathcal{E} \backslash \mathcal{L}(G)} |\itrd{W_e}{t}|^2 \bigr)^{L-1}\mathcal{I}(\itrd{W}{t})
        \label{eqn:bounds_first_integral_below}
    \end{equation}
    for the lower bound.
    Because $\itrd{W}{t} \in \mathcal{S}$ by assumption, $\max_{e \in \mathcal{E} \backslash \mathcal{L}(G)} |\itrd{W_e}{t}|^2 \leq M^2$.
    This completes the proof.
\end{proof}

\subsection{Double induction}
\label{sec:Appendix__Double_induction}

We now use Lemmas~\ref{lem:Compactness__Expanding_tree}--\ref{lem:Conserved_quantities_remain_bounded} together in a double induction to finally prove Proposition~\ref{prop:convergence_of_dropout_on_an_expanding_tree}.
Let $\kappa > 0$ and denote the statements:
\begin{align}
    A(t) & \equiv \{ \mathcal{I}( \itrd{W}{s} ) \leq \mathcal{I}( \itrd{W}{s-1} ) \e{ - 2 \nu_{\min} \kappa \alpha }, \forall
    s \in [t] \}, \label{eqn:double_induction_first_statement}
    \\
    B(t) & \equiv \{ \itrd{W}{s} \in B(\epsilon,I) \cap \mathcal{S} \, \forall
    s \in [t] \}.
    \label{eqn:double_induction_second_statement}
\end{align}
We will prove that there exists a $\kappa > 0$ such that when choosing $\alpha$ appropriately, firstly
\begin{equation}
    A(t) \cap B(t) \Rightarrow B(t+1),
    \label{eqn:Induction_proof__First_step}
\end{equation}
and secondly,
\begin{equation}
    A(t) \cap B(t+1) \Rightarrow A(t+1).
    \label{eqn:Induction_proof__Second_step}
\end{equation}

\noindent
\emph{Step 1: $A(t) \cap B(t) \Rightarrow B(t+1)$.}
We need to prove that $\itrd{W}{t+1} \in B(\epsilon, I) \cap \mathcal{S}$ assuming \eqref{eqn:double_induction_first_statement} and (\ref{eqn:double_induction_second_statement}).
Using \eqref{eqn:bounds_first_integral_above} from the proof of Lemma~\ref{lem:Conserved_quantities_remain_bounded} repeatedly with the bound $\max_{e \in \mathcal{E}} |\itrd{W_e}{t}| \leq M$, we obtain
\begin{equation}
    \itrd{C_f}{t+1} \leq \itrd{C_f}{0} + 4 \norm{\nu}_1 M^{2(L-1)}\alpha^2\sum_{s=0}^t \mathcal{I}(\itrd{W}{s}).
    \label{eqn:recurrence_repeated_upper}
\end{equation}
By \eqref{eqn:double_induction_first_statement}, we can upper bound
\begin{equation}
    \sum_{s=0}^{t} \mathcal{I}(\itrd{W}{s})
    \eqcom{\ref{eqn:double_induction_first_statement}}\leq \sum_{s=0}^{t} \mathcal{I}( \itrd{W}{0})\exp(-2\nu_{\min}\kappa \alpha s) \leq \mathcal{I}( \itrd{W}{0}) \frac{1}{ 1 - \e{ -2\nu_{\min}\kappa \alpha } }.
    \label{eqn:bound_geometric_sum_double_induction}
\end{equation}
If furthermore (C1) $0 < 2 \nu_{\min} \kappa \alpha < 1$, then (i) the inequality $ 1/(1-\exp(-2\nu_{\min} \kappa \alpha)) < 1/(\nu_{\min}\kappa \alpha)$ holds, so that
\begin{equation}
    \itrd{C_{\min}}{t+1}
    \eqcom{\ref{eqn:recurrence_repeated_upper}}\leq \itrd{C_{\min}}{0} + 4 \norm{\nu}_1 M^{2(L-1)}\alpha^2 \sum_{s=0}^{t} \mathcal{I}(\itrd{W}{s})
    \eqcom{i}\leq \itrd{C_{\min}}{0} + 4 \frac{\norm{\nu}_1}{\nu_{\min}}M^{L-1}\alpha \kappa^{-1} \mathcal{I}(\itrd{W}{0}).
    \label{eqn:proof_step}
\end{equation}
In the same manner, we can also prove \eqref{eqn:proof_step} for $\itrd{C_f}{0}$ instead of $\itrd{C_{\min}}{0}$.
This yields
\begin{equation}
    \itrd{C_f}{t+1} \leq \itrd{C_f}{0} + 4\frac{\norm{\nu}_1}{\nu_{\min} \kappa} M^{2(L-1)}\alpha \mathcal{I}(\itrd{W}{0})
    \label{eqn:set_I_recurrence_repeated_upper}
\end{equation}
for any $f \in \mathcal{E} \backslash \mathcal{L}(G)$.
Similarly, for a lower bound, we can use \eqref{eqn:bounds_first_integral_below} repeatedly together with the bound \eqref{eqn:bound_geometric_sum_double_induction} and condition (C1) yielding
\begin{equation}
    \itrd{C_f}{t+1} \geq \itrd{C_f}{0} - 4 \frac{\norm{\nu}_1}{\nu_{\min} \kappa} M^{2(L-1)}\alpha \mathcal{I}(\itrd{W}{0}).
    \label{eqn:set_I_recurrence_repeated_lower}
\end{equation}
for any $f \in \mathcal{E} \backslash \mathcal{L}(G)$.
Now, suppose (D1) $\itrd{C_{\min}}{0} - \kappa^{1/(L-1)} > 0$ and let (C2) the step size satisfy
\begin{equation}
    \alpha \leq \nu_{\min} \kappa \frac{\itrd{C_{\min}}{0} - \kappa^{1/(L-1)}}{8\norm{\nu}_1M^{2(L-1)} \mathcal{I}( \itrd{W}{0})}.
    \label{eqn:proof_stepsize_bound2}
\end{equation}
We have (i) by \eqref{eqn:set_I_recurrence_repeated_upper} and \eqref{eqn:set_I_recurrence_repeated_lower} that
\begin{align}
    \itrd{C_f}{t+1}
     &
    \eqcom{i}\in [\itrd{C_f}{0} - 4 \frac{\norm{\nu}_1}{\nu_{\min}} M^{2(L-1)}\alpha \kappa^{-1} \mathcal{I}(\itrd{W}{0}), \itrd{C_f}{0} + 4 \frac{\norm{\nu}_1}{\nu_{\min}} M^{2(L-1)}\alpha \kappa^{-1} \mathcal{I}(\itrd{W}{0})]
    \nonumber \\  &
       \eqcom{\ref{eqn:proof_stepsize_bound2}} \subseteq [\itrd{C_f}{0} - \frac{1}{2}(\itrd{C_{\min}}{0} - \kappa^{1/(L-1)}), \itrd{C_f}{0} + \frac{1}{2}(\itrd{C_{\min}}{0} - \kappa^{1/(L-1)})]
    \nonumber \\  &
       \eqcom{D1} \subseteq [\itrd{C_f}{0} - \itrd{C_f}{0}/2, \itrd{C_f}{0} + \itrd{C_f}{0}/2]
\subseteq [\itrd{C_f}{0}/2, 3\itrd{C_f}{0}/2]
       = I_f.
\end{align}
Then $\itrd{W}{t+1} \in B(\epsilon, I)$ by \eqref{eqn:definition_compact_set_wrt_I}.
Hence, $M > \itrd{W_f}{t+1} \eqcom{\ref{eqn:definition_first_integrals}} > ( \itrd{C_f}{0}/2 )^{1/2} \geq ( \itrd{C_{\min}}{0}/2 )^{1/2} > \delta$ for any $f \in \mathcal{E} \backslash \mathcal{L}(G)$.
Since moreover $\itrd{C_e}{t+1} > 0$ for all $e \in \mathcal{E} \backslash \mathcal{L}(G)$, we have that if $f \in \mathcal{L}(G)$, then $M^2 > (\itrd{W_j}{t+1})^2 > (\itrd{W_f}{t+1})^2$ for some $j \in \mathcal{E} \backslash \mathcal{L}(G)$.
Consequently $M \geq |\itrd{W_f}{t+1}|$ and $\itrd{W}{t+1} \in \mathcal{S}$.

\noindent
\emph{Step 2: $A(t) \cap B(t+1) \Rightarrow A(t+1)$.}
Suppose that $\itrd{W}{s} \in B(\epsilon, I) \cap \mathcal{S}$ for $s = 0, 1, \ldots, t+1$.
Using the bound in \eqref{eqn:set_I_recurrence_repeated_upper} which requires the induction hypothesis $A(t)$ and (C1) for $\itrd{C_{\min}}{t}$, we obtain
\begin{equation}
    \itrd{C_{\min}}{t} \geq \itrd{C_{\min}}{0} - 4 \frac{\norm{\nu}_1}{\nu_{\min} \kappa} M^{2(L-1)}\alpha \mathcal{I}(\itrd{W}{0}).
    \label{eqn:recurrence_repeated_lower}
\end{equation}
Suppose now for a moment that (C2) the right-hand side of \eqref{eqn:recurrence_repeated_lower} is positive for some sufficiently small $\alpha$.
We could then use the \gls{PL} inequality from Lemma~\ref{lem:PL_inequality__Expanding_tree} together with $\min_{e \in \mathcal{E} \backslash \mathcal{L}(G)} |\itrd{W_e}{t}|^{2(L-1)} \allowbreak \geq (\itrd{C_{\min}}{t})^{L-1}$, that is,
\begin{equation}
    \pnorm{\nabla\mathcal{I}( \itrd{W}{t})}{2}^2
    \geq 4\nu_{\min}(\itrd{C_{\min}}{t})^{L-1} \mathcal{I}( \itrd{W}{t}).
    \label{eqn:PL_inequality_appendix_induction}
\end{equation}
To see how, note that the argumentation around \eqref{eqn:Exponential_decay_due_to_PL_inequality_plus_beta_smoothness} together with \eqref{eqn:PL_inequality_appendix_induction} and (i) the induction hypothesis $B(t+1)$ we have $\itrd{W}{t}, \itrd{W}{t+1} \in B(\epsilon,I) \cap \mathcal{S}$ and (ii) the clause (L1) $\alpha \leq 1/(2\beta)$, implies
\begin{align}
     &
    \mathcal{I}( \itrd{W}{t+1})
    \eqcom{i,ii, \ref{eqn:PL_inequality_appendix_induction}}\leq \mathcal{I}( \itrd{W}{t})\exp\bigl( -2 \nu_{\min} \alpha(\itrd{C_{\min}}{t})^{L-1} \bigr) \nonumber \\  &
       \eqcom{\ref{eqn:recurrence_repeated_lower}}\leq \mathcal{I}( \itrd{W}{t}) \exp{ \Bigl( -2\nu_{\min} \alpha \bigl( \itrd{C_{\min}}{0} - 4 \frac{\norm{\nu}_1}{\nu_{\min} \kappa} M^{2(L-1)}\alpha \mathcal{I}(\itrd{W}{0}) \Bigr) }
    \nonumber                                                                                                                                                        \\  &
       \eqcom{iii}\leq \mathcal{I}( \itrd{W}{0}) \exp{ \Bigl( -2\nu_{\min} \alpha \bigl( \itrd{C_{\min}}{0} - 4 \frac{\norm{\nu}_1}{\nu_{\min} \kappa} M^{2(L-1)}\alpha \mathcal{I}(\itrd{W}{0})\bigr)^{L-1} - 2 \nu_{\min} \alpha\kappa t\Bigr) }
       \label{eqn:Intermediate_bound_on_DWtp1_deterministic}
\end{align}
where we have also used (iii) the induction hypothesis $A(t)$, i.e., that $\mathcal{I}( \itrd{W}{t}) \leq \mathcal{I}( \itrd{W}{0}) \cdot \allowbreak \e{ -2\nu_{\min}\kappa \alpha t}$.

We now investigate the exponent in \eqref{eqn:Intermediate_bound_on_DWtp1_deterministic} for a moment.
Assuming (C2) and if (C3) the right-hand side of \eqref{eqn:Intermediate_bound_on_DWtp1_deterministic} is furthermore smaller than $\mathcal{I}( \itrd{W}{0}) \exp(- 2 \nu_{\min}\kappa \alpha (t +1))$, then the induction step would be complete.
Note finally that both conditions (C2) and (C3) are satisfied when choosing
\begin{equation}
    \kappa \leq \bigl( \itrd{C_{\min}}{0} - 4 \frac{\norm{\nu}_1}{\nu_{\min}}M^{2(L-1)}\alpha \kappa^{-1} \mathcal{I}(\itrd{W}{0})\bigr)^{L-1}
    \label{eqn:proof_lower_bound_interval}
\end{equation}
or equivalently
\begin{equation}
    \alpha \leq \nu_{\min} \kappa \frac{\itrd{C_{\min}}{0} - \kappa^{1/(L-1)}}{4\norm{\nu}_1M^{2(L-1)} \mathcal{I}( \itrd{W}{0})}.
    \label{eqn:proof_stepsize_bound}
\end{equation}
To also satisfy condition (C1), we thus require that
\begin{equation}
    \alpha \leq \min \Bigl( \frac{1}{2 \nu_{\min} \kappa}, \nu_{\min} \kappa \frac{\itrd{C_{\min}}{0} - \kappa^{1/(L-1)}}{4\norm{\nu}_1M^{2(L-1)} \mathcal{I}( \itrd{W}{0})} \Bigr).
\end{equation}

\noindent
\emph{Step 3.}
Let us summarize.
Convergence occurs at rate at most $2\nu_{\min} \kappa \alpha$ if conditions (L1), (D1), (C1)--(C3) hold.
Hence we have to choose $\kappa > 0$ such that $\itrd{C_{\min}}{0} - \kappa^{L-1} > 0$ and
\begin{equation}
    \alpha \leq \min \Bigl( \nu_{\min} \kappa \frac{\itrd{C_{\min}}{0} - \kappa^{1/(L-1)}}{8\norm{\nu}_1 M^{2(L-1)} \mathcal{I}( \itrd{W}{0})}, \frac{1}{2\beta}, \frac{1}{2 \nu_{\min} \kappa} \Bigr).
    \label{eqn:stepsize_prefinal}
\end{equation}

Note that we can maximize the convergence rate $2\nu_{\min}\alpha \kappa$ by maximizing
$
    \kappa^2(\itrd{C_{\min}}{0} - \kappa^{1/(L-1)})
$,
which occurs when $\kappa = (\itrd{C_{\min}}{0})^{L-1}(1+1/(2(L-1)))^{-(L-1)} \geq \e{-1/2}(\itrd{C_{\min}}{0})^{L-1}$.
Substituting this in \eqref{eqn:stepsize_prefinal} we require a step size
\begin{equation}
    \alpha \leq \min \Bigl( \nu_{\min} \frac{ \e{1/2} (\itrd{C_{\min}}{0})^{L} }{8\norm{\nu}_1 (2L-1) M^{2(L-1)} \mathcal{I}( \itrd{W}{0})}, \frac{1}{2\beta}, \frac{1}{2 \nu_{\min} (\itrd{C_{\min}}{0})^{L-1}} \Bigr).
\end{equation}
Finally, we have the bound $\beta \leq 6 \nu_{\max} \abs{\mathcal{E}(G)} \abs{\Gamma(G)} M^{2(L-1)}$ from Lemma~\ref{lem:Compactness__Expanding_tree} in $\mathcal{S}$, so that
\begin{equation}
    \alpha \leq \min \Bigl( \nu_{\min} \frac{ \e{1/2}(\itrd{C_{\min}}{0})^{L} }{16\norm{\nu}_1 L M^{2(L-1)} \mathcal{I}( \itrd{W}{0})}, \frac{1}{12 \nu_{\max} \abs{\mathcal{E}(G)} \abs{\Gamma(G)} M^{2(L-1)}}, \frac{1}{2 \nu_{\min} (\itrd{C_{\min}}{0})^{L-1}} \Bigr).
    \label{eqn:stepsize_final}
\end{equation}
This completes our proof of Proposition~\ref{prop:convergence_of_dropout_on_an_expanding_tree}.
\BlackBox
 \section{Convergence rate in the case of \emph{Dropout} and \emph{Dropconnect} -- Proof of Proposition~\ref{prop:estimate_nu_dependent_p}}
\label{sec:Appendix__Corollary}

We consider first the case of \emph{Dropconnect}.
We have that $\{ F_e \}_{e \in \mathcal{E}}$ are independent and identically distributed $\mathrm{Bernoulli}(p)$ random variables.
Suppose that the base graph $G$ has no cycles and every path is of length $L$.
Then by definition in Lemma~\ref{lem:Path_representation_of_DW}, we have
\begin{align}
    \eta_{\gamma}
     &
    = \sum_{ \{g \in \mathcal{G}| \gamma \in \Gamma(g) \} } \probability{G_F = g} = \sum_{g \in \mathcal{G}} \indicator{\gamma \in \Gamma(g)} \probability{G_F = g}
    \nonumber \\  &
       = \sum_{g \in \mathcal{G}} \probability{ \gamma \in \Gamma(g)| G_F = g} \probability{G_F = g} = \probability{ \gamma \in \Gamma(G_F)}
       \eqcom{i}= p^L
       \label{eqn:probability_graph_dropconnect}
\end{align}
where (i) we have used \emph{Dropconnect}'s distribution on $F$.

Now suppose that additionally we make the stronger assumption that $G$ is an arborescence.
Then by definition in \refCorollary{cor:Path_representation_of_DW__Tree_case_arborescence} $\nu_{\gamma} = \expectation{X^2} \eta_{\gamma}$, and subsequently we can calculate $\norm{\nu}_1 = \expectation{X^2} \sum_{\gamma \in \Gamma(G)} \nu_{\gamma} = \expectation{X^2} \abs{\Gamma(G)} p^L = \expectation{X^2} d_L p^L$.

Now, since by assumption $\max_{\gamma} \abs{z_{\gamma}} \leq M^L$ and $\abs{W_f} \leq M$ for all $f \in \mathcal{E}$, then $\mathcal{I}(\itrd{W}{0}) \leq O(\abs{\Gamma(G)} M^{2L})$ so that substitution of in the definition of $\alpha$ in Proposition~\ref{prop:convergence_of_dropout_on_an_expanding_tree} yields
\begin{equation}
    \alpha = O \Bigl( \frac{(\itrd{C_{\min}}{0})^{L}}{L M^{4L}} \Bigr),
\end{equation}
where we have used that $C_{\min} \leq M^2$.
Finally multiplying by $\tau$ gives the rate
\begin{equation}
    \alpha \tau = O \Bigl( \frac{p^L (\itrd{C_{\min}}{0})^2L}{ L (d_L)^2 M^{4L} } \Bigr).
\end{equation}
Substituting these results in the rate $\tau \alpha$ in Proposition~\ref{prop:convergence_of_dropout_on_an_expanding_tree} yields the result for \emph{Dropconnect}.

Finally we note that for the case of \emph{Dropout}, filtering all nodes independently in an arborescence is equivalent to filtering all edges independently except the edge at the root.
In particular, in \eqref{eqn:probability_graph_dropconnect}, we have $\probability{ \gamma \in \Gamma(G_F)}
    = p^{L-1}$.
The remaining steps of the proof are then the same as for \emph{Dropconnect} and comparing $p^{L}$ with $p^{L-1}$ we can absorb the missing $p$ factor into the $O$ notation, which does not change the order in $L$.
\BlackBox

\section{Inequalities pertaining to the Frobenius norm}
\label{appendix:inequalities_norm_matrix}

\begin{lemma}
    For any matrix $A \in \realNumbers^{m \times n}$ and $1 \leq k < \infty$, it holds that $\sum_{i,j} ( 1 + A_{ij}^2 )^k \leq nm ( 1 + \pnorm{A}{\mathrm{F}} )^{2k}$.
For any two matrices $A \in \realNumbers^{m \times n}$, $B \in \realNumbers^{n \times p}$ and $0 \leq k < \infty$, it holds that $( 1 + \pnorm{AB}{\mathrm{F}} )^k \leq ( 1 + \pnorm{A}{\mathrm{F}} )^k ( 1 + \pnorm{B}{\mathrm{F}} )^k $.
For any two matrices $A, B \in \R^{n \times m}$, it holds that $\norm{A \odot B}_{\mathrm{F}} \leq \norm{A}_{\mathrm{F}}\norm{B}_{\mathrm{F}}$.
    \label{lemma:inequalities_norm_matrices}
\end{lemma}

\begin{proof}
    Recall Minkowski's inequality for sequences; that is
    $
        \bigl( \sum_i | x_i + y_i |^k \bigr)^{1/k}
        \leq \bigl( \sum_i | x_i |^k \bigr)^{1/k} \allowbreak + \bigl( \sum_i | y_i |^k \bigr)^{1/k}
    $, which holds for $1 \leq k < \infty$.
    It (i) implies that for any matrix $A \in \realNumbers^{n \times m}$ and $1 \leq k < \infty$, that
    \begin{equation}
        \sum_{i,j} ( 1 + A_{ij}^2 )^k
        \eqcom{i}\leq \Bigl( ( n m )^{1/k} + \bigl( \sum_{i,j} | A_{i,j}^2 |^{k} \bigr)^{1/k} \Bigr)^k
        \eqcom{ii}\leq nm \Bigl( 1 + \bigl( \sum_{i,j} | A_{i,j}^2 |^{k} \bigr)^{1/k} \Bigr)^k
    \end{equation}
    where (ii) we have used that the function $z^k$ is nondecreasing in $z \geq 0$ whenever $k \geq 0$.
    Because (iii) for the $\ell_k$-norm for sequences it holds that $\pnorm{x}{2k}^2 \leq \pnorm{x}{2}^2$ whenever $1 \leq k < \infty$, we obtain
    \begin{equation}
        \sum_{i,j} ( 1 + A_{ij}^2 )^k
        \eqcom{iii}\leq nm ( 1 + \pnorm{A}{\mathrm{F}}^2 )^{k}
        \eqcom{iv}\leq nm ( 1 + \pnorm{A}{\mathrm{F}} )^{2k}
        \label{eqn:Bound_on_the_quadratic_sum_of_matrix_elements_plus_one_by_Frobenius_norm}
    \end{equation}
    where (iv) we have used that the function $(1+z^2)^k \leq (1+z)^{2k}$ for all $z \geq 0$ whenever $k \geq 0$.
    This proves the first inequality.

    The second inequality is an immediate consequence of the submultiplicativity property of the Frobenius norm and its positivity, i.e.,
    \begin{equation}
        1 + \pnorm{AB}{\mathrm{F}}
        \leq 1 + \pnorm{A}{\mathrm{F}} \pnorm{B}{\mathrm{F}}
        \leq 1 + \pnorm{A}{\mathrm{F}} + \pnorm{B}{\mathrm{F}} + \pnorm{A}{\mathrm{F}} \pnorm{B}{\mathrm{F}}.
        \label{eqn:inequality_submulti}
    \end{equation}
    Raising to the $k$-th power left and right finishes its proof.

    The third inequality follows from strict positivity of the summands:
    \begin{equation}
        \pnorm{ A \odot B }{\mathrm{F}}^2
        = \sum_{i,j} A_{ij}^2 B_{ij}^2
        \leq \Bigl( \sum_{i,j} A_{ij}^2 \Bigr) \Bigl( \sum_{i,j} B_{ij}^2 \Bigr)
        = \pnorm{ A }{\mathrm{F}}^2 \pnorm{ B }{\mathrm{F}}^2.
        \label{eqn:Bound_on_Frobenius_norm_of_A_odot_B}
    \end{equation}
    Each of the inequalities has now been shown.
\end{proof}
 \end{document}